\pdfoutput=1
\documentclass[twoside]{article}
\usepackage[letterpaper, left=3cm, right=3cm, top=4cm, bottom=2.5cm]{geometry}
\usepackage{graphicx}
\usepackage{amsmath,amssymb,amsthm}
\usepackage{authblk}
\usepackage[utf8]{inputenc}
\usepackage{microtype}
\usepackage{mathrsfs} 
\usepackage{enumitem}
\usepackage{calc}
\usepackage{esint}
\usepackage{fancyhdr}
\usepackage{xcolor}
\usepackage[labelfont = bf]{caption}
\usepackage{doi}
\usepackage{subfigure}
\usepackage[numbers]{natbib}
\usepackage{stmaryrd}
\usepackage{mathtools}
\usepackage{booktabs}
\usepackage{colortbl}
\usepackage{tabulary}
\usepackage{diagbox}
\usepackage{caption}
\usepackage{float}

\usepackage{paper_diening}

\numberwithin{equation}{section}
\newtheorem{theorem}{Theorem}[section]
\newtheorem{proposition}[equation]{Proposition}

\newtheorem{corollary}[equation]{Corollary}
\newtheorem{remark}[equation]{Remark}
\newtheorem{lemma}[equation]{Lemma}

\newtheorem{assumption}[equation]{Assumption}

\providecommand{\Vo}{{\mathaccent23 V}}
\providecommand{\Qo}{{\mathaccent23 Q}}
\providecommand{\divo}{{\mathrm{div}\,}}

\DeclareMathAlphabet\mathbfcal{OMS}{cmsy}{b}{n}

\usepackage{titlesec}
\titleformat{\section}{\normalfont\scshape\centering}{\thesection.}{0.5em}{}
\titleformat*{\subsection}{\itshape}
\titleformat*{\subsubsection}{\itshape}

\providecommand{\keywords}[1]
{
	{\small\textit{Keywords:~~} #1}
}

\providecommand{\MSC}[1]
{
	{\small\textit{AMS MSC (2020): ~~} #1}
}

\usepackage{hyperref}
\hypersetup{
	colorlinks=true,
	linkcolor=blue,
	citecolor = blue,
	filecolor=magenta, 
	urlcolor=cyan,
}

\begin{document}
	\setlength{\abovedisplayskip}{5.5pt}
	\setlength{\belowdisplayskip}{5.5pt}
	\setlength{\abovedisplayshortskip}{5.5pt}
	\setlength{\belowdisplayshortskip}{5.5pt}

	\title{Error analysis for a finite element approximation of\\ the steady $p(\cdot)$-Navier--Stokes equations}
	\author[1]{Luigi C. Berselli\thanks{Email: \texttt{luigi.carlo.berselli@unipi.it}}\thanks{funded by  INdAM GNAMPA and Ministero dell'istruzione, dell'università e della ricerca (MIUR, Italian Ministry of Education, University and Research)
			within PRIN20204NT8W4: Nonlinear evolution PDEs, fluid dynamics and transport equations: theoretical foundations and applications.}}
	\author[2]{Alex Kaltenbach\thanks{Email: \texttt{\textcolor{black}{kaltenbach@math.tu-berlin.de}}}\thanks{funded by the Deutsche Forschungsgemeinschaft (DFG, German Research
			Foundation) - 525389262.\vspace*{-5.5mm}}}
	\date{\today}
	\affil[1]{\small{Department of Applied Mathematics, University of Pisa, Largo Bruno Pontecorvo 5, 56127~Pisa,~\textcolor{black}{ITALY}}}
	\affil[2]{\small{Institute of Mathematics, Technical University of Berlin, Straße des 17. Juni 136, 10623~Berlin,~\textcolor{black}{GERMANY}}}
	\maketitle

	\pagestyle{fancy}
	\fancyhf{}
	\fancyheadoffset{0cm}
	\addtolength{\headheight}{-0.25cm}
	\renewcommand{\headrulewidth}{0pt} 
	\renewcommand{\footrulewidth}{0pt}
	\fancyhead[CO]{\textsc{Error analysis for a FE approximation of the $p(\cdot)$-Navier--Stokes equations}}
	\fancyhead[CE]{\textsc{L. C. Berselli and A. Kaltenbach}}
	\fancyhead[R]{\thepage}
	\fancyfoot[R]{}
	
	\begin{abstract}
		\hspace*{-2.5mm}In this paper, we examine a finite element approximation of the steady $p(\cdot)$-Navier--Stokes~\mbox{equations} ($p(\cdot)$ is variable dependent) and prove
		orders of convergence by assuming natural fractional regularity assumptions on the velocity vector field and the kinematic pressure.
		Compared to previous~results, we treat the convective term and employ a more practicable  discretization of the power-law~index~$p(\cdot)$.
		Numerical experiments confirm the quasi-optimality of the  \textit{a priori} error estimates~(for~the~\mbox{velocity}) with respect to fractional regularity assumptions  on the~velocity~vector~field~and~the~kinematic~pressure.
	\end{abstract}
	
	
	\MSC{35J60; 35Q35;  65N15; 65N30; 76A05.}
	
	\section{Introduction}
	\thispagestyle{empty}
	
	\hspace{5mm}In this paper, we 
	examine a finite element approximation of the steady \textit{$p(\cdot)$-Navier--Stokes equations}\enlargethispage{5mm}
	\begin{equation}
		\label{eq:p-navier-stokes}
		\begin{aligned}
			-\divo\bfS (\cdot,\bfD\bfv)+[\nabla\bfv]\bfv+\nabla q&=\bff &&\quad\text{ in }\Omega\,,\\
			\divo\bfv&=0 &&\quad\text{ in }\Omega\,,
			\\
			\bfv &= \mathbf{0} &&\quad\text{ on } \partial\Omega\,.
		\end{aligned}
	\end{equation}
	More precisely,
	for a given vector field $\bff\colon\Omega\to \setR^d$, describing external forces, and a homogeneous Dirichlet
	boundary condition \eqref{eq:p-navier-stokes}$_3$, we seek a
	\textit{velocity vector field}~${\bfv\coloneqq (v_1,\ldots,v_d)^\top\colon \Omega\to
		\setR^d}$ and a scalar \textit{kinematic pressure} $q\colon \Omega\to \setR$ solving \eqref{eq:p-navier-stokes}.
	Here, $\Omega\subseteq \mathbb{R}^d$, $d\in \{2,3\}$, is a bounded \textcolor{black}{simplicial} Lipschitz domain. The \textit{extra-stress tensor} $\bfS (\cdot,\bfD\bfv)\colon\Omega\to \setR^{d\times d}_{\textup{sym}}$ depends on the \textit{\mbox{strain-rate}~tensor} $\bfD\bfv\coloneqq{ \frac{1}{2}(\nabla\bfv\hspace*{-0.15em}+\hspace*{-0.15em}\nabla\bfv^\top)}\colon\hspace*{-0.15em}\Omega\hspace*{-0.15em}\to\hspace*{-0.15em} \smash{\setR^{d\times d}_{\textup{sym}}}$, \textit{i.e.}, the symmetric part of the velocity gradient  ${\nabla\bfv\hspace*{-0.15em}=\hspace*{-0.15em}(\partial_j v_i)_{i,j=1,\ldots,d}
	\colon\hspace*{-0.15em}\Omega\hspace*{-0.15em}\to\hspace*{-0.15em} \setR^{d\times d}}$. The \textit{convective term} $\smash{[\nabla\bfv]\bfv\colon\Omega\to \mathbb{R}^d}$ is defined by $\smash{([\nabla\bfv]\bfv)_i\coloneqq \sum_{j=1}^d{v_j\partial_j v_i}}$ for all $i=1,\ldots,d$.
	
	Throughout the paper, the extra stress tensor $\bfS\colon \Omega\times \mathbb{R}^{d\times d}\to\mathbb{R}^{d\times d}_{\textup{sym}}  $ is supposed to assume the form
	\begin{align}\label{eq:example-stress}
		\bfS (x,\bfA)\coloneqq \mu_0\, (\delta+\vert\bfA^{\textup{sym}}\vert)^{p(x)-2}\bfA^{\textup{sym}}\,,
	\end{align}
	where $\mu_0>0$, $\delta\ge 0$, and the \textit{power-law index} $p\colon \Omega\to (1,+\infty)$ is at least (Lebesgue) measurable with
	\begin{align*}
		1<p^-\coloneqq \underset{x\in \Omega}{\textrm{ess\,inf}}{\,p(x)}\leq p^+\coloneqq \underset{x\in \Omega}{\textrm{ess\,sup}}{\,p(x)}<\infty\,.
	\end{align*}
	
	The steady $p(\cdot)$-Navier--Stokes equations \eqref{eq:p-navier-stokes} are a prototypical example of a non-linear system with variable growth conditions. They appear naturally in physical models for so-called~\textit{smart fluids}, \textit{e.g.}, electro-rheological fluids (\textit{cf}.\ \cite{RR1,rubo}), micro-polar electro-rheological~fluids (\textit{cf}.\ \cite{win-r,eringen-book}), \textcolor{black}{magneto-rheological fluids (\textit{cf}.\ \cite{magneto})},  chemically reacting fluids (\textit{cf}.\ \cite{LKM78,HMPR10}), and thermo-rheological fluids (\textit{cf}.\ \cite{Z97,AR06}). In all these models, the power-law index $p(\cdot)$ is a function depending on certain physical quantities, \textit{e.g.}, an electric field,~a~concentration field of a chemical material or a temperature field, and, thus, implicitly depends on $x\in \Omega$. In addition, the governing non-linearity \eqref{eq:example-stress} in the steady $p(\cdot)$-Navier--Stokes equations \eqref{eq:p-navier-stokes}  has applications in the field of image reconstruction (\textit{cf}.\ \mbox{\cite{AMS08,CLR06,LLP10}}).
	
	\subsection{Related contributions}\vspace*{-0.5mm}
	
	\hspace{5mm}We summarize the known analytical and numerical results.\vspace*{-2mm}\enlargethispage{2.5mm}
	
	\subsubsection{Existence results}\vspace*{-0.5mm}

	\hspace{5mm}The existence analysis of the steady $p(\cdot)$-Navier--Stokes equations \eqref{eq:p-navier-stokes} is by now well-understood~and developed essentially in the following main steps:
	
	\begin{itemize}[noitemsep,topsep=2pt,leftmargin=!,labelwidth=\widthof{(iii)}]
		\item[(i)]  In \hspace*{-0.1mm}\cite{rubo}, \hspace*{-0.1mm}using \hspace*{-0.1mm}the \hspace*{-0.1mm}theory \hspace*{-0.1mm}of \hspace*{-0.1mm}pseudo-monotone \hspace*{-0.1mm}operators, \hspace*{-0.1mm}M.\ \hspace*{-0.1mm}R\r{u}\v{z}i\v{c}ka \hspace*{-0.1mm}proved \hspace*{-0.1mm}the \hspace*{-0.1mm}weak \hspace*{-0.1mm}solvability~\hspace*{-0.1mm}of~\hspace*{-0.1mm}the steady $p(\cdot)$-Navier--Stokes equations \eqref{eq:p-navier-stokes} for a bounded power-law~index~${p\hspace{-0.15em}\in\hspace{-0.15em} \mathcal{P}^{\infty}(\Omega)}$~with~${p^-\hspace{-0.1em}>\hspace{-0.1em} \frac{3d}{d+2}}$;
		\item[(ii)] In \cite{huber-paper}, using the $L^\infty$-truncation technique, A.\ Huber proved the weak solvability of  the steady $p(\cdot)$-Navier--Stokes \hspace*{-0.1mm}equations \hspace*{-0.1mm}\eqref{eq:p-navier-stokes} \hspace*{-0.1mm}for \hspace*{-0.1mm}a \hspace*{-0.1mm}$\log$-Hölder \hspace*{-0.1mm}continuous power-law~\hspace*{-0.1mm}index~\hspace*{-0.1mm}${p\hspace{-0.1em}\in\hspace{-0.1em} \mathcal{P}^{\log}(\Omega)}$~\hspace*{-0.1mm}with~\hspace*{-0.1mm}${p^-\hspace{-0.15em}>\hspace{-0.15em}\frac{2d}{d+1}}$;
		\item[(iii)]  In \cite{dms}, using the $W^{1,\infty}$-truncation technique, L.\ Diening et al.\ proved the  weak solvability of the steady $p(\cdot)$-Navier--Stokes equations \eqref{eq:p-navier-stokes} for a $\log$-Hölder continuous power-law index $p\in \mathcal{P}^{\log}(\Omega)$ with $p^->\smash{\frac{2d}{d+2}}$. This result is almost optimal as for $p^-<  \smash{\frac{2d}{d+2}}$, one cannot ensure~that~${\mathbf{v}\in \smash{L^2(\Omega;\mathbb{R}^d)}}$, \textit{i.e.},  $\mathbf{v}\otimes \mathbf{v}\in L^1(\Omega;\mathbb{R}^{d\times d}_{\textrm{sym}})$, so that the convective term is not well-defined in~a~\mbox{distributional}~sense. The same limitation also holds in the case of a constant~\mbox{power-law}~index.
	\end{itemize}\vspace*{-1mm}
	
	\subsubsection{Numerical analyses}\vspace*{-0.5mm}
	
	\hspace*{5mm}The numerical analysis of fluids with shear-dependent viscosities (\textit{i.e.}, $p=\textrm{const}$) and, thus, of the steady $p(\cdot)$-Navier--Stokes equations \eqref{eq:p-navier-stokes} started with the seminal contribution of D. Sandri (\textit{cf}.~\cite{San1993}), with \hspace*{-0.1mm}improvements \hspace*{-0.1mm}by J.~\hspace*{-0.1mm}W. \hspace*{-0.1mm}Barret \hspace*{-0.1mm}and \hspace*{-0.1mm}W.~\hspace*{-0.1mm}B. \hspace*{-0.1mm}Liu (\textit{cf}.\ \hspace*{-0.1mm}\cite{baliu}) \hspace*{-0.1mm}deriving \hspace*{-0.1mm}the \hspace*{-0.1mm}error \hspace*{-0.1mm}estimates~\hspace*{-0.1mm}in~\hspace*{-0.1mm}\mbox{\textit{quasi-norms}}.~\hspace*{-0.1mm}For a finite element approximation of the steady $p$-Stokes equations (\textit{i.e.}, $p=\textrm{const}$), \textit{i.e.}, the steady $p$-Navier--Stokes \hspace*{-0.1mm}equations \hspace*{-0.1mm}in \hspace*{-0.1mm}the \hspace*{-0.1mm}case \hspace*{-0.1mm}of \hspace*{-0.1mm}a \hspace*{-0.1mm}slow \hspace*{-0.1mm}(laminar) \hspace*{-0.1mm}flow, \hspace*{-0.1mm}which \hspace*{-0.1mm}motivates~\hspace*{-0.1mm}to~\hspace*{-0.1mm}neglect~\hspace*{-0.1mm}the~\hspace*{-0.1mm}convective~\hspace*{-0.1mm}term~\hspace*{-0.1mm}$[\nabla \bfv]\bfv$ in \eqref{eq:p-navier-stokes}$_1$, L.~Belenki~et~al. (\textit{cf}.~\cite{bdr-phi-stokes}) derived \textit{a priori} error estimates in the~so-called~\textit{natural~distance}; see~also A. Hirn (\textit{cf}.\ \cite{Hi13a}). The numerical analysis of problems 
	with variable exponents non-linearities, however, is less developed and we are  only aware of the following contributions treating the steady case:
	
	\begin{itemize}[noitemsep,topsep=2pt,leftmargin=!,labelwidth=\widthof{(iii)}]
		
		\item[(i)] For \hspace*{-0.15mm}a \hspace*{-0.15mm}finite \hspace*{-0.15mm}element \hspace*{-0.15mm}approximation \hspace*{-0.15mm}of \hspace*{-0.15mm}the \hspace*{-0.15mm}$p(\cdot)$-Laplace \hspace*{-0.15mm}equation \hspace*{-0.15mm}with \hspace*{-0.15mm}Dirichlet \hspace*{-0.15mm}boundary \hspace*{-0.15mm}\mbox{condition},~\hspace*{-0.15mm}\textit{i.e.}, the \hspace*{-0.15mm}steady \hspace*{-0.15mm}$p(\cdot)$-Navier--Stokes \hspace*{-0.15mm}equations \hspace*{-0.15mm}with \hspace*{-0.1mm}the \hspace*{-0.15mm}symmetric \hspace*{-0.15mm}gradient \hspace*{-0.25mm}$\bfD$ \hspace*{-0.25mm}replaced \hspace*{-0.15mm}by \hspace*{-0.15mm}the\hspace*{-0.15mm} full~\hspace*{-0.15mm}\mbox{gradient}~\hspace*{-0.35mm}$\nabla$, without convective term $[\nabla\bfv]\bfv$~in~\eqref{eq:p-navier-stokes}$_1$, and without incompressibility constraint \eqref{eq:p-navier-stokes}$_2$, D.\ Breit et al.\ (\textit{cf}.~\cite{BDS15}) derived \textit{a priori} error estimates in the natural distance.
		In addition, L.~M. Del Pezzo~et~al. (\textit{cf}.~\cite{DPLM12}) studied the (weak) convergence of an Interior Penalty Discontinuous Galerkin (IPDG) method and  M. Caliari and S. Zuccher (\textit{cf}.~\cite{CZ17}) implemented a quasi-Newton minimization for the corresponding energy formulation; 
		
		\item[(ii)] For a finite element approximation of the steady $p(\cdot)$-Stokes equations, \textit{i.e.}, in the case of a slow (laminar) flow (which motivates to neglect the  the convective term $[\nabla \bfv]\bfv$ in \eqref{eq:p-navier-stokes}$_1$), %L.\ C.\ Berselli et al. (\textit{cf}.~
		in \cite{BBD15} \textit{a~priori} error estimates in the natural distance were derived.~To~be~more~precise, given a solution pair $(\bfv,q)^\top\in \smash{W^{1,p(\cdot)}_0(\Omega;\mathbb{R}^d)\times L_0^{p'(\cdot)}(\Omega)}$ of \eqref{eq:p-navier-stokes} and a discrete solution pair $(\bfv_h,q_h)^\top\in \textcolor{black}{\Vo_{h,0}}\times \Qo_h$, where  $ \textcolor{black}{\Vo_{h,0}}$ and $\Qo_h$ are suitable finite element spaces, 
		they derived~the~a~priori~error~estimates
		\begin{align}\label{eq:old_apriori}
			\begin{aligned}
				\|\bfF_h(\cdot,\bfD\bfv)-\bfF_h(\cdot,\bfD\bfv_h)\|_{2,\Omega}^2&\lesssim h^{\min\{2,(p^+)'\}}+h^{2\alpha}\,,\\
				\|q-q_h\|_{p_h'(\cdot)}^2&\lesssim  h^{\smash{\frac{\min\{((p^+)')^2,4\}}{(p^-)'}}}+h^{2\alpha}\,,
			\end{aligned}
		\end{align}
		where %\bfF_h\colon \Omega\times\mathbb{R}^{d\times d}\to \mathbb{R}^{d\times d}_{\textup{sym}}$ is defined by 
		$\bfF_h(x,\bfA)\hspace*{-0.05em}\coloneqq \hspace*{-0.05em} (\delta+\vert \bfA^{\textup{sym}}\vert)^{\smash{\frac{p_h(x)-2}{2}}}\bfA^{\textup{sym}}$ for a.e.\ $x\in \Omega$ and all $\bfA\hspace*{-0.05em}\in\hspace*{-0.05em} \mathbb{R}^{d\times d}$~and~${p_h\hspace*{-0.05em}\in\hspace*{-0.05em} \mathbb{P}^0(\mathcal{T}_h)}$~is~an \hspace*{-0.1mm}element-wise \hspace*{-0.1mm}constant \hspace*{-0.1mm}approximation \hspace*{-0.1mm}of \hspace*{-0.1mm}$p\hspace*{-0.15em}\in \hspace*{-0.15em}C^{0,\alpha}(\overline{\Omega})$, $\alpha\hspace*{-0.15em}\in\hspace*{-0.15em} (0,1]$, \hspace*{-0.1mm}under \hspace*{-0.1mm}the \hspace*{-0.1mm}regularity~\hspace*{-0.1mm}assumptions~\hspace*{-0.1mm}that  $\bfF(\cdot,\bfD\bfv)\hspace*{-0.1em}\in \hspace*{-0.1em} W^{1,2}(\Omega;\mathbb{R}^{d\times d})$, where %${\bfF\colon \Omega\times \mathbb{R}^{d\times d}\to \mathbb{R}^{d\times d}_{\textup{sym}}}$ is defined by
		$\bfF(x,\bfA)\hspace*{-0.1em}\coloneqq\hspace*{-0.1em} (\delta+\vert \bfA^{\textup{sym}}\vert)^{\frac{p(x)-2}{2}}\bfA^{\textup{sym}}$ for all $x\hspace*{-0.1em}\in\hspace*{-0.1em} \Omega$ and $\bfA\hspace*{-0.1em}\in\hspace*{-0.1em} \mathbb{R}^{d\times d}$, 
		and that $q\in W^{1,p'(\cdot)}(\Omega)$; 
		
		\item[(iii)] For the model describing the steady motion of a chemically reacting fluid, S.\ Ko et al.\ (\textit{cf}.\ \cite{KPS18,KS19}) proved the weak convergence of conforming, discretely inf-sup-stable finite element~\mbox{approximations}. In \cite{GHS23}, P.~A.\ Gazca--Orozco et al.  proposed an iterative scheme for the approximation of the model describing the steady motion of a chemically reacting~fluid  deploying conforming, discretely inf-sup-stable finite element approximations.

	\end{itemize}
	
	\subsection{New contributions}\enlargethispage{2mm}
	
	\hspace*{5mm}The purpose of this paper is to extend the results in the paper \cite{BBD15} with respect to several aspects~and, in~this~way, to establish a firm foundation for further numerical analyses of related complex models involving the steady $p(\cdot)$-Navier--Stokes equations \eqref{eq:p-navier-stokes} such as, \textit{e.g.}, electro-rheological~fluids~(\textit{cf}.~\mbox{\cite{RR1,rubo}}), micro-polar electro-rheological~fluids (\textit{cf}.\ \cite{win-r,eringen-book}), \textcolor{black}{magneto-rheological fluids (\textit{cf}.\ \cite{magneto})},  chemically reacting fluids (\textit{cf}.~\cite{LKM78,HMPR10}), and thermo-rheological~fluids (\textit{cf}.~\mbox{\cite{Z97,AR06}}):
	
	\begin{itemize}[noitemsep,topsep=2pt,leftmargin=!,labelwidth=\widthof{3.}]
		\item[1.] \textit{Quadrature \hspace*{-0.1mm}rule:} \hspace*{-0.1mm}In \hspace*{-0.1mm}this \hspace*{-0.1mm}paper, \hspace*{-0.1mm}to \hspace*{-0.1mm}define \hspace*{-0.1mm}an \hspace*{-0.1mm}element-wise \hspace*{-0.1mm}constant \hspace*{-0.1mm}approximation \hspace*{-0.1mm}of \hspace*{-0.1mm}the \hspace*{-0.1mm}\mbox{power-law} index $p(\cdot)$, 
		we consider a general one-point quadrature rule. We recall that in earlier contributions (\textit{cf}.\  \cite{BDS15,BBD15}), a restrictive quadrature rule was employed that
		afforded to find the minimum of the power-law index $p(\cdot)$ on each element, \textit{i.e.}, $p_h|_T\coloneqq \textrm{argmin}_{x\in T}{p(x)}$ for all~${T\in \mathcal{T}_h}$;
		this is a potentially slow~and~costly~task  --if the power-law index $p(\cdot)$ is non-constant.
		
		\item[2.] \textit{Fractional \hspace*{-0.1mm}regularity/error \hspace*{-0.1mm}decay \hspace*{-0.1mm}rates:} \hspace*{-0.1mm}In \hspace*{-0.1mm}the \hspace*{-0.1mm}case \hspace*{-0.1mm}of \hspace*{-0.1mm}a \hspace*{-0.1mm}non-Lipschitz~\hspace*{-0.1mm}but~\hspace*{-0.1mm}only~\hspace*{-0.1mm}Hölder~\hspace*{-0.1mm}\mbox{continuous}~\hspace*{-0.1mm}power-law index $p\in C^{0,\alpha}(\overline{\Omega})$, $\alpha\in (0,1]$,  one cannot hope for the \textit{``full''} regularity $\bfF(\cdot,\bfD\bfv)\in W^{1,2}(\Omega;\mathbb{R}^{d\times d})$ and $q\in W^{1,p'(\cdot)}(\Omega)$, but instead it is reasonable 
		to expect (\textit{cf}.\  \cite[Rem.\ 4.5]{BDS15})~the~\textit{``partial''}~\mbox{regularity} $\bfF(\cdot,\bfD\bfv)\in  N^{\beta,2}(\Omega;\mathbb{R}^{d\times d})$, $\beta\in (0,1]$, where $N^{\beta,2}(\Omega;\mathbb{R}^{d\times d})$ denotes the Nikolski\u{\i} space.~Concerning the \hspace*{-0.1mm}pressure, \hspace*{-0.1mm}we \hspace*{-0.1mm}propose \hspace*{-0.1mm}to \hspace*{-0.1mm}consider \hspace*{-0.1mm}the \hspace*{-0.1mm}regularity  \hspace*{-0.1mm}$q \hspace*{-0.1mm}\in \hspace*{-0.1mm} H^{\gamma,p'(\cdot)}(\Omega)$, $\gamma \hspace*{-0.1mm}\in  \hspace*{-0.1mm}(0,1]$, \hspace*{-0.1mm}where \hspace*{-0.1mm}$H^{\gamma,p'(\cdot)}(\Omega)$~\hspace*{-0.1mm}\mbox{denotes} the fractional variable Haj\l asz--Sobolev space (\textit{cf}.\ \cite{Yang03,ZS22}),~which~is~more~\mbox{appropriate}~to~our~\mbox{problem}.
		Since the  space $H^{\gamma,p'(\cdot)}(\Omega)$ is slightly different from the classical Sobolev--Slobodeckij space~(\textit{cf}.~\cite{Hitchhiker12}), several new fractional interpolation error estimates are derived.

		\item[3.] \textit{Convective term:} 
		One further extension compared to \cite{BBD15} is the treatment of the convective term, which is more challenging in the case of a non-constant power-law index $p\in C^{0,\alpha}(\overline{\Omega})$, $\alpha\in (0,1]$, and given only fractional regularity assumptions on the solutions. To this end, we need to impose the restriction $p^-\ge\frac{3d}{d+2}$ on the power-law index $p\in C^{0,\alpha}(\overline{\Omega})$, $\alpha\in (0,1]$, which was not needed in \cite{BBD15} and solely comes from the treatment of the convective term.
		
			\item[4.] \textit{Quasi-optimality:} We establish the optimality of the derived error decay rates (for the velocity) with respect to fractional regularity assumptions on the velocity and the pressure via numerical experiments. In addition, imposing an alternative \textit{``natural''} fractional regularity assumption~on~the~pressure, \textit{i.e.}, $(\delta+\vert \bfD\bfv\vert)^{\smash{\frac{2-p(\cdot)}{2}}}\vert \nabla^\gamma p\vert\in L^2(\Omega)$,  $\gamma\in (0,1]$, we derive an \textit{a priori} error estimate with an error decay rate that does not depend critically on the maximal (\textit{i.e.}, $p^+$) and minimal~(\textit{i.e.},~$p^-$) value of the power-law index $p\in C^{0,\alpha}(\overline{\Omega})$, $\alpha\in (0,1]$, but is constant~and~also~\mbox{quasi-optimal}~(for~the~\mbox{velocity}).
		
	\end{itemize}
	
	\textit{This \hspace*{-0.1mm}paper \hspace*{-0.1mm}is \hspace*{-0.1mm}organized \hspace*{-0.1mm}as \hspace*{-0.1mm}follows:} \hspace*{-0.1mm}In \hspace*{-0.1mm}Section \hspace*{-0.1mm}\ref{sec:preliminaries}, \hspace*{-0.1mm}we \hspace*{-0.1mm}introduce \hspace*{-0.1mm}the \hspace*{-0.1mm}relevant \hspace*{-0.1mm}notation~\hspace*{-0.1mm}and~\hspace*{-0.1mm}function~\hspace*{-0.1mm}spaces, and recall definitions and results addressing  (generalized) $N$-functions  and the \hspace*{-0.1mm}extra-stress~\hspace*{-0.1mm}tensor~\hspace*{-0.1mm}\eqref{eq:example-stress}. 
	In Section \ref{sec:p-navier-stokes}, we introduce two equivalent weak formulations of the steady $p$-Navier--Stokes equations \eqref{eq:p-navier-stokes} and examine natural (fractional) regularity assumptions for weak solutions of these formulations. In Section \ref{sec:discrete_p-navier-stokes}, we introduce two equivalent discrete weak formulations of the steady $p$-Navier--Stokes equations  \eqref{eq:p-navier-stokes} and examine two projectors for their stability properties. 
	In Section \ref{sec:fractional_interpolation_estimates},~we~derive~several fractional interpolation error estimates for these projectors.
	In Section \ref{sec:a_priori}, we derive \textit{a priori} error estimates for the approximation of the equivalent weak formulations of the steady $p$-Navier--Stokes equations  \eqref{eq:p-navier-stokes} via
	the equivalent discrete weak formulations of the steady $p$-Navier--Stokes~equations~\eqref{eq:p-navier-stokes}.~In~Section~\ref{sec:experiments}, via numerical experiments, we examine the \textit{a priori} error estimates derived in Section \ref{sec:a_priori} for \mbox{quasi-optimality}.

	\newpage
	\section{Preliminaries}\label{sec:preliminaries}\enlargethispage{8.5mm}
	
	\hspace*{5mm}Throughout the paper, let $\Omega\hspace*{-0.05em}\subseteq\hspace*{-0.05em} \mathbb{R}^d$,~$d\hspace*{-0.05em}\in\hspace*{-0.05em}\{2,3\}$,  be a bounded \textcolor{black}{simplicial} Lipschitz domain.~The~\mbox{integral} mean of a (Lebesgue) integrable function $f\colon \hspace*{-0.15em} M\hspace*{-0.15em} \to\hspace*{-0.15em} \mathbb{R}$ over a (Lebesgue)~measurable~set~${M\hspace*{-0.15em} \subseteq\hspace*{-0.15em}  \mathbb{R}^d}$~with~\textcolor{black}{${\vert M\vert\hspace*{-0.15em} >\hspace*{-0.15em}0}$ is denoted by $\langle f\rangle_M\coloneqq {\frac 1 {|M|}\int_M f \,\textup{d}x}$}. For (Lebesgue) measurable functions~${f,g\colon M\to \mathbb{R}}$~and~a~(Lebesgue) measurable set $M\subseteq \mathbb{R}^d$,
	we will write $(f,g)_M\coloneqq \int_M f g\,\textup{d}x$, 
	whenever~the~right-hand side is well-defined. %We employ the notation~$\wedge$ and $\vee$, for the minimum and maximum of the numbers,~respectively.\vspace*{-1mm}
	\subsection{Variable Lebesgue spaces, variable Sobolev spaces, Nikolski\u{\i} spaces, and fractional variable Haj\l asz--Sobolev spaces}
	
	\hspace*{5mm}Let $M\subseteq \mathbb{R}^d$, $d\in \mathbb{N}$, be a (Lebesgue) measurable set
	and $p\colon M\to [1,+\infty]$ be a (Lebesgue) measurable function, a so-called  \textit{variable
		exponent}. By $\mathcal{P}(M)$, we denote the \textit{set of variable exponents}. Then, for $p\in \mathcal{P}(M)$, we denote by
	${p^+\coloneqq \textup{ess\,sup}_{x\in
			M}{p(x)}}$~and~${p^-\coloneqq \textup{ess\,inf}_{x\in
			M}{p(x)}}$ its constant~\textit{limit~exponents}. Moreover,~by
	$\mathcal{P}^{\infty}(M)\coloneqq \{p\in\mathcal{P}(M)\mid
	p^+<\infty\}$, we denote the \textit{set of bounded variable exponents}. For ${p\in\mathcal{P}^\infty(M)}$ and a (Lebesgue) measurable function $v\in L^0(M;\mathbb{R}^{\textcolor{black}{\ell}})$, $\textcolor{black}{\ell}\in \mathbb{N}$, we define the \textit{modular (with respect to $p$)} by\vspace*{-1mm}
	\begin{align*}
		\rho_{p(\cdot),M}(v)\coloneqq \int_{M}{\vert v\vert^{p(\cdot)}\,\mathrm{d}x}\,.
	\end{align*}
	
	Then, for $p\in \mathcal{P}(M)$ and $\textcolor{black}{\ell}\in \mathbb{N}$,
	the \textit{variable Lebesgue space} is defined by
	\begin{align*}
	\smash{L^{p(\cdot)}(M;\mathbb{R}^{\textcolor{black}{\ell}})\coloneqq \big\{ v\in L^0(M;\mathbb{R}^{\textcolor{black}{\ell}})\mid \rho_{p(\cdot),M}(v)<\infty\big\}\,.}
	\end{align*}
	The \textit{Luxembourg norm} $\| v\|_{p(\cdot),M}\coloneqq \inf\{\lambda> 0\mid \rho_{p(\cdot),M}(\frac{v}{\lambda})\leq 1\}$ turns $L^{p(\cdot)}(M;\mathbb{R}^{\textcolor{black}{\ell}})$~into~a~\mbox{Banach}~space (\textit{cf}.\ \cite[Thm.\  3.2.7]{dhhr}).
	
	Moreover, for an open set $G\subseteq \mathbb{R}^d$, $d\in \mathbb{N}$, $p\in\mathcal{P}^\infty(G)$ and $\textcolor{black}{\ell}\in \mathbb{N}$, the \textit{variable Sobolev space} is defined by 
	\begin{align*}
		\smash{W^{1,p(\cdot)}(G;\mathbb{R}^{\textcolor{black}{\ell}})\coloneqq \big\{ v\in L^{p(\cdot)}(G;\mathbb{R}^{\textcolor{black}{\ell}})\mid \nabla v\in L^{p(\cdot)}(G;\mathbb{R}^{\textcolor{black}{\ell}\times d})\big\}\,.}
	\end{align*} 
	The \textit{variable Sobolev norm} $\| v\|_{1,p(\cdot),G}\coloneqq \| v\|_{p(\cdot),G}+\|\nabla v\|_{p(\cdot),G}$ turns $W^{1,p(\cdot)}(G;\mathbb{R}^{\textcolor{black}{\ell}})$ into a Banach space\\[-0.5mm] (\textit{cf}.\  \cite[Thm.\  8.1.6]{dhhr}). The closure of $C^\infty_c(\Omega;\mathbb{R}^{\textcolor{black}{\ell}})$ in $W^{1,p(\cdot)}(G;\mathbb{R}^{\textcolor{black}{\ell}})$ is defined by $W^{1,p(\cdot)}_0(G;\mathbb{R}^{\textcolor{black}{\ell}})$.
	
	To express fractional regularity of the velocity, we need the notation of Nikolski\u{\i} spaces (\textit{cf}.\  \cite{Nikol75}). For an open set $G\subseteq \mathbb{R}^d$, $d\in \mathbb{N}$, $p \in  [1, \infty)$, $\beta\in (0,1]$, and $v\in L^p(G)$, the  \textit{Nikolski\u{\i} semi-norm}~is~defined~by 
	\begin{align}\label{eq:nikolski_semi-norm}
		[v]_{N^{\beta,p}(G)}\coloneqq \sup_{ \tau\in \mathbb{R}^d\setminus\{0\}}{
			\frac{1}{\vert \tau\vert^\beta}\bigg(\int_{G\cap (G- \tau)}{\vert v(x +  \tau)-v(x)\vert^p\,\mathrm{d}x}\bigg)^{\smash{1/p}}}<\infty\,.
	\end{align}
	Then, for $p \in  [1, \infty)$ and $\beta\in (0,1]$, the \textit{Nikolski\u{\i} space} is defined by
	\begin{align*}
		\smash{N^{\beta,p}(G)\coloneqq \big\{ v\in L^p(G)\mid [v]_{N^{\beta,p}(G)}<\infty\big\}\,.}
	\end{align*}
	The \textit{Nikolski\u{\i} norm} $\|\cdot\|_{N^{\beta,p}(G)}\coloneqq \| \cdot\|_{p,G}+[\cdot]_{N^{\beta,p}(G)}$ turns $N^{\beta,p}(G)$  into a Banach~space (\textit{cf}.~\mbox{\cite[Sec.~4.7]{Nikol75}}).
	
	To express fractional regularity of the pressure, we need the notation of variable fractional Haj\l asz--Sobolev spaces (\textit{cf}.\  \cite{Yang03,ZS22}). For  an open set $G\subseteq \mathbb{R}^d$, $d\in \mathbb{N}$, $p\in \mathcal{P}^\infty(G)$, and $\gamma\in (0,1]$, a function  $v\in L^{p(\cdot)}(G)$ has a \textit{($\gamma$-order) upper Haj\l asz gradient} if there exists $g\in L^{p(\cdot)}(G;\mathbb{R}_{\ge 0})$ such that 
	\begin{align}\label{eq:hajlasz_gradient}
		\vert v(x)-v(y)\vert \leq (g(x)+g(y))\,\vert x-y\vert^{\gamma}\quad\text{ for a.e. }x,y\in G\,.
	\end{align}
	For every $v\in L^{p(\cdot)}(G)$, the \textit{set of ($\gamma$-order) upper Haj\l asz gradients} is defined by ${\mathrm{Gr}(v)\coloneqq \{g\in L^{p(\cdot)}(G;}$ $\mathbb{R}_{\ge 0})\mid \eqref{eq:hajlasz_gradient}\text{ holds}\}$.
	Then, for every $v\in L^{p(\cdot)}(G)$ and $\gamma\in (0,1]$, the \textit{Haj\l asz--Sobolev space} is defined by
	\begin{align*}
		\smash{H^{\gamma,p(\cdot)}(G)\coloneqq \big\{v\in L^{p(\cdot)}(G)\mid  \mathrm{Gr}(v)\neq \emptyset\big\}\,.}
	\end{align*}
	The \textit{Haj\l asz--Sobolev norm} $\|v\|_{\gamma,p(\cdot),G}\coloneqq \|v\|_{p(\cdot),G}+\inf_{g\in \mathrm{Gr}(v)}{\|g\|_{p(\cdot),G}}$ turns $H^{\gamma,p(\cdot)}(G)$ into a Banach space (\textit{cf}.\  \cite[Prop.\  2.5]{ZS22}). If $p^->1$, then for any $v\in H^{\gamma,p(\cdot)}(G)$, by $\vert \nabla^{\gamma} v\vert \coloneqq\textup{arg\,min}_{g\in \mathrm{Gr}(v)}\|g\|_{p(\cdot),G}$, we denote  the \textit{minimal $\gamma$-order upper Haj\l asz gradient}, whose (unique) existence is a consequence~of~the direct method in the calculus of variations, due to the non-emptiness, convexity, and~closedness~of~$\mathrm{Gr}(v)\subseteq L^{p(\cdot)}(G)$ for all $v\in H^{\gamma,p(\cdot)}(G)$ and the continuity and strict convexity of $\|\cdot\|_{p(\cdot),G}$ (\textit{cf}.\  \cite[Thm.\  3.4.9]{dhhr}).

	\subsection{(Generalized) $N$-functions}
	
	\hspace*{5mm}A \hspace{-0.1em}(real) \hspace{-0.1em}convex function
	$\psi\colon\hspace{-0.1em}\mathbb{R}_{\geq 0} \to \mathbb{R}_{\geq 0}$ is called an
	\textit{$N$-function},~if~${\psi(0)=0}$,~${\psi(t)>0}$~for~all~${t>0}$,
	$\lim_{t\rightarrow0} \psi(t)/t=0$, and
	$\lim_{t\rightarrow\infty} \psi(t)/t=\infty$. If, in addition, $\psi\in C^1(\mathbb{R}_{\geq 0})\cap C^2(\mathbb{R}_{> 0})$~and~${\psi''(t)\hspace{-0.1em}>\hspace{-0.1em}0}$ for all $t>0$, we call $\psi$ a \textit{regular $N$-function}. For a regular $N$-function  ${\psi \colon \mathbb{R}_{\geq 0}\to \mathbb{R}_{\geq 0}}$,~we~have~that $\psi (0)=\psi'(0)=0$,
	$\psi'\colon\hspace{-0.1em}\mathbb{R}_{\geq 0} \to \mathbb{R}_{\geq 0}$ is increasing and $\lim _{t\to \infty} \psi'(t)=\infty$.~For~a~given~\mbox{$N$-function} ${\psi \colon\mathbb{R}_{\geq 0} \to \mathbb{R}_{\geq 0}}$, we define the corresponding \textit{(Fenchel) conjugate \mbox{$N$-function}} $\psi^*\colon\mathbb{R}_{\geq 0} \to \mathbb{R}_{\geq 0}$,~for~every~$t\ge 0$,~by
	${\psi^*(t)\coloneqq \sup_{s \geq 0} \{st
		-\psi(s)\}}$, which satisfies $(\psi^*)' =
	(\psi')^{-1}$ in $\mathbb{R}_{\ge 0}$. An $N$-function $\psi$ satisfies~the~\textit{$\Delta_2$-con-dition}
	(in short, $\psi \hspace*{-0.1em}\in\hspace*{-0.1em} \Delta_2$), if there exists $K\hspace*{-0.1em}>\hspace*{-0.1em} 2$ such that~for~every~${t \hspace*{-0.1em}\ge\hspace*{-0.1em} 0}$,~it~holds~that~${\psi(2\,t) \hspace*{-0.1em}\leq\hspace*{-0.1em} K\, \psi(t)}$.~Then, \hspace*{-0.15mm}we \hspace*{-0.15mm}denote \hspace*{-0.15mm}the
	\hspace*{-0.1mm}smallest \hspace*{-0.15mm}such \hspace*{-0.15mm}constant \hspace*{-0.15mm}by \hspace*{-0.15mm}$\Delta_2(\psi)\hspace*{-0.15em}>\hspace*{-0.15em}0$. \hspace*{-0.15mm}We \hspace*{-0.15mm}say \hspace*{-0.15mm}that \hspace*{-0.15mm}an \hspace*{-0.15mm}$N$-function~\hspace*{-0.15mm}${\psi\colon\hspace*{-0.1em}\mathbb{R}_{\ge 0}\hspace*{-0.15em}\to \hspace*{-0.15em}\mathbb{R}_{\ge 0}}$ satisfies the \textit{$\nabla_2$-condition} (in short, $\psi\in \nabla_2$), if its (Fenchel) conjugate $\psi^*\colon\mathbb{R}_{\ge 0}\to \mathbb{R}_{\ge 0}$ is an $N$-function satisfying the $\Delta_2$-condition. 
	If $\psi\colon\mathbb{R}_{\ge 0}\to \mathbb{R}_{\ge 0}$ satisfies the $\Delta_2$- and the $\nabla_2$-condition (in~short, $\psi\in \Delta_2\cap \nabla_2$), then, there holds 
	the following refined version~of~the~\mbox{\textit{$\varepsilon$-Young}}~\textit{inequality}: for every
	$\varepsilon> 0$, there exists a constant $c_\varepsilon>0 $, depending only on
	$\Delta_2(\psi),\Delta_2( \psi ^*)<\infty$, such that for every $ s,t\geq0 $, it holds that
	\begin{align}
		\label{ineq:young}
		s\,t\leq  c_\varepsilon \,\psi^*(s)+\varepsilon \, \psi(t)\,.
	\end{align}
	Let $M\subseteq \mathbb{R}^d$, $d\in \mathbb{N}$, be a (Lebesgue) measurable set. Then, a
	function $\psi \colon M \times \mathbb{R}_{\ge 0} \to \mathbb{R}_{\ge 0}$ is called a \textit{generalized $N$-function} if it is Carath\'eodory mapping and $\psi(x,\cdot)\colon \mathbb{R}_{\ge 0}\to \mathbb{R}_{\ge 0}$ is for a.e.\ $x \in M$ an
	$N$-function. For a generalized $N$-function $\psi \colon  M\times\mathbb{R}_{\ge 0} \to\mathbb{R}_{\ge 0}$ and a (Lebesgue) measurable function $f\in L^0(M;\mathbb{R}^{\textcolor{black}{\ell}})$,~$\textcolor{black}{\ell}\in \mathbb{N}$, we define the \textit{modular (with respect to $\psi$)} by
	\begin{align*}
		\rho_{\psi,M}(v)\coloneqq \int_M \psi(\cdot,\vert v\vert )\,\textup{d}x\,.
	\end{align*} For a generalized 
	$N$-function  $\psi \colon  M\times\mathbb{R}_{\ge 0} \to\mathbb{R}_{\ge 0}$ and $\textcolor{black}{\ell}\in \mathbb{N}$, the \textit{generalized Orlicz space} is defined by
	\begin{align*}
		L^{\psi}(M;\mathbb{R}^{\textcolor{black}{\ell}})\coloneqq \big\{v\in L^0(M;\mathbb{R}^{\textcolor{black}{\ell}})\mid
		\rho_{\psi,M}(v)<\infty\big\}\,.
	\end{align*}
	The Luxembourg norm 
	$\smash{\norm {v}_{\psi,M}}\coloneqq  \smash{\inf \set{\lambda >0\mid
			\rho_{\psi,M}(v/\lambda) \le 1}}$ turns
	$L^\psi(M;\mathbb{R}^{\textcolor{black}{\ell}})$ into a Banach space (\textit{cf}.\  \cite[Thm.\  2.3.13]{dhhr}).  If $\psi\colon M\times \mathbb{R}_{\ge 0}\to \mathbb{R}_{\ge 0}$ is a generalized
	$N$-function, then, for every $v\in L^{\psi}(M;\mathbb{R}^{\textcolor{black}{\ell}})$ and
	$u\in L^{\psi^*}(M;\mathbb{R}^{\textcolor{black}{\ell}})$, it holds the \textit{generalized Hölder inequality} (\textit{cf}.\ \cite[Lem.\ 2.6.5]{dhhr})
	\begin{align}\label{eq:gen_hoelder}
		(u,v)_M\leq 2\,\|u\|_{\psi^*,M}\,\|v\|_{\psi,M}\,.
	\end{align}
	
	\subsection{Basic properties of the non-linear operators}\label{sec:basic}
	
	\hspace*{5mm}Throughout the entire paper, we always assume that  $\bfS\colon\Omega\times \mathbb{R}^{d\times d}\to \mathbb{R}^{d\times d}_{\textup{sym}}$ has \textit{$(p(\cdot),\delta)$-structure}, where $p\in \mathcal{P}^{\infty}(\Omega)$ with $p^->1$, $\delta\ge 0$, and $\mu_0>0$, \textit{i.e.}, for a.e.\ $x\in \Omega$ and~every~${\bfA\in \mathbb{R}^{d\times d}}$,~it~holds~that
	\begin{align}\label{def:A}
		\bfS(x,\bfA)\coloneqq \mu_0\,(\delta+\vert \bfA^{\textup{sym}}\vert )^{p(x)-2}\bfA^{\textup{sym}}\,.
	\end{align}
	For given $p\in \mathcal{P}^{\infty}(\Omega)$ with $p^->1$ and $\delta\ge 0$, we introduce the \textit{special generalized $N$-function}~$\varphi\coloneqq \varphi_{p,\delta}\colon\Omega\times\mathbb{R}_{\ge 0}\to \mathbb{R}_{\ge 0}$, for a.e.\ $x\in \Omega$ and all $t\ge 0$, defined by
	\begin{align} 
		\label{eq:def_phi} 
		\varphi(x,t)\coloneqq \int _0^t \varphi'(x,s)\, \mathrm ds\,,\quad\text{where}\quad
		\varphi'(x,t) \coloneqq (\delta +t)^{p(x)-2} t\,.
	\end{align}
	For (Lebesgue) measurable functions $f,g\colon\Omega \to \mathbb{R}_{\ge 0}$, we write
	$f\sim g$ (or $f\lesssim g$) if~there~exists~a constant \hspace{-0.1mm}$c\hspace{-0.1em}>\hspace{-0.1em}0$ \hspace{-0.1mm}such \hspace{-0.1mm}that \hspace{-0.1mm}$c^{-1}g\hspace{-0.1em}\leq\hspace{-0.1em} f\hspace{-0.1em}\leq\hspace{-0.1em} c\,g$ \hspace{-0.1mm}(or \hspace{-0.1mm}$ f\hspace{-0.1em}\leq\hspace{-0.1em} c\,g$) \hspace{-0.1mm}a.e.\ \hspace{-0.1mm}in \hspace{-0.1mm}$\Omega$. \hspace{-0.1mm}In \hspace{-0.1mm}particular,~i\hspace{-0.1mm}f~\hspace{-0.1mm}not~\hspace{-0.1mm}otherwise~\hspace{-0.1mm}specified, we always assume that the hidden constant in $\sim$ and $\lesssim$ depends only on $p^-,p^+>1$, $\delta\ge 0$, and $\mu_0 >0$.
	
	Then, $\varphi\colon \Omega\times\mathbb{R}_{\ge 0}\to \mathbb{R}_{\ge 0}$ satisfies, uniformly with respect to $\delta\ge  0$ and a.e.\ $x\in \Omega$, the
	$\Delta_2$-condition with $\textup{ess\,sup}_{x\in \Omega}{\Delta_2(\varphi(x,\cdot))}\lesssim 2^{\smash{\max \{2,p^+\}}}$. In addition, 
	the (Fenchel) conjugate function (with respect to the second argument) $\varphi^*\colon\Omega\times\mathbb{R}_{\ge 0}\to \mathbb{R}_{\ge 0}$ satisfies, uniformly~with~respect~to~${t \ge 0}$,~${\delta \ge 0}$,~and a.e.\ $x\in \Omega$, $\varphi^*(x,t) \sim
	(\delta^{p(x)-1} + t)^{p'(x)-2} t^2$ and the $\Delta_2$-condition with
	$\textup{ess\,sup}_{x\in \Omega}{\Delta_2(\varphi^*(x,\cdot))} \lesssim 2^{\smash{\max \{2,(p^-)'\}}}$.\newpage
	
	For a generalized $N$-function $\psi\colon\Omega\times \mathbb{R}_{\ge 0}\to \mathbb{R}_{\ge 0}$, we introduce \textit{shifted generalized $N$-functions} $\psi_a\colon\Omega\times \mathbb{R}_{\ge 0}\to \mathbb{R}_{\ge 0}$, ${a\ge 0}$, for a.e.\ $x\in \Omega$ and all $a,t\ge 0$, defined by
	\begin{align}
		\label{eq:phi_shifted}
		\psi_a(x,t)\coloneqq \int _0^t \psi_a'(x,s)\, \mathrm ds\,,\quad\text{where}\quad
		\psi'_a(x,t)\coloneqq \psi'(x,a+t)\frac {t}{a+t}\,.
	\end{align}
	
	\begin{remark} \label{rem:phi_a}
		For the special $N$-function $\varphi\colon \Omega\times \mathbb{R}_{\ge 0}\to \mathbb{R}_{\ge 0}$ (\textit{cf}.\ \eqref{eq:def_phi}), uniformly~with~respect~to~${a,t\ge 0}$ and a.e.\ $x\in \Omega$, it holds that
		\begin{align}
			\varphi_a(x,t)& \sim (\delta+a+t)^{p(x)-2} t^2\,,\label{rem:phi_a.1}\\
			(\varphi_a)^*(x,t)
			&	\sim ((\delta+a)^{p(x)-1} + t)^{\smash{p'(x)-2}} t^2\,.\label{rem:phi_a.2}
		\end{align}
		The families $\{\varphi_a\}_{\smash{a \ge 0}},\{(\varphi_a)^*\}_{\smash{a \ge 0}}\colon\Omega\times  \mathbb{R}_{\ge 0}\to \mathbb{R}_{\ge 0}$ satisfy, uniformly with respect to $a\ge  0$,~the~\mbox{$\Delta_2$-condi-} tion with 
		\begin{align}
			\textup{ess\,sup}_{x\in \Omega}{\Delta_2(\varphi_a(x,\cdot))}& \lesssim 2^{\smash{\max \{2,p^+\}}}\,,\label{rem:phi_a.3}\\
			\textup{ess\,sup}_{x\in \Omega}{\Delta_2((\varphi_a)^*(x,\cdot))} &\lesssim 2^{\smash{\max \{2,(p^-)'\}}}\,,\label{rem:phi_a.4}
		\end{align}
		respectively.
		%${\textup{ess\,sup}_{x\in \Omega}{\Delta_2(\varphi_a(x,\cdot))} \lesssim 2^{\smash{\max \{2,p^+\}}}}\!$ and
	%	${\textup{ess\,sup}_{x\in \Omega}{\Delta_2((\varphi_a)^*(x,\cdot))} \lesssim 2^{\smash{\max \{2,(p^-)'\}}}}\!$,~respectively.
	\end{remark}
	
	Closely related to the extra-stress tensor $\bfS\colon\Omega\times\mathbb{R}^{d\times d}\to \mathbb{R}^{d\times d}_{\textup{sym}}$ defined by \eqref{def:A}~are~the~non-linear mappings $\bfF,\bfF^*\colon\Omega\times\mathbb{R}^{d\times d}\to \mathbb{R}^{d\times d}_{\textup{sym}}$, for a.e.\ $x\in \Omega$ and every $\bfA\in \mathbb{R}^{d\times d}$ defined by
	\begin{align}
		\begin{aligned}
			\bfF(x,\bfA)&\coloneqq (\delta+\vert \bfA^{\textup{sym}}\vert)^{\smash{\frac{p(x)-2}{2}}}\bfA^{\textup{sym}}\,,\\ \bfF^*(x,\bfA)&\coloneqq (\delta^{p(x)-1}+\vert \bfA^{\textup{sym}}\vert)^{\smash{\frac{p'(x)-2}{2}}}\bfA^{\textup{sym}}
			\,.\label{eq:def_F}
		\end{aligned}
	\end{align}

	The relations between
	$\bfS,\bfF,\bfF^*\colon\Omega\times\mathbb{R}^d
	\to \mathbb{R}^d$ and
	$\varphi_a,(\varphi^*)_a,(\varphi_a)^*\colon\Omega\times\mathbb{R}_{\ge
		0}\to \mathbb{R}_{\ge
		0}$,~${a\ge 0}$, are presented in
	the following proposition.
	
	\begin{proposition}
		\label{lem:hammer}
		Uniformly with respect to every $t\ge 0$, 
		$\bfA, \bfB \in \mathbb{R}^{d\times d}$, and a.e.\ $ x,y\in \Omega$,~we~have~that
		\begin{align}
			(\bfS(x,\bfA) - \bfS(x,\bfB))
			\cdot(\bfA-\bfB ) &\sim \smash{\vert \bfF(x,\bfA) - \bfF(x,\bfB)\vert^2}\notag
			\\&\sim \varphi_{\vert \bfA^{\textup{sym}} \vert }(x,\vert \bfA^{\textup{sym}} - \bfB^{\textup{sym}} \vert )\label{eq:hammera}
			\\&\sim (\varphi_{\vert \bfA^{\textup{sym}} \vert })^*(x,\vert \bfS(x,\bfA)-\bfS(x,\bfB)\vert )
			\,,\notag\\
			\smash{\vert \bfF^*(x,\bfA) - \bfF^*(x,\bfB)\vert^2}
			\label{eq:hammerf}
			&\sim \smash{\smash{(\varphi^*)}_{\smash{\vert \bfA^{\textup{sym}} \vert }}(x,\vert \bfA^{\textup{sym}}  - \bfB^{\textup{sym}} \vert )}\,,\\
			\label{eq:hammerg}
			\smash{\smash{(\varphi^*)}_{\smash{\vert \bfS(x,\bfA)\vert }}(x,t)}
			&\sim \smash{\smash{(\varphi}_{\smash{\vert \bfA^{\textup{sym}} \vert }})^*(x,t)}\,,\\
			\label{eq:hammerh}
			\smash{\vert \bfF^*(x,\bfS(x,\bfA)) - \bfF^*(x,\bfS(y,\bfB))\vert^2}
			&\sim \smash{\smash{(\varphi}_{\smash{\vert \bfA^{\textup{sym}} \vert }})^*(x,\vert \bfS(x,\bfA)-\bfS(y,\bfB)\vert)}\,.
		%	\label{eq:hammeri}
		%	\smash{\vert \bfS(x,\bfA) - \bfS(x,\bfB))\vert^2}
		%	&\lesssim (\delta+\vert \bfA\vert +\vert \bfB\vert)^{p(x)-2}\vert \bfA-\bfB\vert\,.
		\end{align}
	\end{proposition} 
	
	\begin{proof}
		For the equivalences \eqref{eq:hammera}--\eqref{eq:hammerg}, we refer to \cite[Rem. A.9]{BDS15}. The equivalence 
		\eqref{eq:hammerh} follows from the equivalences \eqref{eq:hammerf} and \eqref{eq:hammerg}.
	\end{proof}
	
	In addition, we will frequently resort to the following shift change result.
	
	\begin{lemma}\label{lem:shift-change}
		For every $\varepsilon>0$, there exists $c_\varepsilon \geq 1$ (depending only
		on~$\varepsilon>0$, $p^-,p^+>1$, and $\delta\ge 0$) such that for every $t\ge 0$, 
		$\bfA, \bfB \in \mathbb{R}^{d\times d}$, and a.e.\ $ x\in \Omega$, it holds that
		\begin{align}
			\varphi_{\vert \bfA^{\textup{sym}} \vert}(x,t)&\leq c_\varepsilon\, \varphi_{\vert \bfB^{\textup{sym}} \vert }(x,t)
			+\varepsilon\, \vert \bfF(x,\bfA) - \bfF(x,\bfB)\vert^2\,,\label{lem:shift-change.1}
			\\
			(\varphi_{\vert \bfA^{\textup{sym}} \vert})^*(x,t)&\leq c_\varepsilon\, (\varphi_{\vert \bfB^{\textup{sym}} \vert })^*(x,t)
			+\varepsilon\, \vert \bfF(x,\bfA) - \bfF(x,\bfB)\vert^2\,.\label{lem:shift-change.3}
		\end{align}
	\end{lemma}
	
	\begin{proof}
		See \cite[Rem. A.9]{BDS15}.
	\end{proof}

	\begin{remark}
		\label{rem:natural_dist}
		Due to \eqref{eq:hammera}, uniformly with respect to $\bfu, \bfz \in W^{1,p(\cdot)}(\Omega;\mathbb{R}^d)$,~it~holds~that
		\begin{align*}
			(\bfS(\cdot,\bfD \bfu) -
			\bfS(\cdot,\bfD \bfz),\bfD \bfu - \bfD \bfz)_\Omega
			&\sim
			\|\bfF(\cdot,\bfD \bfu)-\bfF(\cdot,\bfD \bfz)\|_{2,\Omega}^2 \\&\sim \rho_{\varphi_{\vert \bfD \bfu\vert},\Omega}(\bfD \bfu -
			\bfD \bfz)\,.
		\end{align*}
		We refer to all three equivalent quantities as the \textup{natural distance}.  Note that $\varphi_{\vert \bfD \bfu\vert}\colon \Omega\times\mathbb{R}_{\ge 0}\to \mathbb{R}_{\ge 0}$ for every $\bfu\in  W^{1,p(\cdot)}(\Omega;\mathbb{R}^d)$ is a generalized $N$-function. 
	\end{remark}
	
	\subsection{$\log$-Hölder continuity and important related results}
	
	\hspace*{5mm}In this subsection, we recall the notion of $\log$-Hölder continuity of variable exponents and collect important related results that are used in this paper.\enlargethispage{5mm}

	For an open set $G\subseteq \setR^d$, $d\in \setN$,  we say that a bounded exponent $p\in \mathcal P^\infty (G)$ is locally
	\textit{$\log$-Hölder continuous}, if there is a constant $c_1>0$ such that
	for every $x,y\in G$, it holds that
	\begin{align*}
		\vert p(x)-p(y)\vert \leq \frac{c_1}{\log(e+1/\vert x-y\vert)}\,.
	\end{align*}
	We say that $p \in \mathcal P^\infty (G)$ satisfies the \textit{$\log$-Hölder decay condition}, if there exist 
	constants $c_2>0$ and $p_\infty\in \setR$ such that for every $x\in G$, it holds that
	\begin{align*}
		\vert p(x)-p_\infty\vert \leq\frac{c_2}{\log(e+\vert x\vert)}\,.
	\end{align*}
	We say that $p$ is \textit{globally $\log$-Hölder continuous} on $G$, if it is locally 
	$\log$-Hölder continuous and satisfies the $\log$-Hölder decay condition. 
	The constants $c_1$ and $c_2$ are called the \textit{local $\log$-Hölder constant and the $\log$-Hölder decay constant}, respectively. Then, $c_{\log}(p)\coloneqq \max\{c_1,c_2\}$ is called the \textit{$\log$-Hölder constant.}
	Moreover,
	we denote by $\mathcal{P}^{\log}(G)$, the  \textit{set of  all
		globally $\log$-Hölder continuous
		variable exponents on $G$}.
	
	We recall four fundamental results, which will find use in the sequel:
	
	For a cube $Q\subseteq \mathbb{R}^d$, $d\in \mathbb{N}$, we denote by $\ell(Q)>0$ the corresponding side length. Then, we have the following result which implies a
	discrete norm equivalence between Luxembourg norms with respect to $\log$-Hölder continuous exponents and their element-wise constant approximations (\textit{cf}.\  Lemma \ref{lem:norm_equiv}).
	
	\begin{lemma}\label{lem:local_change}
		Let $p\in \mathcal{P}^{\log}(\mathbb{R}^d)$, $d\in \mathbb{N}$. Then,  for every $m>0$, there exists a constant $c>0$, depending only on $m$, $c_{\log}(p)$, and $p^+$, such that for every cube (or ball) $Q\subseteq \mathbb{R}^d$ with~$\ell(Q)\leq 1$,~$a\in [0,1]$, and $t\ge 0$ with $\vert Q\vert^m\leq t \leq \vert Q\vert^{-m}$, for every $x,y\in Q$, it holds that
		\begin{align*}
			(a+t)^{p(x)-p(y)}\leq c\,.
		\end{align*}
	\end{lemma}
	
	\begin{proof}
		See \cite[Lem.\  2.1]{BDS15}.
	\end{proof}
	
	A crucial role in the hereinafter analysis plays the following substitute of Jensen's inequality for shifted generalized $N$-functions, the so-called \textit{key estimate}.\vspace*{-0.5mm}
	
	\begin{lemma}[Key estimate]\label{lem:key-estimate}
		Let $p\in \mathcal{P}^{\log}(\mathbb{R}^d)$, $d\in \mathbb{N}$. Then, for every $m>0$, there exists a constant $c>0$, depending only on $m$, $c_{\log}(p)$, and $p^-$, such that for every cube (or ball) $Q\subseteq \mathbb{R}^d$~with~$\ell(Q)\leq 1$, $a\ge 0$, and $z\in L^{p'(\cdot)}(Q)$ with $a+\langle \vert z\vert \rangle_Q\leq \vert Q\vert^{-m}$, for every $x\in Q$, it holds that\vspace*{-0.5mm}
		\begin{align*}
			(\varphi_a)^*(x,\langle  \vert z\vert \rangle_Q)\leq c\, \langle (\varphi_a)^*(\cdot,\vert z\vert )\rangle_Q+c\,\vert Q\vert^m\,.
		\end{align*}
	\end{lemma}
	
	\begin{proof}
		Follows along the lines of the proof of \cite[Thm.\  2.4]{BDS15} up to minor adjustments (\textit{e.g.}, using~\eqref{rem:phi_a.2} instead of \eqref{rem:phi_a.1}).
	\end{proof}
	
	In addition, we have the following generalizations of Sobolev's and Korn's inequality.
	
	\begin{theorem}[Sobolev's inequality]\label{thm:sobolev}
		Let $G\subseteq \mathbb{R}^d$, $d\in \mathbb{N}$, be a bounded domain and $p\in \mathcal{P}^{\log}(G)$ with $1<p^-\leq p^+<d$. Moreover, denote by $p^*\coloneqq \smash{\frac{dp}{d-p}}\in  \mathcal{P}^{\log}(G)$ the \textit{Sobolev conjugate~exponent}. Then, for every $\smash{z\in W_0^{1,p(\cdot)}(G)}$, it holds that $z\in L^{p^*(\cdot)}(G)$ with\vspace*{-0.5mm}
		\begin{align*}
			\|z\|_{p^*(\cdot),G}\lesssim \|\nabla z\|_{p(\cdot),G}\,,
		\end{align*}
		where $\lesssim$ depends only on $d$, $c_{\log}(p)$, and $p^+$.\vspace*{-0.5mm}
	\end{theorem}
	
	\begin{proof}
		See \cite[Thm.\  8.3.1]{dhhr}.
	\end{proof}
	
	\begin{theorem}[Korn's inequality]\label{thm:korn}
		Let $G\subseteq \mathbb{R}^d$, $d\in \mathbb{N}$, be a bounded domain and $p\in \mathcal{P}^{\log}(G)$ with $p^->1$. Then,  for every $\bfz\in W_0^{1,p(\cdot)}(G;\mathbb{R}^d)$, it holds that
		\begin{align*}
			\|\bfD \bfz\|_{p(\cdot),G}\lesssim \|\nabla \bfz\|_{p(\cdot),G}\,,
		\end{align*}
		where $\lesssim$ depends only on $d$, $c_{\log}(p)$, and $p^+$.\vspace*{-0.5mm}
	\end{theorem}
	
	\begin{proof}
		See \cite[Thm.\  14.3.21]{dhhr}.
	\end{proof}

	\section{The  steady $p(\cdot)$-Navier--Stokes equations}\label{sec:p-navier-stokes}\vspace*{-0.5mm}
	
	\hspace{5mm}In this section, let us recall some  well-known facts about the steady $p(\cdot)$-Navier--Stokes equations~\eqref{eq:p-navier-stokes}; for a thorough analytical examination of this problem, please refer to the contributions \cite{AM,der-2d-erf,huber-paper,dms,dhhr}.\vspace*{-0.5mm}\enlargethispage{6.5mm}
	
	\subsection{Weak formulations}
	
	\hspace*{5mm}We define the following function spaces:
	\begin{align*}
		\begin{aligned}
			V&\coloneqq W^{1,p(\cdot)}(\Omega;\mathbb{R}^d)\,,&&\Vo\coloneqq W^{1,p(\cdot)}_0(\Omega;\mathbb{R}^d)\,,\\[-0.5mm]
			Q&\coloneqq L^{p'(\cdot)}(\Omega)\,,&&\Qo\coloneqq L^{p'(\cdot)}_0(\Omega)\coloneqq \big\{z\in L^{p'(\cdot)}(\Omega)\mid \langle z\rangle_\Omega=0\big\}\,.
		\end{aligned}
	\end{align*}
	With this notation, assuming that $p^-\ge \smash{\frac{3d}{d+2}}$, the weak formulation of the steady $p(\cdot)$-Navier--Stokes equations \eqref{eq:p-navier-stokes} as a non-linear saddle point like problem is the following:
	
	\textit{Problem (Q).}\hypertarget{Q}{} For  given $\bff\in L^{p'(\cdot)}(\Omega;\mathbb{R}^d)$, find $(\bfv,q)^\top\in \Vo\times \Qo$ such that for every $(\bfz,z)^\top\in \Vo\times Q $,\linebreak\hspace*{4.5mm} it holds that\vspace*{-0.5mm}
	\begin{align*}
		(\bfS(\cdot,\bfD\bfv),\bfD\bfz)_{\Omega}+([\nabla \bfv]\bfv,\bfz)_{\Omega}-(q,\divo\bfz)_{\Omega}&=(\bff,\bfz)_{\Omega}\,,\\
		(\divo\bfv,z)_{\Omega}&=0\,.
	\end{align*}
	Equivalently, one can reformulate Problem (\hyperlink{P}{P}) \textit{``hiding''} the pressure.
	
	\textit{Problem (P).}\hypertarget{P}{} For given $\bff\in L^{p'(\cdot)}(\Omega;\mathbb{R}^d)$, find $\bfv\in \Vo(0)$ such that for every $\bfz\in\Vo(0)$, it holds that\vspace*{-0.5mm}
	\begin{align*}
		(\bfS(\cdot,\bfD\bfv),\bfD\bfz)_{\Omega}+([\nabla \bfv]\bfv,\bfz)_{\Omega}=(\bff,\bfz)_{\Omega}\,,
	\end{align*}
	\hspace{5mm}where $\Vo(0)\coloneqq \{\bfz\in \Vo\mid  (\divo\bfz,z)_{\Omega}=0\text{ for all }z\in Q\}$.
	
	The names \textit{Problem (\hyperlink{Q}{Q})} and \textit{Problem (\hyperlink{P}{P})} are traditional in the literature (\textit{cf}.\  \cite{BF1991,bdr-phi-stokes,BBD15}).~The~well-posedness of Problem (\hyperlink{Q}{Q}) and Problem (\hyperlink{P}{P}) is usually proved in two steps:
	first, using pseudo-monotone operator theory (\textit{cf}.\ \cite{rubo}), the well-posedness of Problem (\hyperlink{P}{P}) is shown; then, given the well-posedness of Problem (\hyperlink{P}{P}), the well-posedness of Problem (\hyperlink{Q}{Q})
	follows using the following~\mbox{inf-sup}~stability~result:
	\begin{lemma}\label{lem:inf_sup}
		Let $p\in \mathcal{P}^{\log}(\Omega)$ with $p^->1$. Then,~for~every~$z\in \Qo$, it holds that\vspace*{-0.5mm}
		\begin{align*}
			\|z\|_{p'(\cdot),\Omega}\lesssim \sup_{\bfz\in \Vo\,:\,\smash{\|\nabla\bfz\|_{p(\cdot),\Omega}\leq 1}}{(z,\divo\bfz)_{\Omega}}\,,
		\end{align*}
		where $\lesssim$ depends only on $d$, $p^-$, $p^+$, $c_{\log}(p)$, and $\Omega$.
	\end{lemma}
	
	\begin{proof}
		See \cite[Thm.\  14.3.18]{dhhr}.\vspace*{-0.5mm}
	\end{proof}
	
	Note that the restriction $p^-\ge\smash{ \frac{3d}{d+2}}$ is only needed to ensure the well-posedness of the weak convective term. If the latter is omitted in Problem (\hyperlink{Q}{Q}) and Problem (\hyperlink{P}{P}), we can even consider~the~case~$p^->1$.\vspace*{-0.5mm}
	
	\subsection{Regularity assumptions}\vspace*{-0.5mm}

	\hspace*{5mm}According to \cite[Rem.\ 4.5]{BDS15}, for a non-Lipschitz but only Hölder continuous power-law index, \textit{i.e.}, $p\in C^{0,\alpha}(\overline{\Omega})$ with $\alpha<1$,
 one cannot hope for the \textit{``full''}~regularity~${\bfF(\cdot,\bfD\bfv)\in W^{1,2}(\Omega;\mathbb{R}^{d\times d})}$,~in~general, but  instead it is reasonable to expect the \textit{``partial''} regularity\enlargethispage{2.5mm}
 	\begin{alignat}{2}\label{eq:natural_regularity.velocity}
 		\bfF(\cdot,\bfD\bfv)&\in \smash{N^{\beta,2}(\Omega;\mathbb{R}^{d\times d})}&&\quad\text{ for some  }\beta\in (0,1]\,.
 \intertext{
 Concerning the regularity of the pressure, we propose to consider the \textit{``partial''} regularity}
\label{eq:natural_regularity.pressure}
			q& \in  \smash{H^{\gamma,p'(\cdot)}(\Omega)}&&\quad\text{ for some }\gamma\in (0,1]\,.
	\end{alignat}
	A \hspace*{-0.1mm}first \hspace*{-0.1mm}important \hspace*{-0.1mm}consequence \hspace*{-0.1mm}of \hspace*{-0.1mm}the \hspace*{-0.1mm}regularity \hspace*{-0.1mm}assumptions \hspace*{-0.1mm}\eqref{eq:natural_regularity.velocity}, \hspace*{-0.1mm}\eqref{eq:natural_regularity.pressure} \hspace*{-0.1mm}is \hspace*{-0.1mm}an \hspace*{-0.1mm}improved \hspace*{-0.1mm}integrability~\hspace*{-0.1mm}result.\vspace*{-0.5mm}
	
	\begin{lemma}\label{lem:improved_integrability}
		Let $p\in C^0(\overline{\Omega})$ with $p^->1$. Then, the following statements apply:
		\begin{itemize}[noitemsep,topsep=2pt,leftmargin=!,labelwidth=\widthof{(ii)}]
			\item[(i)] If $\bfz \hspace*{-0.1em}\in\hspace*{-0.1em} V$ with $\bfF(\cdot,\bfD\bfz) \hspace*{-0.1em}\in\hspace*{-0.1em} N^{\beta,2}(\Omega;\mathbb{R}^{d\times d})$, then 
			$\bfz\hspace*{-0.1em}\in\hspace*{-0.1em} W^{1,\smash{\frac{dp(\cdot)}{d-2\beta}}}(\Omega;\mathbb{R}^d)$ if $2\beta \hspace*{-0.1em}<\hspace*{-0.1em}d$~and~${\bfz\hspace*{-0.1em}\in \hspace*{-0.1em}W^{1,s}(\Omega;\mathbb{R}^d)}$ for all $s\in (1,+\infty)$ if $2\beta =d$.
			
			\item[(ii)] If $z \in  H^{\gamma,p'(\cdot)}(\Omega)$, then  $z\in  L^{r(\cdot)} (\Omega)$	for every  $r\in C^0(\overline{\Omega})$ with $r+\varepsilon\leq (p'(\cdot))_\gamma^*$ in $\Omega$ for some $\varepsilon>0$, where $t_\gamma^*\coloneqq\frac{d t}{d-\gamma t}$ if $t< \frac{d}{\gamma}$ and $t_\gamma^*=\infty$ if $t\ge\frac{d}{\gamma}$.
		\end{itemize}

	\end{lemma}
	
	\begin{proof}
		\textit{ad (i).} If $2\beta<d$, then  $N^{\beta,2}(\Omega)\hookrightarrow L^{\smash{\frac{2d}{d-2\beta}}}(\Omega)$
		and if $2\beta =d$, then $N^{\beta,2}(\Omega)\hookrightarrow L^s(\Omega)$ for all $s\hspace*{-0.1em}\in\hspace*{-0.1em} [1,\infty)$.
		Thus, if $2\beta \hspace*{-0.1em}<\hspace*{-0.1em}d$, then $\vert \bfD\bfz\vert \hspace*{-0.1em}\in\hspace*{-0.1em} L^{\smash{\frac{dp(\cdot)}{d-2\beta}}}(\Omega)$ and if $2\beta \hspace*{-0.1em}=\hspace*{-0.1em}d$,~then~${\vert \bfD\bfz\vert\hspace*{-0.1em} \in\hspace*{-0.1em} L^s(\Omega)}$~for~all~${s\hspace*{-0.1em}\in\hspace*{-0.1em} [1,\infty)}$.
		
		\textit{ad (ii).} Due to the continuity of $r,p'\colon \overline{\Omega}\to (1,+\infty)$,   there exists a covering of $\overline{\Omega}$~by~open~balls~${B_i\subseteq \mathbb{R}^d}$, $i=1,\ldots,m$, $m\in \mathbb{N}$, such that, setting   $r^+_i=\sup_{\smash{x\in  B_i}}{r(x)}$ and $p^+_i=\sup_{\smash{x\in  B_i}}{p(x)}$ for all $i=1,\ldots,m$, we have that
		$\smash{r^+_i \leq	((p^+_i)')_\gamma^*}$. 
		Therefore, since
		$H^{\gamma,p'(\cdot)}(B_i)\hookrightarrow H^{\gamma,\smash{(p^+_i)'}}(B_i)$,  ${H^{\gamma,\smash{(p^+_i)'}}(B_j)\hookrightarrow L^{\smash{((p^+_i)')_\gamma^*}}(B_j)}$ (\textit{cf}.\  \cite[Prop.\ 1.4]{Yang03}), and $L^{((p^+_i)')_\gamma^*}(B_i)\hookrightarrow L^{r^+_i}(B_i)\hookrightarrow L^{r(\cdot)}(B_i)$~for~all~${i=1,\ldots,m}$, we~obtain~${q\in L^{r(\cdot)}(\Omega)}$.
	\end{proof}

	The following lemma shows that in the case of \textit{``full regularity''} assumptions on the power-law index, \textit{i.e.}, $p\in C^{0,1}(\overline{\Omega})$, and the velocity vector field, \textit{i.e.}, $\bfF(\cdot,\bfD\bfv)\in W^{1,2}(\Omega;\mathbb{R}^{d\times d})$, and if $p^-\ge  2$~and~${\delta>0}$, one can equally expect the pressure to have  full regularity.\vspace*{-0.5mm}

	\begin{lemma}\label{lem:pres}
		Let $p\in C^{0,1}(\overline{\Omega})$ with $p^-\ge  2$, let $\delta>0$, and let $(\bfv,q)^\top \in \Vo(0)\times \Qo$  be a weak solution of Problem (\hyperlink{Q}{Q}) such that
		$\bfF(\cdot,\bfD\bfv)\in W^{1,2}(\Omega;\mathbb{R}^{d\times d})$.  Then, the following~statements~apply:
		\begin{itemize}[noitemsep,topsep=2pt,leftmargin=!,labelwidth=\widthof{(ii)}]
			\item[(i)]  If $\bff\in L^{\textcolor{black}{p'(\cdot)}}(\Omega;\mathbb{R}^d)$, then it holds that $q \in  W^{1,p'(\cdot)}(\Omega)$. 
			\item[(ii)] If $\bff\in L^2(\Omega;\mathbb{R}^d)$, then it holds that $(\delta    +\vert\bfD\bfv\vert)^{2-p(\cdot)}\vert\nabla q\vert ^2 \in L^1(\Omega)$.
		\end{itemize}
	\end{lemma} 
	
	\begin{proof}
		\textit{ad (i).} Analogously to \cite[Lems. \textcolor{black}{2.45--2.47}]{BK23_pxDirichlet}, abbreviating the flux $\smash{\widehat{\bfS}}\coloneqq \bfS(\cdot,\bfD\bfv)\in L^{p'(\cdot)}(\Omega;\mathbb{R}^{d\times d})$, we deduce that $\bfF^*(\cdot,\smash{\widehat{\bfS}})\in W^{1,2}(\Omega;\mathbb{R}^{d\times d})$~with
		\begin{alignat}{2}
			\vert\nabla  \bfF(\cdot,\bfD\bfv)\vert+(1+\vert \bfD\bfv\vert^{p(\cdot)s}) &\sim \vert\nabla  \bfF^*(\cdot,\smash{\widehat{\bfS}})\vert
			+(1+\vert\smash{\widehat{\bfS}}\vert^{p'(\cdot)s})&&\quad \text{ a.e.\ in }\Omega\,,\label{lem:pres.1}\\[-0.5mm]
			\vert \nabla \bfF(\cdot,\bfD\bfv)\vert^2+\mu(\bfv)&\sim (\delta+\vert \bfD\bfv\vert)^{p(\cdot)-2}\vert \nabla \bfD\bfv\vert^2+\mu(\bfv)&&\quad \text{ a.e.\ in }\Omega\,,\label{lem:pres.2}\\[-0.5mm]
			\vert\nabla  \bfF^*(\cdot,\smash{\widehat{\bfS}})\vert^2+\mu^*(\smash{\widehat{\bfS}})&\sim (\delta^{p(\cdot)-1}+\vert \smash{\widehat{\bfS}}\vert)^{p'(\cdot)-2}\vert \nabla \smash{\widehat{\bfS}}\vert^2+\mu^*(\smash{\widehat{\bfS}})&&\quad \text{ a.e.\ in }\Omega\,,\label{lem:pres.3}
		\end{alignat}
		where $s>1$ is a constant which can chosen to be close to $1$, so that, by Lemma \ref{lem:improved_integrability}, we have that $(1+\vert\smash{\widehat{\bfS}}\vert^{p'(\cdot)s})\sim (1+\vert \bfD\bfv\vert^{p(\cdot)s})\in L^1(\Omega)$, and where
		\begin{align*}
			\mu(\bfv)&\coloneqq \vert \ln(\delta+\vert \bfD\bfv\vert)\vert^2(\delta+\vert \bfD\bfv\vert)^{p(\cdot)-2}\vert \bfD\bfv\vert^{\textcolor{black}{2}}  \vert \nabla p\vert^{\textcolor{black}{2}} 
			\in L^1(\Omega)\,,\\[-0.5mm]
			\mu^*(\smash{\widehat{\bfS}})&\coloneqq \vert \ln(\delta^{p(\cdot)-1}+\vert \smash{\widehat{\bfS}}\vert)\vert^2(\delta^{p(\cdot)-1}+\vert \smash{\widehat{\bfS}}\vert)^{p'(\cdot)-2}\vert \smash{\widehat{\bfS}}\vert^{\textcolor{black}{2}} \vert \nabla p\vert^{\textcolor{black}{2}} 
			\in L^1(\Omega)\,.
		\end{align*} 
		Due to $p^-\ge  2$ and $\delta>0$, from \eqref{lem:pres.2}, it follows that $\bfD\bfv \in  W^{1,2}(\Omega;\mathbb{R}^{d\times d})$. By the usual~\mbox{algebraic}~\mbox{identity} (\textit{cf}.\  \hspace*{-0.1mm}\cite[Lem.\  \hspace*{-0.1mm}6.3.]{mnr3}), \hspace*{-0.1mm}this \hspace*{-0.1mm}implies \hspace*{-0.1mm}that  \hspace*{-0.1mm}$ \bfv\hspace*{-0.15em}\in \hspace*{-0.15em}W^{2,2}(\Omega;\mathbb{R}^d)$ \hspace*{-0.1mm}and, \hspace*{-0.1mm}thus, \hspace*{-0.1mm}$\bfv\hspace*{-0.15em} \in\hspace*{-0.15em} L^\infty(\Omega;\mathbb{R}^d)$~\hspace*{-0.1mm}and~\hspace*{-0.1mm}${\bfv\otimes \bfv\hspace*{-0.15em}\in\hspace*{-0.15em} W^{1,2}(\Omega;\mathbb{R}^{d\times d})}$.  On the other hand, due to $(\delta^{p(\cdot)-1}+\vert \smash{\widehat{\bfS}}\vert)^{p'(\cdot)-2}\sim (\delta+\vert \bfD\bfv\vert)^{2-p(\cdot)}$, from \eqref{lem:pres.1} and \eqref{lem:pres.3}, it follows that $(\delta+\vert \bfD\bfv\vert)^{p(\cdot)-2}\vert \nabla \smash{\widehat{\bfS}}\vert^2\in L^1(\Omega)$ with 
		\begin{align}
				\smash{\vert\nabla  \bfF(\cdot,\smash{\widehat{\bfS}})\vert^2+	\mu^*(\smash{\widehat{\bfS}})\sim (\delta+\vert \bfD\bfv\vert)^{2-p(\cdot)}\vert \nabla \smash{\widehat{\bfS}}\vert^2+	\mu^*(\smash{\widehat{\bfS}})\quad \text{ a.e.\ in }\Omega\,.}\label{lem:pres.4}
		\end{align}
		As a consequence, from \eqref{lem:pres.4}, using the $\varepsilon$-Young inequality \eqref{ineq:young} with $\psi=\vert\cdot\vert^{\frac{2}{p'(\cdot)}}$ (since $p^-\ge 2$, \textit{i.e.}, $p'(\cdot)\leq 2$ in $\Omega$) and $\varepsilon=1$, we obtain $\smash{\widehat{\bfS}}\in W^{1,p'(\cdot)}(\Omega;\mathbb{R}^{d\times d})$ with
		\begin{align*}
			\left.	\begin{aligned}
				\vert\nabla \smash{\widehat{\bfS}}\vert^{p'(\cdot)}&=\vert\nabla \smash{\widehat{\bfS}}\vert^{p'(\cdot)}(\delta+\vert \bfD\bfv\vert)^{p'(\cdot)\frac{2-p(\cdot)}{2}}(\delta+\vert \bfD\bfv\vert)^{-p'(\cdot)\frac{2-p(\cdot)}{2}}
				\\[-0.5mm]&\lesssim  (\delta+\vert \bfD\bfv\vert)^{2-p(\cdot)}\vert \nabla\smash{\widehat{\bfS}}\vert^2+(\delta+\vert \bfD\bfv\vert)^{p(\cdot)}
			\end{aligned}\right\}\quad\text{ a.e.\ in }\Omega\,.
		\end{align*}
		Eventually, using Problem (\hyperlink{Q}{Q}), that $\bff\hspace{-0.175em} \in\hspace{-0.175em}  L^{p'(\cdot)}(\Omega;\mathbb{R}^d)$,~and that $\bfv\otimes \bfv\hspace{-0.175em} \in\hspace{-0.175em}  W^{1,2}(\Omega;\mathbb{R}^{d\times d})\hspace{-0.175em} \hookrightarrow\hspace{-0.175em}  W^{1,p'(\cdot)}(\Omega;\mathbb{R}^{d\times d})$ (since $p^-\ge 2$, \textit{i.e.}, $p'(\cdot)\leq 2$ in $\Omega$), we conclude that  $q \in
		W^{1, p'(\cdot)}(\Omega)$~with
		\begin{align}\label{lem:pres.5}
				\vert \nabla q\vert \leq \vert \nabla (\smash{\widehat{\bfS}}-\bfv\otimes \bfv)\vert+ \vert\bff\vert\quad\text{ a.e. in }\Omega \,.
		\end{align}
		
		\textit{ad (ii).} Multiplying \eqref{lem:pres.5} with $(\delta+\vert  \bfD\bfv\vert)^{\frac{2-p(\cdot)}{2}}$, exploiting that $p^-\ge 2$ and $\delta>0$, we find that
		\begin{align}\label{lem:pres.6}
			\left.\begin{aligned}
				(\delta+\vert  \bfD\bfv\vert)^{2-p(\cdot)}\vert \nabla q\vert^2 &\lesssim (\delta+\vert  \bfD\bfv\vert)^{2-p(\cdot)}\vert \nabla (\smash{\widehat{\bfS}}-\bfv\otimes \bfv)\vert^2+(\delta+\vert  \bfD\bfv\vert)^{2-p(\cdot)} \vert\bff\vert^2
				\\[-0.5mm]&\lesssim (\delta+\vert  \bfD\bfv\vert)^{2-p(\cdot)}\vert \nabla \smash{\widehat{\bfS}}\vert^2  +\delta^{\smash{2-p(\cdot)}}(\vert \nabla(\bfv\otimes \bfv)\vert^2+ \vert\bff\vert^2)
			\end{aligned}\right\}\quad\text{ a.e. in }\Omega \,.
		\end{align}
		Therefore, using in \eqref{lem:pres.6} that $\bff \in L^2(\Omega;\mathbb{R}^d)$, that $\bfv\otimes \bfv\in W^{1,2}(\Omega;\mathbb{R}^{d\times d})$, and \eqref{lem:pres.4}~together~with~\eqref{lem:pres.1}, we conclude that $	(\delta+\vert  \bfD\bfv\vert)^{2-p(\cdot)}\vert \nabla q\vert^2\in L^1(\Omega)$.
	\end{proof}

	\section{The discrete steady $p(\cdot)$-Navier--Stokes equations}\label{sec:discrete_p-navier-stokes}
	
	\hspace*{5mm}In this section, we introduce the discrete steady $p(\cdot)$-Navier--Stokes equations.\vspace*{-1mm}\enlargethispage{2.5mm}
	
	\subsection{Triangulations}
	
	\hspace*{5mm}Throughout the paper, we denote by $\{\mathcal{T}_h\}_{h>0}$ a family~of~regular (\textit{i.e.}, uniformly shape~\mbox{regular}~and conforming) triangulations of $\Omega\subseteq \mathbb{R}^d$, $d\in\{2,3\}$, consisting of $d$-dimensional~\mbox{simplices}~(\textit{cf}.~\cite{EG21}). 
	Here, $h>0$ refers to the \textit{maximal~mesh-size}, \textit{i.e.}, if we set $h_T\coloneqq  \textup{diam}(T)$ for all $T\in \mathcal{T}_h$,~then~${h 
		= \max_{T\in \mathcal{T}_h}{h_T}
	}$.
	For every $T \in \mathcal{T}_h$,
	we denote by $\rho_T>0$, the supremum of diameters of inscribed balls~contained~in~$T$. We assume that there exists a constant $\omega_0>0$, independent of $h>0$, such that $\max_{T\in \mathcal{T}_h}{h_T}{\rho_T^{-1}}\le
	\omega_0$. The smallest such constant is called the \textit{chunkiness} of $\{\mathcal{T}_h\}_{h>0}$. For every $T\in \mathcal{T}_h$, the corresponding \textit{element patch} is defined by $\omega_T\coloneqq \bigcup\{T'\in\mathcal{T}_h\mid T'\cap T\neq \emptyset\}$.
	
	\subsection{Finite element spaces and projectors}
	
	\hspace{5mm}Given
	$m \in \mathbb N_0$ and $h>0$, we denote by $\mathbb{P}^m(\mathcal{T}_h)$ the space of (possibly discontinuous) scalar functions that are polynomials of degree at most $m$ on each simplex $T\in \mathcal{T}_h$, and set $\mathbb{P}^m_c(\mathcal{T}_h)\coloneqq \mathbb{P}^m(\mathcal{T}_h)\cap C^0(\overline{\Omega})$.
	Then, given $k\in\mathbb{N}$ and $\ell \in \mathbb N_0$, we~denote~by
	\begin{align}
		\begin{aligned}
			V_h&\subseteq \textcolor{black}{(\mathbb{P}^k_c(\mathcal{T}_h))^d}\,, &&\,\Vo_h\coloneqq V_h\cap \Vo\,,\\
			Q_h&\subseteq \mathbb{P}^{\textcolor{black}{\ell}}(\mathcal{T}_h)\,, &&\Qo_h \coloneqq Q_h\cap \Qo\,,
		\end{aligned}
	\end{align}
	appropriate conforming
	finite element spaces such that the following two assumptions are satisfied:
	
		\begin{assumption}[Projection operator $\Pi_h^Q$]
		\label{ass:PiY}
		We assume that $\setR
		\subseteq Q_h$ and that there exists a linear projection operator
		$\Pi_h^Q\colon Q \to Q_h$, which is \textup{locally $L^1$-stable}, \textit{i.e.}, for every $q\in Q$ and $T\in \mathcal{T}_h$,~it~holds~that
		\begin{align}
			\label{eq:PiYstab}
			\langle \vert\Pi_h^Q q\vert\rangle_T  \lesssim  \langle\vert q\vert\rangle_{\omega_T}\,.
		\end{align}
	\end{assumption}
	
	\begin{assumption}[Projection operator $\Pi_h^V$]\label{ass:proj-div}
		We assume that $\mathbb{P}^1_c(\mathcal{T}_h) \subseteq V_h$ and that there
		exists a linear projection operator $\Pi_h^V\colon  V \to V_h$ with the following properties:
		\begin{itemize}[noitemsep,topsep=2pt,leftmargin=!,labelwidth=\widthof{(iii)}]
			\item[(i)] \textup{Preservation of divergence in the $Q_h^*$-sense:} For every $\bfz \in V$ and  $z_h \in Q_h$, it holds that
			\begin{align}
				\label{eq:div_preserving}
				(\divo \bfz,z_h)_{\Omega} &= (\divo\Pi_h^V
				\bfz,z_h)_{\Omega} \,;
			\end{align}
			\item[(ii)] \textup{Preservation of homogeneous Dirichlet boundary values:} $\Pi_h^V(\Vo) \subseteq \Vo_h$;
			\item[(iii)] \textup{Local $L^1$-$W^{1,1}$-stability:} For every $\bfz \in V$ and $T\in \mathcal{T}_h$, it holds that
			\begin{align}
				\label{eq:Pidivcont}
				\langle\vert\Pi_h^V\bfz\vert\rangle_T &\lesssim \langle
				\vert\bfz\vert\rangle_{\omega_T} + h_T\, \langle  \vert\nabla \bfz\vert\rangle_{\omega_T} \,.
			\end{align}
		\end{itemize}
	\end{assumption}

	Next, we present a list of common mixed finite element spaces $\{V_h\}_{h>0}$ and $\{Q_h\}_{h>0}$ with projectors $\{\Pi_h^V\}_{h>0}$ and $\{\Pi_h^Q\}_{h>0}$ on regular grids $\{\mathcal{T}_h\}_{h>0}$ satisfying both  Assumption~\ref{ass:PiY}~and~Assumption~\ref{ass:proj-div}, respectively; for a detailed presentation, we recommend 
	the textbook \cite{BBF13}.
	
	\begin{remark}\label{FEM.Q}
		The following discrete spaces and projectors satisfy Assumption~\ref{ass:PiY}:
		\begin{description}[noitemsep,topsep=2pt,leftmargin=!,labelwidth=\widthof{(iii)},font=\normalfont\itshape]
			\item[(i)] If $Q_h= \mathbb{P}^{\textcolor{black}{\ell}}(\mathcal{T}_h)$ for some $\textcolor{black}{\ell}\ge 0$, then $\Pi_h^Q$ can be chosen as (local) $L^2$-projection operator or, more generally, as a Cl\'ement type quasi-interpolation operator.
			
			\item[(ii)] If \hspace*{-0.1mm}$Q_h\hspace*{-0.1em}=\hspace*{-0.1em}\mathbb{P}^{\textcolor{black}{\ell}}_{\textcolor{black}{c}}(\mathcal{T}_h)$ \hspace*{-0.1mm}for \hspace*{-0.1mm}some \hspace*{-0.1mm}$\textcolor{black}{\ell}\hspace*{-0.1em}\ge \hspace*{-0.1em} 1$, \hspace*{-0.1mm}then \hspace*{-0.1mm}$\Pi_h^Q$ \hspace*{-0.1mm}can \hspace*{-0.1mm}be \hspace*{-0.1mm}chosen \hspace*{-0.1mm}as \hspace*{-0.1mm}a \hspace*{-0.1mm}Cl\'ement \hspace*{-0.1mm}type \hspace*{-0.1mm}quasi-interpolation~\hspace*{-0.1mm}operator.
			
		\end{description}
	\end{remark}
	
	\begin{remark}\label{FEM.V}
		The following discrete spaces  and projectors satisfy Assumption~\ref{ass:proj-div}:
		\begin{description}[noitemsep,topsep=2pt,leftmargin=!,labelwidth=\widthof{(iii)},font=\normalfont\itshape]
			\item[(i)] The \textup{MINI element} for $d\in \{2,3\}$, \textit{i.e.}, $V_h=\textcolor{black}{(\mathbb{P}^1_c(\mathcal{T}_h)\bigoplus\mathbb{B}(\mathcal{T}_h))^d}$, where $\mathbb{B}(\mathcal{T}_h)$ is the bubble function space, and $Q_h=\mathbb{P}^1_c(\mathcal{T}_h)$, introduced in \cite{ABF84} for $d=2$; see also \cite[Chap.\ II.4.1]{GR86} and \cite[Sec.\ 8.4.2, 8.7.1]{BBF13}. An operator \hspace*{-0.1mm}$\Pi_h^V$ \hspace*{-0.1mm}satisfying \hspace*{-0.1mm}Assumption \hspace*{-0.1mm}\ref{ass:proj-div} 
			\hspace*{-0.1mm}is \hspace*{-0.1mm}given \hspace*{-0.1mm}in \hspace*{-0.1mm}\cite[Appx.\ \hspace*{-0.1mm}A.1]{bdr-phi-stokes};~\hspace*{-0.1mm}see~\hspace*{-0.1mm}also~\hspace*{-0.1mm}\mbox{\cite[Lem.~\hspace*{-0.1mm}4.5]{GL01}}.
			
			\item[(ii)] The \textup{Taylor--Hood element} for $d\in\{2,3\}$, \textit{i.e.}, $V_h=\textcolor{black}{(\mathbb{P}^2_c(\mathcal{T}_h))^d}$ and $Q_h=\mathbb{P}^1_c(\mathcal{T}_h)$, introduced in \cite{TH73} for $d=2$; see  also \cite[Chap.\ II.4.2]{GR86}, and its generalizations; see, \textit{e.g.}, \cite[Sec.\ 8.8.2]{BBF13}. An operator $\Pi_h^V$ satisfying Assumption \ref{ass:proj-div} is given in \cite[Thm.\ 3.1, 32]{GS03} or \cite{DST2021}.
			
			\item[(iii)] The \textup{conforming Crouzeix--Raviart element} for $d\hspace*{-0.05em}=\hspace*{-0.05em}2$, \textit{i.e.}, $V_h\hspace*{-0.05em}=\hspace*{-0.05em}\textcolor{black}{(\mathbb{P}^2_c(\mathcal{T}_h)\bigoplus\mathbb{B}(\mathcal{T}_h))^2}$~and~${Q_h\hspace*{-0.05em}=\hspace*{-0.05em}\mathbb{P}^1(\mathcal{T}_h)}$, introduced in \cite{CR73}; see also \cite[Ex.\ 8.6.1]{BBF13}. An operator $\Pi_h^V$ satisfying Assumption~\ref{ass:proj-div}(i) is given in \cite[p.\ 49]{CR73} and it can be shown to satisfy Assumption \ref{ass:proj-div}(ii); see, \textit{e.g.}, \cite[Thm.\ 3.3]{GS03}.
			
			\item[(iv)] The {\textup{first order Bernardi--Raugel element}} for $d\hspace{-0.15em}\in \hspace{-0.15em} \{2,3\}$, \textit{i.e.}, 
			$V_h\hspace{-0.15em}=\hspace{-0.15em}\textcolor{black}{(\mathbb{P}^1_c(\mathcal{T}_h)\bigoplus\mathbb{B}_{\tiny \mathscr{F}}(\mathcal{T}_h))^d}$,~where~$\mathbb{B}_{\tiny \mathscr{F}}(\mathcal{T}_h)$ is the facet bubble function space, and $Q_h=\mathbb{P}^0(\mathcal{T}_h)$, introduced in \cite[Sec. II]{BR85}. For $d=2$~is~often  referred to as \textup{reduced $\mathbb{P}^2$-$\mathbb{P}^0$-element} or as \textup{2D SMALL element}; see,  \textit{e.g.},  \cite[Rem.\ 8.4.2]{BBF13} and \cite[Chap.\ II.2.1]{GR86}.  An operator $\Pi_h^V$ satisfying Assumption \ref{ass:proj-div} is given in
			\cite[Sec.\ II.]{BR85}.
			
			\item[(v)] The \hspace*{-0.15mm}{\textup{second \hspace*{-0.15mm}order \hspace*{-0.15mm}Bernardi--Raugel \hspace*{-0.15mm}element}} \hspace*{-0.15mm}for \hspace*{-0.15mm}$d\hspace*{-0.15em}=\hspace*{-0.15em}3$, \hspace*{-0.15mm}introduced \hspace*{-0.1mm}in \hspace*{-0.15mm}\cite[\hspace*{-0.5mm}Sec.\ \hspace*{-0.5mm}III]{BR85}; \hspace*{-0.1mm}see \hspace*{-0.15mm}also~\hspace*{-0.15mm}\mbox{\cite[\hspace*{-0.5mm}Ex.~\hspace*{-0.5mm}8.7.2]{BBF13}} \hspace*{-0.1mm}and \hspace*{-0.1mm}\cite[Chap.\ \hspace*{-0.1mm}II.2.3]{GR86}. \hspace*{-0.1mm}An \hspace*{-0.1mm}operator \hspace*{-0.1mm}$\Pi_h^V$ \hspace*{-0.1mm}satisfying \hspace*{-0.1mm}Assumption \hspace*{-0.1mm}\ref{ass:proj-div} \hspace*{-0.1mm}is \hspace*{-0.1mm}given~\hspace*{-0.1mm}in~\hspace*{-0.1mm}\mbox{\cite[Sec.~\hspace*{-0.1mm}III.3]{BR85}}.%; see also \cite{tscherpel-phd}.
		\end{description}
	\end{remark}
	
	\subsection{Discrete weak formulations}\vspace*{-0.5mm}
	
	\hspace{5mm}An important aspect in the numerical approximation of the steady $p(\cdot)$-Navier--Stokes equations \eqref{eq:p-navier-stokes} consists in the discretization of the $x$-dependent non-linearity \eqref{def:A}. Here, it is convenient~to~use~a~simple one-point quadrature rule. More precisely, if $p\in C^0(\overline{\Omega})$ with $p^->1$, then we define the element-wise constant power-law index $p_h\in \mathbb{P}^0(\mathcal{T}_h)$, the generalized $N$-function $\varphi_h\colon\Omega\times\mathbb{R}_{\ge 0}\to \mathbb{R}_{\ge 0}$, and the non-linear operators $\bfS_h,\bfF_h,\bfF_h^*\colon\Omega\times\mathbb{R}^{d\times d}\to  \mathbb{R}^{d\times d}_{\textup{sym}}$ for every $\bfA\in \mathbb{R}^{d\times d}$, $T\in  \mathcal{T}_h$,  and a.e.\ $x\in T$  by
	\begin{align}
		\begin{aligned}
			p_h(x)&\coloneqq p(\xi_T)\,,\qquad
			&&\hspace*{-1.75mm}\varphi_h(x,\vert \bfA\vert )\coloneqq \varphi(\xi_T,\vert \bfA\vert)\,,\\
			\bfS_h(x,\bfA)&\coloneqq \bfS(\xi_T,\bfA)\,,\qquad
			&& \bfF_h(x,\bfA)\coloneqq \bfF(\xi_T,\bfA)\,,\qquad
			\bfF_h^*(x,\bfA)\coloneqq \bfF^*(\xi_T,\bfA)\,,
		\end{aligned}\label{def:A_h}
	\end{align}
	where $\xi_T\in T$ is some arbitrary quadrature point, \textit{e.g.}, the barycenter of the element $T$.
	
	\begin{remark}\label{rem:uniform}
		Since the hidden constants in all equivalences in Section \ref{sec:basic} depend only on $p^-,p^+>1$ and $\delta\ge 0$ and since $p^-\leq p_h^-\le p^+_h\leq p^+$ a.e.\ in $\Omega$ for all $h>0$,     
		the same equivalences apply to the discretizations \eqref{def:A_h} with the hidden constants depending~only~on~$p^-,p^+\in (1,\infty)$ and $\delta\ge 0$. 
	\end{remark}
	
	Given the definitions \eqref{def:A_h}, we introduce the discrete counterparts to Problem (\hyperlink{Q}{Q}) and Problem~(\hyperlink{P}{P}), respectively:
	
	\textit{Problem (Q$_h$).}\hypertarget{Qh}{} For given $\bff\in L^{p'(\cdot)}(\Omega;\mathbb{R}^d)$, find $(\bfv_h,q_h)^\top\in \Vo_h\times \Qo_h$ such that for every $(\bfz_h,z_h)^\top\in $\linebreak\hspace*{4.5mm} $\Vo_h\times Q_h $, it holds that\vspace*{-0.5mm}
	\begin{align*}
		(\bfS_h(\cdot,\bfD\bfv_h),\bfD\bfz_h)_{\Omega}+\tfrac{1}{2}(\bfz_h\otimes\bfv_h,\nabla\bfv_h)_{\Omega}-\tfrac{1}{2}(\bfv_h\otimes\bfv_h,\nabla\bfz_h)_{\Omega}-(q_h,\divo\bfz_h)_{\Omega}&=(\bff,\bfz_h)_{\Omega}\,,\\
		(\divo\bfv_h,z_h)_{\Omega}&=0\,.
	\end{align*}
	Equivalently, one can reformulate Problem (\hyperlink{Qh}{Q$_h$}) \textit{``hiding''} the discrete pressure.
	
	\textit{Problem (P$_h$).}\hypertarget{Ph}{} For given $\bff\in L^{p'(\cdot)}(\Omega;\mathbb{R}^d)$, find $\bfv_h\in \textcolor{black}{\Vo_{h,0}}$ such that for every $\bfz_h\in \textcolor{black}{\Vo_{h,0}}$,~it~holds~that
	\begin{align*}
		(\bfS_h(\cdot,\bfD\bfv_h),\bfD\bfz_h)_{\Omega}+\tfrac{1}{2}(\bfz_h\otimes\bfv_h,\nabla\bfv_h)_{\Omega}-\tfrac{1}{2}(\bfv_h\otimes\bfv_h,\nabla\bfz_h)_{\Omega}=(\bff,\bfz_h)_{\Omega}\,,
	\end{align*}
	\hspace*{4.5mm} where $\textcolor{black}{\Vo_{h,0}}\coloneqq \{\bfz_h\in \Vo_h\mid  (\divo\bfz_h,z_h)_{\Omega}=0\text{ for all }z_h\in Q_h\}$.

	The well-posedness of Problem (\hyperlink{Qh}{Q$_h$}) and Problem (\hyperlink{Ph}{P$_h$}) can be established as in the continuous case in two steps:
	first, by using pseudo-monotone operator theory, the well-posedness of Problem~(\hyperlink{Ph}{P$_h$})~is~shown; then, given the well-posedness of Problem~(\hyperlink{Ph}{P$_h$}),
	the well-posedness of (\hyperlink{Qh}{Q$_h$}) follows using the following (variable exponent) discrete inf-sup stability result:

	\begin{lemma}\label{lem:ismd}
		Let $p\in \mathcal{P}^{\log}(\Omega)$ with $p^->1$ and let Assumption \ref{ass:proj-div} be satisfied. Then,~for~every~$z_h\in \Qo_h$, it holds that\vspace*{-0.5mm}
		\begin{align*}
			\|z_h\|_{p'_h(\cdot),\Omega}\lesssim \sup_{\bfz_h\in \Vo_h\,:\,\smash{\|\nabla\bfz_h\|_{p_h(\cdot),\Omega}\leq 1}}{(z_h,\divo\bfz_h)_{\Omega}}\,,
		\end{align*}
		where $\lesssim$ depends only on $d$, $k$, $\textcolor{black}{\ell}$, $p^-$, $p^+$, $c_{\log}(p)$, $\omega_0$, and $\Omega$.
	\end{lemma}
	
	The proof of Lemma \ref{lem:ismd} is a consequence of the continuous inf-sup stability result (\textit{cf}.\  Lemma \ref{lem:inf_sup}) in conjunction with 
	the following discrete norm equivalence.\enlargethispage{2.5mm}
	
	\begin{lemma}\label{lem:norm_equiv}
		Let $p\in \mathcal{P}^{\log}(\Omega)$ with $p^->1$ and $n\in \mathbb{N}\cup\{0\}$. Then, 
		for every $z_h\in \mathbb{P}^n(\mathcal{T}_h)$,~it~holds~that $\|z_h\|_{p_h(\cdot),\Omega}\sim  \|z_h\|_{p(\cdot),\Omega}$, 
		where $\sim$ depends on only on $d$, $n$, $p^-$, $p^+$, $c_{\log}(p)$, and $\omega_0$.
	\end{lemma}
	
	\begin{proof}
		The proof is analogous to \cite[Lem.\  3.1]{BBD15}, where $\xi_T \coloneqq \textrm{arg\,min}_{x\in T}{p(x)}$ for all $T\in \mathcal{T}_h$.
		
		First, \hspace*{-0.1mm}assume \hspace*{-0.1mm}that \hspace*{-0.1mm}$\|z_h\|_{p_h(\cdot),\Omega}\leq 1$, \hspace*{-0.1mm}which, \hspace*{-0.1mm}by \hspace*{-0.1mm}the \hspace*{-0.1mm}norm-modular \hspace*{-0.1mm}unit \hspace*{-0.1mm}ball \hspace*{-0.1mm}property~\hspace*{-0.1mm}(\textit{cf}.~\hspace*{-0.1mm}\mbox{\cite[\hspace*{-0.1mm}Lem.~\hspace*{-0.1mm}3.2.4]{dhhr}}), implies that $\rho_{p_h(\cdot),\Omega}(z_h)\leq 1$. Then, using  discrete local norm equivalences (\textit{cf}.\  \cite[Lem.\  12.1]{EG21}) and  Jensen's inequality, for every $T\in \mathcal{T}_h$, 
		due to $\vert T\vert^{-1/p(\xi_T)}\leq \vert T\vert^{-1/p^-}$ (since $h_T\leq 1$), it holds that
		\begin{align*}
			\smash{\|z_h\|_{\infty,T}\lesssim \langle\vert z_h\vert\rangle_{T}\leq \langle\vert z_h\vert^{p(\xi_T )}\rangle_{T}^{1/p(\xi_T )}
				=\vert T\vert ^{-1/p(\xi_T )}\rho_{p_h(\cdot),T}(z_h)^{1/p(\xi_T )}
				\leq \vert T\vert ^{-1/p(\xi_T )}\leq \vert T\vert ^{-1/p^-}\,,}
		\end{align*}
		where $\lesssim$ depends only on $n$ and $\omega_0$.
		By Lemma \ref{lem:local_change}, for every $T\in \mathcal{T}_h$, $x\in T$,~and~${t\in [\vert T\vert^{1/p^-}, \vert T\vert^{-1/p^-}]}$, it holds that $t^{p(x)-p(\xi_T )}\leq c$, where
		$c>0$ depends only on $d$, $n$, $p^-$, $p^+$, $c_{\log}(p)$, and $\omega_0$,~so~that choosing $t=1+\vert z_h\vert $, we find that 
		\begin{align*}
			\smash{\rho_{p(\cdot),\Omega}(z_h)\leq \rho_{p(\cdot),\Omega}(1+\vert z_h\vert )\lesssim  \rho_{p_h(\cdot),\Omega}(1+\vert z_h\vert )\lesssim 1\,,}
		\end{align*}
		which implies that $\|z_h\|_{p(\cdot),\Omega}\lesssim 1$.
		
		Second, \hspace*{-0.1mm}assume \hspace*{-0.1mm}that \hspace*{-0.1mm}$\|z_h\|_{p(\cdot),\Omega}\hspace*{-0.1em}\leq\hspace*{-0.1em} 1$, \hspace*{-0.1mm}which, \hspace*{-0.1mm}by \hspace*{-0.1mm}the \hspace*{-0.1mm}norm-modular \hspace*{-0.1mm}unit \hspace*{-0.1mm}ball \hspace*{-0.1mm}property \hspace*{-0.1mm}(\textit{cf}.~\hspace*{-0.1mm}\mbox{\cite[\hspace*{-0.5mm}Lem.~\hspace*{-0.5mm}3.24.]{dhhr}}), implies that $\rho_{p(\cdot),\Omega}(z_h)\leq 1$. 
		As before, 
		for every $T\in \mathcal{T}_h$, but now setting $p_T^-\coloneqq\textrm{min}_{x\in T}{p(x)}$,~it~holds~that
		\begin{align*}
			\smash{\|z_h\|_{\infty,T}\lesssim \langle\vert z_h\vert\rangle_{T}\leq \langle\vert z_h\vert^{p_T^-}\rangle_{T}^{1/p_T^-}
				\lesssim \vert T\vert ^{-1/p_T^-}\rho_{p(\cdot),\Omega}(1+\vert z_h\vert )^{1/p_T^-}
				\lesssim  \vert T\vert ^{-1/p_T^-}\leq \vert T\vert ^{-1/p^-}\,.}
		\end{align*}
		By Lemma \ref{lem:local_change}, for every $T\in \mathcal{T}_h$, $x\in  T$, and $t\in [\vert T\vert^{1/p^-}, \vert T\vert^{-1/p^-}]$, it~holds~that~${t^{p(\xi_T)-p(x )}\leq c}$, where
		$c>0$ depends only on $d$, $n$, $p^-$, $p^+$, $c_{\log}(p)$, and $\omega_0$,~so~that~choosing $t=1+\vert z_h\vert $, we find that 
		\begin{align*}
			\smash{ \rho_{p_h(\cdot),\Omega}(z_h)\leq \rho_{p_h(\cdot),\Omega}(1+\vert z_h\vert )\lesssim \rho_{p(\cdot),\Omega}(1+\vert z_h\vert )\lesssim 1\,,}
		\end{align*}
		which implies that $\|z_h\|_{p_h(\cdot),\Omega}\lesssim 1$.\vspace*{-0.5mm}
	\end{proof}
	
	\begin{proof}[Proof (of Lemma \ref{lem:ismd}).] We proceed as in the proof of
		\cite[Lem.\  5.2]{BBD15},  by using Lemma \ref{lem:norm_equiv} instead of \cite[Lem.\  3.1]{BBD15}.\vspace*{-0.5mm}
	\end{proof}
	
	In addition, testing Problem (\hyperlink{Qh}{Q$_h$}) with $(\bfz_h,z_h)^\top\coloneqq (\bfv_h,q_h)^\top\in \Vo_h\times \Qo_h$, resorting to the growth conditions of the extra-stress tensor (\textit{cf}.\  Section \ref{sec:basic}) and the discrete inf-sup stability~result~(\textit{cf}.~Lemma~\ref{lem:ismd}), 
	one readily finds a constant $c>0$, depending only on $d$, $k$, $\textcolor{black}{\ell}$, $p^-$, $p^+$, $c_{\log}(p)$, $\omega_0$, and $\Omega$, such that\vspace*{-0.5mm}\enlargethispage{10mm}
	\begin{align}
		\smash{\|q_h\|_{p_h'(\cdot),\Omega}+ \|\bfD\bfv_h\|_{p_h(\cdot),\Omega}\leq c\,.}\label{eq:stability}
	\end{align}
	
	\subsection{Stability estimates for the projectors}\vspace*{-0.5mm}
	
	\hspace{5mm}The following stability result for $\Pi_h^Q$ (\textit{cf}.\  Assumption \ref{ass:PiY}) in terms of  modulars with respect to  conjugate shifted $N$-functions applies.\vspace*{-0.5mm}

	\begin{lemma}\label{lem:stab_Pi_Q}
		Let $p\in \mathcal{P}^{\log}(\Omega)$ and $c_0>0$. Then, for every $m\in \mathbb{N}$, 
		$T\in \mathcal{T}_h$, $z\in \smash{L^{p'(\cdot)}(\omega_T)}$, and $a\ge 0$ with $a+\langle\vert z\vert\rangle_{\omega_T}\leq c_0\,\vert T\vert ^{-m}$, it holds that\vspace*{-1mm}
		\begin{align}\label{lem:stab_Pi_Q.local}
			\smash{\rho_{(\varphi_a)^*,T}(\Pi_h^Q z)\lesssim h_T^{m}+\rho_{(\varphi_a)^*,\omega_T}(z)\,,}
		\end{align}
		where $\lesssim$ depends only on $d$, $\textcolor{black}{\ell}$, $m$, $p^+$, $p^-$, $c_{\log}(p)$,  $\omega_0$, and $c_0$.
		In addition, for every $m\in \mathbb{N}$, $z\in \smash{L^{p'(\cdot)}(\Omega)}$, and $a\ge 0$ with $a+\|z\|_{1,\Omega}\leq c_0$,~it~holds that\vspace*{-0.5mm}
		\begin{align}\label{lem:stab_Pi_Q.global}
			\smash{ \rho_{(\varphi_a)^*,\Omega}(\Pi_h^Q z)\lesssim h^m
				+\rho_{(\varphi_a)^*,\Omega}(z)\,,}
		\end{align}
		where  $\lesssim$ depends only on $d$, $\textcolor{black}{\ell}$, $m$, $p^+$, $p^-$, $c_{\log}(p)$,  $\omega_0$, and $c_0$.
	\end{lemma}
	
	\begin{proof}
		\textit{ad \eqref{lem:stab_Pi_Q.local}.} Follows analogously to the proof of \cite[Lem.\  3.4]{BDS15}, by using Lemma \ref{lem:key-estimate} instead of \cite[Thm.\ 2.4]{BDS15}.
		
		\textit{ad \eqref{lem:stab_Pi_Q.global}.} If $a+\|z\|_{1,\Omega}\hspace*{-0.1em}\leq\hspace*{-0.1em} c_0$, then, using that, owing to  $h_T\leq 1$, it holds that $\vert T\vert ^{-1}\leq \vert T\vert^{-(d+m)}$, we can find a constant $c>0$, depending on  $\omega_0$~and~$c_0$, such that for every $T\in \mathcal{T}_h$, we have that\vspace*{-0.5mm}
		\begin{align}\label{lem:stab_Pi_Q.0}
			a+\langle\vert z\vert\rangle_{\omega_T}\leq c\,\vert T\vert^{-1}\leq c\,\vert T\vert ^{-(d+m)}\,.
		\end{align}
		Due to \eqref{lem:stab_Pi_Q.0}, resorting to \eqref{lem:stab_Pi_Q.local}, there exists a constant $c>0$, depending on $d$, $\textcolor{black}{\ell}$, $m$, $p^+$, $p^-$, $c_{\log}(p)$,~$\omega_0$, and $c_0$, such that for every $T\in \mathcal{T}_h$, it holds that\vspace*{-0.5mm}
		\begin{align}\label{lem:stab_Pi_Q.1}
			\rho_{(\varphi_a)^*,T}(\Pi_h^Q z)\leq c\,(h_T^m\,\vert T\vert +\rho_{(\varphi_a)^*,\omega_T}(z))\,.
		\end{align}
		Then, summation of \eqref{lem:stab_Pi_Q.1} with respect to $T\in \mathcal{T}_h$ yields 
		the claimed global stability~estimate~\eqref{lem:stab_Pi_Q.global}.
	\end{proof}
	
	As a consequence of Lemma \ref{lem:stab_Pi_Q}, we derive the following $Q$-stability result for $\Pi_h^Q$ (\textit{cf}.~Assumption~\ref{ass:PiY}) in terms of the Luxembourg norm, \textcolor{black}{which has already been proved in \cite[Prop.\ 3.9]{KPS18} for the case of Hölder continuous exponent.}\enlargethispage{5mm}
	
	\begin{lemma}\label{lem:stab_Pi_Q_norm}
		Let $p\in \mathcal{P}^{\log}(\Omega)$. Then, for every $T\in \mathcal{T}_h$ and  $z\in \smash{L^{p'(\cdot)}(\omega_T)}$, it holds that
		\begin{align}\label{lem:stab_Pi_Q_norm.1}
			\smash{\|\Pi_h^Qz\|_{p'(\cdot),T}\lesssim \|z\|_{p'(\cdot),\omega_T} \,,}
		\end{align}
		where $\lesssim $ depends only on $d$, $\textcolor{black}{\ell}$, $p^-$, $p^+$, $c_{\log}(p)$,~and~$\omega_0$. 
		In addition, for every $z\in\smash{ L^{p'(\cdot)}(\Omega)}$,~it~holds~that
		\begin{align}\label{lem:stab_Pi_Q_norm.2}
			\smash{\|\Pi_h^Qz \|_{p'(\cdot),\Omega}\lesssim \|z\|_{p'(\cdot),\Omega} \,.}
		\end{align}
	\end{lemma}
	
	\begin{proof}\let\qed\relax
		\textit{ad \eqref{lem:stab_Pi_Q_norm.1}.}
		First, let $z\hspace*{-0.1em}\in\hspace*{-0.1em} L^{p'(\cdot)}(\omega_T)$ be arbitrary with $\|z\|_{p'(\cdot),\omega_T}\hspace*{-0.1em}\leq\hspace*{-0.1em} 1$ or equivalently~${\rho_{p'(\cdot),\omega_T}(z)\hspace*{-0.1em}\leq\hspace*{-0.1em} 1}$. Then, 
		let $\lambda\hspace*{-0.15em}>\hspace*{-0.15em}0$ be such that $\rho_{p'(\cdot),\omega_T}(z/\lambda)\hspace*{-0.15em}\leq\hspace*{-0.15em} 1$. As ${\langle \vert z/\lambda\vert \rangle_{\omega_T}\hspace*{-0.15em}\leq \hspace*{-0.15em}\vert \omega_T\vert^{-1}2^{p^+}\hspace*{-0.1em}(\vert \omega_T\vert \hspace*{-0.15em}+\hspace*{-0.15em}\rho_{p'(\cdot),\omega_T}( z/\lambda))\hspace*{-0.15em}\leq\hspace*{-0.15em} c_0\,\vert T\vert^{-1}}$,
		where \hspace*{-0.1mm}$c_0>0$ \hspace*{-0.1mm}depends \hspace*{-0.1mm}only \hspace*{-0.1mm}on \hspace*{-0.1mm}$p^+\hspace*{-0.1em}$ \hspace*{-0.1mm}and \hspace*{-0.1mm}$\omega_0$,
		\hspace*{-0.1mm}due \hspace*{-0.1mm}to \hspace*{-0.1mm}Lemma \hspace*{-0.1mm}\ref{lem:stab_Pi_Q}\eqref{lem:stab_Pi_Q.local} \hspace*{-0.1mm}(with \hspace*{-0.1mm}$a\hspace*{-0.1em}=\hspace*{-0.1em}\delta\hspace*{-0.1em}=\hspace*{-0.1em}0$ \hspace*{-0.1mm}and \hspace*{-0.1mm}${m\hspace*{-0.1em}=\hspace*{-0.1em}1}$),~\hspace*{-0.1mm}it~\hspace*{-0.1mm}holds~\hspace*{-0.1mm}that
		\begin{align}\label{lem:stab_Pi_Q_norm.3}
			\smash{\rho_{p'(\cdot),T}(\Pi_h^Q (z/\lambda))\leq c\,(h_T+\rho_{p'(\cdot),\omega_T}(z/\lambda))\leq c\,,}
		\end{align}
		where $c\hspace*{-0.1em}>\hspace*{-0.1em}1$ depends on $d$, $\textcolor{black}{\ell}$, $p^-$, $p^+$, $c_{\log}(p)$, $\omega_0$, and $c_0$. From \eqref{lem:stab_Pi_Q_norm.3}, we get ${\rho_{p'(\cdot),T}(\Pi_h^Q z/(c^{1/(p^-)'}\lambda))\hspace*{-0.1em}\leq\hspace*{-0.1em} 1}$, which, by the definition of the Luxembourg norm, implies that
		\begin{align}\label{lem:stab_Pi_Q_norm.4}
			\smash{	\|\Pi_h^Qz \|_{p'(\cdot),T}\leq  c^{1/(p^-)'}\,\lambda\,.}
		\end{align}
		Taking in \eqref{lem:stab_Pi_Q_norm.4} the infimum with respect to all $\lambda>0$ such that $\rho_{p'(\cdot),\omega_T}(z/\lambda)\leq  1$, by the definition of the Luxembourg norm, yields that\vspace*{-1mm}
		\begin{align}\label{lem:stab_Pi_Q_norm.5}
			\smash{	\|\Pi_h^Qz \|_{p'(\cdot),T}\leq  c^{1/(p^-)'}\,\|z\|_{p'(\cdot),\omega_T}\,.}
		\end{align}
		Since both sides in \eqref{lem:stab_Pi_Q_norm.5} are homogeneous with respect to scalars, with a standard scaling argument, we conclude that \eqref{lem:stab_Pi_Q_norm.5} applies for all  $z\in L^{p'(\cdot)}(\omega_T)$.
		
		\textit{ad \eqref{lem:stab_Pi_Q_norm.2}.} 	Using \cite[Cor.\  7.3.21]{dhhr}, \eqref{lem:stab_Pi_Q_norm.1}, that $\|1\|_{p'(\cdot),\omega_T}\lesssim \|1\|_{p'(\cdot),T}$, and, again, \cite[Cor.~7.3.21]{dhhr}, we find that\vspace*{-1mm}
		\begin{align*}
			\| \Pi_h^Q z\|_{p'(\cdot),\Omega}%&=\bigg\|\sum_{T\in \mathcal{T}_h}{\chi_T \Pi_h^Q z}\bigg\|_{p(\cdot),\Omega}
			%\\
			%&
			\lesssim \bigg\|\sum_{T\in \mathcal{T}_h}{ \chi_T\frac{\| \Pi_h^Q z\|_{p'(\cdot),T}}{\|1\|_{p'(\cdot),T}}}\bigg\|_{p'(\cdot),\mathbb{R}^d}
			%\\&
			\lesssim \bigg\|\sum_{T\in \mathcal{T}_h}{\chi_{\omega_T} \frac{\|z\|_{p'(\cdot),\omega_T}}{\|1\|_{p'(\cdot),\omega_T}}}\bigg\|_{p'(\cdot),\mathbb{R}^d}
			%\\&
			\lesssim\|z\|_{p'(\cdot),\Omega}\,.\tag*{$\qedsymbol$}
		\end{align*}
	\end{proof}
	
	The following stability result for $\Pi_h^V$ (\textit{cf}.\  Assumption \ref{ass:proj-div}), in terms of  modulars with respect to  shifted $N$-functions, applies.\enlargethispage{2.5mm}

	\begin{lemma}\label{lem:stab_Pi_div_V}
		Let $p\in \mathcal{P}^{\log}(\Omega)$  and $c_0>0$. Then, for every $m\in \mathbb{N}$,   $T\in \mathcal{T}_h$, $\bfz\in \smash{W^{1,p(\cdot)}(\omega_T;\mathbb{R}^d)}$, and $a\ge  0$ with $a+\langle\vert \nabla\bfz\vert\rangle_{\omega_T}\leq  c_0\,\vert T\vert ^{-m}$, it holds that
		\begin{align}\label{lem:stab_Pi_div_V.local}
			\smash{\rho_{\varphi_a,T}(\nabla \Pi_h^V \bfz)\lesssim h_T^m+\rho_{\varphi_a,\omega_T}(\nabla \bfz)\,,}
		\end{align}
		where $\lesssim$ depends only on $d$, $k$, $m$, $p^+$,~$c_{\log}(p)$,  $\omega_0$, and $c_0$.
		In addition, for every $m\hspace*{-0.1em}\in \hspace*{-0.1em} \mathbb{N}$,~${\bfz\hspace*{-0.1em}\in\hspace*{-0.1em} \smash{W^{1,p(\cdot)}(\Omega;\mathbb{R}^d)}}$, and $a\ge 0$ with $a+\|\nabla \bfz\|_{1,\Omega}\leq c_0$, it holds that
		\begin{align}\label{lem:stab_Pi_div_V.global}
			\smash{\rho_{\varphi_a,\Omega}(\nabla \Pi_h^V \bfz)\lesssim h^m+\rho_{\varphi_a,\Omega}(\nabla \bfz)\,,}
		\end{align}
		where $\lesssim$ depends only on $d$, $k$, $m$, $p^+$,~$c_{\log}(p)$,  $\omega_0$, and $c_0$.
	\end{lemma}

	\begin{proof}
		\textit{ad \eqref{lem:stab_Pi_div_V.local}.} See \cite[Thm.\  4.2 a)]{BBD15}.
		
		\textit{ad \eqref{lem:stab_Pi_div_V.global}.} We proceed as in the proof of Lemma \ref{lem:stab_Pi_Q}\eqref{lem:stab_Pi_Q.global} up to minor adjustments. 
	\end{proof}

	As a consequence of Lemma \ref{lem:stab_Pi_div_V}, we derive the following $V$-stability result for $\Pi_h^V$ (\textit{cf}.~\mbox{Assumption}~\ref{ass:proj-div}) in terms of the Luxembourg norm.
	
	\begin{lemma}\label{lem:stab_Pi_div_V_norm}
		Let $p\hspace*{-0.1em}\in\hspace*{-0.1em} \mathcal{P}^{\log}(\Omega)$ with $p^-\hspace*{-0.1em}>\hspace*{-0.1em}1$. Then, for every $T\hspace*{-0.1em}\in\hspace*{-0.1em} \mathcal{T}_h$ and  $\bfz\hspace*{-0.1em}\in\hspace*{-0.1em} \smash{W^{1,p(\cdot)}(\omega_T;\mathbb{R}^d)}$,~it~holds~that
		\begin{align}\label{lem:stab_Pi_div_V_norm.1}
			\smash{	\|\nabla \Pi_h^V \bfz\|_{p(\cdot),T}\lesssim \|\nabla \bfz\|_{p(\cdot),\omega_T} \,,}
		\end{align}
		where \hspace*{-0.15mm}$\lesssim $ \hspace*{-0.15mm}depends \hspace*{-0.15mm}only \hspace*{-0.15mm}on \hspace*{-0.15mm}$d$, \hspace*{-0.15mm}$k$, \hspace*{-0.15mm}$p^-$\hspace*{-0.15mm}, \hspace*{-0.15mm}$p^+$\hspace*{-0.15mm}, \hspace*{-0.15mm}$c_{\log}(p)$,~\hspace*{-0.15mm}and~\hspace*{-0.15mm}$\omega_0$. 
		\hspace*{-0.15mm}In \hspace*{-0.15mm}addition, \hspace*{-0.15mm}for \hspace*{-0.15mm}every \hspace*{-0.15mm}$\bfz\hspace*{-0.15em}\in\hspace*{-0.15em} \smash{W^{1,p(\cdot)}(\Omega;\mathbb{R}^d)}$,~\hspace*{-0.15mm}it~\hspace*{-0.15mm}holds~\hspace*{-0.15mm}that
		\begin{align}\label{lem:stab_Pi_div_V_norm.2}
			\smash{	\|\nabla \Pi_h^V \bfz\|_{p(\cdot),\Omega}\lesssim \|\nabla \bfz\|_{p(\cdot),\Omega} \,.}
		\end{align}
	\end{lemma}
	
	\begin{proof}
		We proceed as in the proof of Lemma \ref{lem:stab_Pi_Q_norm}, using Lemma \ref{lem:stab_Pi_div_V} instead of Lemma \ref{lem:stab_Pi_Q}.
	\end{proof}

	\section{Fractional interpolation error estimates for the velocity and the pressure}\label{sec:fractional_interpolation_estimates}
	
	\hspace{5mm}In this subsection, we derive fractional interpolation error estimates for $\Pi_h^V$ (\textit{cf}.\ Assumption~\ref{ass:proj-div})~and $\Pi_h^Q$ (\textit{cf}.\ Assumption \ref{ass:PiY}), which are the source of fractional error decay rates in dependence of the respective fractional regularity assumptions on the velocity  and the pressure. 
	
	Given fractional regularity of the velocity expressed in Nikolski\u{\i} spaces, we have the following~fractio-nal interpolation error estimate for $\Pi_h^V$  (\textit{cf}.\ Assumption~\ref{ass:proj-div}) measured in the discrete~natural~distance.
	
	\begin{lemma}\label{lem:Pi_div_F}
		Let $p\in C^{0,\alpha}(\overline{\Omega})$ with $p^->1$ and $\alpha\in (0,1]$ and let $\bfz\in \smash{W^{1,p(\cdot)}_0(\Omega;\mathbb{R}^d)}$ be such that
		$\bfF(\cdot,\bfD\bfz)\in \smash{N^{\beta,2}(\Omega;\mathbb{R}^{d\times d})}$ with $\beta\in (0,1]$. Then,  there exists a constant $s>1$, which can be chosen to be close to $1$ if $h_T>0$ is close to $0$, such that
		for every $T\in \mathcal{T}_h$, it holds that
		\begin{align}
			\|\bfF_h(\cdot,\bfD\bfz)-\bfF_h(\cdot,\bfD\Pi_h^V\bfz)\|_{2,T}^2\lesssim 
			h_T^{2\alpha}\,\|1+\vert \bfD\bfz\vert^{p(\cdot)s}\|_{1,\omega_T}   
			+ h_T^{2\beta}\,[ \bfF(\cdot,\bfD\bfz)]_{N^{\beta,2}(\omega_T)}^2\,,\label{lem:Pi_div_F_1}
		\end{align}
		where $\lesssim$ depends only on $d$, $k$, $p^-$, $p^+$, $[p]_{\alpha,\Omega}$, $s$, $\omega_0$, and $\|\bfD\bfz\|_{p(\cdot),\Omega}$.
		In particular, it holds that
		\begin{align}
			\|\bfF_h(\cdot,\bfD\bfz)-\bfF_h(\cdot,\bfD\Pi_h^V\bfz)\|_{2,\Omega}^2\lesssim  h^{2\alpha}\,\big(1+\rho_{p(\cdot)s,\Omega}(\bfD\bfz)\big)+
			h^{2\beta}\,[ \bfF(\cdot,\bfD\bfz)]_{N^{\beta,2}(\Omega)}^2\,.\label{lem:Pi_div_F_2}
		\end{align}
	\end{lemma}
	
	\begin{proof}
		Appealing to \cite[Thm.\  4.3]{BBD15}, there exists a constant $s>1$, which can be chosen to be close to $1$ if $h_T>0$ is close to $0$, such that for every $T\in \mathcal{T}_h$, we have that 
		\begin{align}\label{eq:Pi_div_F.1}
			\begin{aligned}
				\|\bfF_h(\cdot,\bfD\bfz)-\bfF_h(\cdot,\bfD\Pi_h^V\bfz)\|_{2,T}^2&\lesssim h_T^{2\alpha}\,\|1+\vert \bfD\bfz\vert^{p(\cdot)s}\|_{1,\omega_T} \\&\quad +\|\bfF(\cdot,\bfD\bfz)-\langle\bfF(\cdot,\bfD\bfz)\rangle_{\omega_T}\|_{2,\omega_T}^2\,,
			\end{aligned}
		\end{align}
		where \hspace*{-0.1mm}$\lesssim$ \hspace*{-0.1mm}depends \hspace*{-0.1mm}only \hspace*{-0.1mm}on \hspace*{-0.1mm}$d$, \hspace*{-0.1mm}$k$, \hspace*{-0.1mm}$p^-$\hspace*{-0.1mm}, \hspace*{-0.1mm}$p^+$\hspace*{-0.1mm}, \hspace*{-0.1mm}$[p]_{\alpha,\Omega}$, \hspace*{-0.1mm}$s$, \hspace*{-0.1mm}$\omega_0$, and $\|\bfD\bfz\|_{p(\cdot),\Omega}$.
		\hspace*{-0.5mm}Using \hspace*{-0.1mm}\cite[\hspace*{-0.1mm}(4.6), (4.7)]{breit-lars-etal},~\hspace*{-0.1mm}for~\hspace*{-0.1mm}\mbox{every}~\hspace*{-0.1mm}${T\hspace*{-0.15em}\in \hspace*{-0.15em} \mathcal{T}_h}$, we find that\vspace*{-1mm}
		\begin{align}\label{eq:Pi_div_F.2}
			\|\bfF(\cdot,\bfD\bfz)-\langle\bfF(\cdot,\bfD\bfz)\rangle_{\omega_T}\|_{2,\omega_T}^2 
			\lesssim h_T^{2\beta} \,[\bfF(\cdot,\bfD\bfz)]_{N^{\beta,2}(\omega_T)}^2\,.
		\end{align}
		Eventually, combining \eqref{eq:Pi_div_F.1} and \eqref{eq:Pi_div_F.2}, \hspace*{-0.1mm}we \hspace*{-0.1mm}arrive \hspace*{-0.1mm}at \hspace*{-0.1mm}the \hspace*{-0.1mm}first \hspace*{-0.1mm}claimed \hspace*{-0.1mm}interpolation \hspace*{-0.1mm}error \hspace*{-0.1mm}estimate~\hspace*{-0.1mm}\eqref{lem:Pi_div_F_1}. \hspace*{-0.1mm}The \hspace*{-0.1mm}second \hspace*{-0.1mm}claimed \hspace*{-0.1mm}interpolation \hspace*{-0.1mm}error \hspace*{-0.1mm}estimate \hspace*{-0.1mm}\eqref{lem:Pi_div_F_2} \hspace*{-0.1mm}follows \hspace*{-0.1mm}via \hspace*{-0.1mm}summation \hspace*{-0.1mm}as \hspace*{-0.1mm}in~\hspace*{-0.1mm}\mbox{\cite[Thm.~\hspace*{-0.1mm}5]{breit-lars-etal}}.
	\end{proof}
	
	Resorting to Proposition \ref{lem:A-Ah}, we can derive an analogue of \textcolor{black}{Lemma \ref{lem:Pi_div_F}} for $\bfF\colon \Omega\times \mathbb{R}^{d\times d}\to \smash{\mathbb{R}^{d\times d}_{\textup{sym}}}$ instead of $\bfF_h\colon \Omega\times \mathbb{R}^{d\times d}\to \smash{\mathbb{R}^{d\times d}_{\textup{sym}}}$, $h\in (0,1]$.\enlargethispage{5mm}
	
	\begin{lemma}\label{lem:Pi_div_F.2}
		Let $p\in C^{0,\alpha}(\overline{\Omega})$ with $p^->1$ and $\alpha\in (0,1]$ and let $\bfz\in \smash{W^{1,p(\cdot)}_0(\Omega;\mathbb{R}^d)}$ be such that
		$\bfF(\cdot,\bfD\bfz)\in N^{\beta,2}(\Omega;\mathbb{R}^{d\times d})$ with $\beta\in (0,1]$. Then,  there exists a constant $s>1$, which can be chosen to be close to $1$ if $h_T>0$ is close to $0$, such that for every $T\in \mathcal{T}_h$, it holds that
		\begin{align}\label{lem:Pi_div_F.2_1}
			\|\bfF(\cdot,\bfD\bfz)-\bfF(\cdot,\bfD\Pi_h^V\bfz)\|_{2,T}^2\lesssim  h_T^{2\alpha}\,\|1+\vert \bfD\bfz\vert^{p(\cdot)s}\|_{1,\omega_T}
			+	h_T^{2\beta}\,[ \bfF(\cdot,\bfD\bfz)]_{N^{\beta,2}(\omega_T)}^2\,,
		\end{align}
		where $\lesssim$ depends only on $d$, $k$, $p^-$, $p^+$, $[p]_{\alpha,\Omega}$, $s$, $\omega_0$, and $\|\bfD\bfz\|_{p(\cdot),\Omega}$.
		In particular, it holds that
		\begin{align}\label{lem:Pi_div_F.2_2}
			\|\bfF(\cdot,\bfD\bfz)-\bfF(\cdot,\bfD\Pi_h^V\bfz)\|_{2,\Omega}^2\lesssim h^{2\alpha}\,\big(1+\rho_{p(\cdot)s,\Omega}(\bfD\bfz)\big)
			+h^{2\beta}\,[ \bfF(\cdot,\bfD\bfz)]_{N^{\beta,2}(\Omega)}^2\,.
		\end{align}
	\end{lemma}
	
	\begin{proof}
		Appealing to Proposition \ref{lem:A-Ah}\eqref{eq:Fh-F} and Lemma \ref{lem:stab_Pi_div_V}\eqref{lem:stab_Pi_div_V.local} (with $m=d$), there exists  a constant $s>1$, which can chosen to be close to $1$ if $h_T>0$ is close to $0$, such that for every $T\in \mathcal{T}_h$, we have that
		\begin{align}\label{eq:Pi_div_F.2.1}
			\begin{aligned}
				\|\bfF_h(\cdot,\bfD\bfz)-\bfF(\cdot,\bfD\bfz)\|_{2,T}^2&\lesssim h_T^{2\alpha}\,\|1+\vert \bfD\bfz\vert^{p(\cdot)s}\|_{1,T}\,;\\%[1.5mm]
				\|\bfF_h(\cdot,\bfD\Pi_h^V\bfz)-\bfF(\cdot,\bfD\Pi_h^V\bfz)\|_{2,T}^2&\lesssim h_T^{2\alpha}\,\|1+\vert \bfD\Pi_h^V\bfz\vert^{p(\cdot)s}\|_{1,T}\lesssim  h_T^{2\alpha}\,\|1+\vert \bfD\bfz\vert^{p(\cdot)s}\|_{1,\omega_T}\,,\\
			\end{aligned}
		\end{align}
		where $\lesssim$ depends only on $d$, $k$, $p^-$, $p^+$, $[p]_{\alpha,\Omega}$, $s$, $\omega_0$, and $\|\bfD\bfz\|_{p(\cdot),\Omega}$.
		Using Lemma \ref{lem:Pi_div_F}\eqref{lem:Pi_div_F_1} and \eqref{eq:Pi_div_F.2.1}, we~conclude~that
		\begin{align*}
			\|\bfF(\cdot,\bfD\bfz)-\bfF(\cdot,\bfD\Pi_h^V\bfz)\|_{2,T}^2&\lesssim \|\bfF_h(\cdot,\bfD\bfz)-\bfF(\cdot,\bfD\bfz)\|_{2,T}^2
			+\|\bfF_h(\cdot,\bfD\bfz)-\bfF_h(\cdot,\bfD\Pi_h^V\bfz)\|_{2,T}^2
			\\&\quad +\|\bfF_h(\cdot,\bfD\Pi_h^V\bfz)-\bfF(\cdot,\bfD\Pi_h^V\bfz)\|_{2,T}^2
			\\& \lesssim h_T^{2\alpha}\,\|1+\vert \bfD\bfz\vert^{p(\cdot)s}\|_{1,\omega_T}
			+ h_T^{2\beta}\,[\bfF(\cdot,\bfD\bfz)]_{N^{\beta,2}(\omega_T)}^2\,,
		\end{align*}
		which is the first claimed interpolation error estimate \eqref{lem:Pi_div_F.2_1}. \hspace*{-0.1mm}The \hspace*{-0.1mm}second \hspace*{-0.1mm}claimed \hspace*{-0.1mm}interpolation \hspace*{-0.1mm}error \hspace*{-0.1mm}estimate \hspace*{-0.1mm}\eqref{lem:Pi_div_F.2_2} \hspace*{-0.1mm}follows \hspace*{-0.1mm}via \hspace*{-0.1mm}summation \hspace*{-0.1mm}as \hspace*{-0.1mm}in \hspace*{-0.1mm}\mbox{\cite[Thm.~\hspace*{-0.1mm}5]{breit-lars-etal}}.
	\end{proof}
	
	While Lemma \ref{lem:Pi_div_F} and Lemma \ref{lem:Pi_div_F.2} play an important role in the extraction of fractional error decay rates  from higher-order terms related to the extra-stress tensor, the following corollary 
	aids to  extract fractional  error decay rates from higher-order terms related to the convective term.\enlargethispage{6mm}
	
	\begin{corollary}\label{cor:Pi_div_F}
		Let $p\in C^{0,\alpha}(\overline{\Omega})$ with $p^->1$ and $\alpha\in (0,1]$, $r\coloneqq \min\{2,p\}\in C^{0,\alpha}(\overline{\Omega})$, and let $\bfz\in W^{1,p(\cdot)}_0(\Omega;\mathbb{R}^d)$ be such that
		$\bfF(\cdot,\bfD\bfz)\in N^{\beta,2}(\Omega;\mathbb{R}^{d\times d})$ with $\beta\in (0,1]$. Then, there exists a constant $s>1$, which can be chosen to be close to $1$ if $h_T>0$ is close to $0$, such that for every $T\in \mathcal{T}_h$,~it~holds~that
		\begin{align}\label{cor:Pi_div_F.1}
			\smash{  \|\bfD\bfz-\bfD\Pi_h^V\bfz\|_{r(\cdot),T}^2\lesssim h_T^{2\alpha}\,\|1+\vert \bfD\bfz\vert^{p(\cdot)s}\|_{1,\omega_T}
				+ h_T^{2\beta}\,[ \bfF(\cdot,\bfD\bfz)]_{N^{\beta,2}(\omega_T)}^2\,,}
		\end{align}
		where $\lesssim$ depends only on $d$, $k$, $p^-$, $p^+$, $[p]_{\alpha,\Omega}$, $s$, $\omega_0$, and $\|\bfD\bfz\|_{p(\cdot),\Omega}$. In particular, it holds that
		\begin{align}\label{cor:Pi_div_F.2}
			\smash{  \|\bfD\bfz-\bfD\Pi_h^V\bfz\|_{r(\cdot),\Omega}^2\lesssim  h^{2\alpha}\,(1+\rho_{p(\cdot)s,\Omega}(\bfD\bfz))
				+ h^{2\beta}\,[ \bfF(\cdot,\bfD\bfz)]_{N^{\beta,2}(\Omega)}^2 \,.}
		\end{align}
	\end{corollary}
	
	\begin{proof} \textit{ad \eqref{cor:Pi_div_F.1}.} Using Lemma \ref{lem:sobolev2F}, Lemma \ref{lem:stab_Pi_div_V}\eqref{lem:stab_Pi_div_V.local} (with $a=\delta=0$ and $m=d$), and~Lemma~\ref{lem:Pi_div_F.2}\eqref{lem:Pi_div_F.2_1}, for every 
		$T\in \mathcal{T}_h$, we find that\vspace*{-0.5mm}
		\begin{align*}
			\|\bfD\bfz-\bfD\Pi_h^V\bfz\|_{r(\cdot),T}&\lesssim 
			\|\bfF(\cdot,\bfD\bfz)-\bfF(\cdot,\bfD\Pi_h^V\bfz)\|_{2,T}\\&\quad \times\big((1+\rho_{p(\cdot),T}(\vert \bfD\bfz\vert+\vert \bfD\Pi_h^V\bfz\vert))^{\smash{1/p^-}}+(\min\{1,\delta\})^{\smash{2-p^+}}\big)\\&\lesssim  
			\|\bfF(\cdot,\bfD\bfz)-\bfF(\cdot,\bfD\Pi_h^V\bfz)\|_{2,T}\\&\quad \times\big((1+\rho_{p(\cdot),T}(\bfD\bfz))^{\smash{1/p^-}}+(\min\{1,\delta\})^{\smash{2-p^+}}\big)
			\\&\lesssim  \|\bfF(\cdot,\bfD\bfz)-\bfF(\cdot,\bfD\Pi_h^V\bfz)\|_{2,T}\\&\lesssim   h_T^{2\alpha}\,\big(1+\rho_{p(\cdot)s,\omega_T}(\bfD\bfz)\big)
			+h_T^{2\beta}\,[ \bfF(\cdot,\bfD\bfz)]_{N^{\beta,2}(\omega_T)}^2 \,,
		\end{align*}    
		which is the first claimed interpolation error estimate \eqref{cor:Pi_div_F.1}. 
		
		\textit{ad \eqref{cor:Pi_div_F.2}.} \!The \hspace*{-0.1mm}second \hspace*{-0.1mm}claimed \hspace*{-0.1mm}interpolation \hspace*{-0.1mm}error \hspace*{-0.1mm}estimate \hspace*{-0.1mm}\eqref{cor:Pi_div_F.2} \hspace*{-0.1mm}follows \hspace*{-0.1mm}analogously~\hspace*{-0.1mm}with~\hspace*{-0.1mm}Lemma~\hspace*{-0.1mm}\ref{lem:sobolev2F}, Lemma \ref{lem:stab_Pi_div_V}\eqref{lem:stab_Pi_div_V.global} (with $a=\delta=0$ and $m=d$), and  Lemma \ref{lem:Pi_div_F.2}\eqref{lem:Pi_div_F.2_2}.
	\end{proof}
	
	Given fractional regularity of the pressure expressed in fractional variable Haj\l asz--Sobolev spaces, we have the following fractional interpolation error estimate for $\Pi_h^Q$  (\textit{cf}.\ Assumption~\ref{ass:PiY}) measured in the non-discrete Luxembourg norm.

	\begin{lemma}\label{lem:Pi_Q_norm}
		Let $p\in \mathcal{P}^{\log}(\Omega)$ with $p^->1$ and let $z\in H^{\gamma,p'(\cdot)}(\Omega)$ with $\gamma\in (0,1]$ with Haj\l asz gradient $\vert \nabla^{\gamma} z\vert \coloneqq \textup{argmin}_{g\in \mathrm{Gr}(z)}{\|g\|_{p'(\cdot),\Omega}}\in L^{p'(\cdot)}(\Omega)$. Then, for every $T\in \mathcal{T}_h$, it holds that
		\begin{align}
			\smash{	\|z-\Pi_h^Q z\|_{p'(\cdot),T}\lesssim h_T^{\gamma}\,\|\vert \nabla^{\gamma} z\vert \|_{p'(\cdot),\omega_T}\,,}\label{lem:Pi_Q_norm.local}
		\end{align}
		where $\lesssim$ depends only on $d$, $\textcolor{black}{\ell}$, $p^-$, $p^+$, $c_{\log}(p)$, and $\omega_0$. In particular, it holds that
		\begin{align}
			\smash{\|z-\Pi_h^Q z\|_{p'(\cdot),\Omega}\lesssim h^{\gamma}\,\|\vert \nabla^{\gamma} z\vert \|_{p'(\cdot),\Omega}\,.}\label{lem:Pi_Q_norm.global}
		\end{align}
	\end{lemma}
	
	\begin{proof}
		\textit{ad \eqref{lem:Pi_Q_norm.local}.} Using that ${\omega_T\subseteq \overline{B_{\textrm{diam}(\omega_T)}^d(x_T)}}$, where we denote by $x_T\in T$ the barycenter~of~$T$, that ${\vert \omega_T\vert \sim \vert B_{\textrm{diam}(\omega_T)}^d(x_T)\vert}$, and the $L^{p'(\cdot)}(\mathbb{R}^d)$-$L^{p'(\cdot)}(\mathbb{R}^d)$-stability of the Hardy--Littlewood~maximal~operator $M_d\colon  L^{p'(\cdot)}(\mathbb{R}^d)\to L^{p'(\cdot)}(\mathbb{R}^d)$ (\textit{cf}.\  \cite[Thm.\  4.3.8]{dhhr}), 
		we find that
		\begin{align}\label{eq:Pi_Q_norm.1}
			\begin{aligned}
				\|\langle\vert \nabla^{\gamma} z\vert\rangle_{\omega_T}\|_{p'(\cdot),\omega_T}&\lesssim \big\|\chi_{\omega_T}\langle\chi_{\omega_T}\vert \nabla^{\gamma} z\vert\rangle_{\smash{B_{\textrm{diam}(\omega_T)}^d(x_T)}}\big\|_{p'(\cdot),\mathbb{R}^d}
				\\&\lesssim \|M_d(\chi_{\omega_T}\vert \nabla^{\gamma} z\vert)\|_{p'(\cdot),\mathbb{R}^d} 
				\lesssim \|\vert \nabla^{\gamma }z\vert \|_{p'(\cdot),\omega_T}\,.
			\end{aligned}
		\end{align}
		By the definition of the Haj\l azs gradient (\textit{cf}.\  \eqref{eq:hajlasz_gradient}), for every $T\in \mathcal{T}_h$ and a.e. $x\in \omega_T$, we have that
		\begin{align}\label{eq:Pi_Q_norm.2}
			\smash{	\vert z(x)-\langle z\rangle_{\omega_T}\vert \leq \langle\vert  z(x)-z\vert \rangle_{\omega_T}\lesssim h_T^{\gamma}\, \big(\vert\nabla^\gamma z\vert(x)+\langle\vert  \nabla^\gamma z\vert \rangle_{\omega_T}\big)\,.}
		\end{align}
		Therefore, using that $\smash{\Pi_h^Q \langle z\rangle_{\omega_T}=\langle z\rangle_{\omega_T}}$, Lemma \ref{lem:stab_Pi_Q_norm}\eqref{lem:stab_Pi_Q_norm.1}, \eqref{eq:Pi_Q_norm.1}, and \eqref{eq:Pi_Q_norm.2}, we conclude that
		\begin{align*}
			\|z-\Pi_h^Q z\|_{p'(\cdot),T} &= \|z-\langle z\rangle_{\omega_T}- \Pi_h^Q (z-\langle z\rangle_{\omega_T})\|_{p'(\cdot),T} 
			\\&\lesssim h_T^{\gamma}\,\|\vert\nabla^{\gamma} z\vert+\langle\vert \nabla^{\gamma} z\vert\rangle_{\omega_T}\|_{p'(\cdot),T}+ \|z-\langle z\rangle_{\omega_T}\|_{p'(\cdot),\omega_T} 
			%\\&
			\lesssim h_T^{\gamma}\,\|\vert \nabla^{\gamma }z\vert \|_{p'(\cdot),\omega_T}\,,
		\end{align*}
		which is the claimed fractional approximation error estimate \eqref{lem:Pi_Q_norm.local}.
		
		\textit{ad \eqref{lem:Pi_Q_norm.global}.} Using \cite[Cor.\  7.3.21]{dhhr}, \eqref{lem:Pi_Q_norm.local}, $\|1\|_{p'(\cdot),\omega_T}\lesssim \|1\|_{p'(\cdot),T}$, and, again, \cite[Cor.\  7.3.21]{dhhr},~we~get
		\begin{align*}
			\|z-\Pi_h^Q z\|_{p'(\cdot),\Omega} 
			&\lesssim \bigg\|\sum_{T\in \mathcal{T}_h}{\chi_T \frac{	\|z-\Pi_h^Q z\|_{p'(\cdot),T}}{\|1\|_{p'(\cdot),T}}}\bigg\|_{p'(\cdot),\mathbb{R}^d} 
			\\&\lesssim\bigg\|\sum_{T\in \mathcal{T}_h}{\chi_{\omega_T} \frac{h_T^{\gamma}\,\|\vert \nabla^{\gamma }z\vert \|_{p'(\cdot),\omega_T}}{\|1\|_{p'(\cdot),\omega_T}}}\bigg\|_{p'(\cdot),\mathbb{R}^d}
			\\&
			\lesssim h^{\gamma}\,\|\vert \nabla^{\gamma }z\vert \|_{p'(\cdot),\Omega}\,,
		\end{align*}
		which is the claimed fractional approximation error estimate \eqref{lem:Pi_Q_norm.global}.
	\end{proof}
	
	In addition to the fractional interpolation error estimate measured in the non-discrete Luxembourg norm (\textit{cf}.\  Lemma \ref{lem:Pi_Q_norm}), the following  fractional interpolation error estimate 
	measured in the modular with respect to the conjugate of shifted generalized $N$-function applies.
	
	\begin{lemma}\label{lem:Pi_Q}
		Let $p\in C^{0,\alpha}(\overline{\Omega})$ with $p^->1$ and $\alpha\in (0,1]$, let $z\in H^{\gamma,p'(\cdot)}(\Omega)$ with $\gamma\in (0,1]$ with Haj\l asz gradient $\vert \nabla^{\gamma} z\vert \coloneqq \textup{argmin}_{g\in \mathrm{Gr}(z)}{\|g\|_{p'(\cdot),\Omega}}\in L^{p'(\cdot)}(\Omega)$, and let $\bfA\in L^{p(\cdot)}(\Omega;\mathbb{R}^{d\times d})$. Then, for every $m>0$ and  $T\in \mathcal{T}_h$, it holds that
		\begin{align}\label{lem:Pi_Q.1}
			\begin{aligned}
				\rho_{(\varphi_{\vert \bfA\vert})^*,T}(z-\Pi_h^Q z)&\leq c\,\big(h_T^m+\rho_{(\varphi_{\vert \bfA\vert})^*,\omega_T}(h_T^{\gamma}\vert \nabla^{\gamma} z\vert)\big)\\&\quad + c\,\|\bfF(\cdot,\bfA)-\bfF(\cdot,\langle\bfA\rangle_{\omega_T})\|_{2,\omega_T}^2 \,,
			\end{aligned}
		\end{align}
		where $\lesssim$ depends only on $d$, $k$, $\textcolor{black}{\ell}$, $m$, $p^-$, $p^+$, $[p]_{\alpha,\Omega}$, $\omega_0$, $\|\bfA\|_{p(\cdot),\Omega}$, $\|z\|_{\gamma,p'(\cdot),\Omega}$.
		In particular,~it~holds~that 
		\begin{align}\label{lem:Pi_Q.2}
			\begin{aligned}
				\rho_{(\varphi_{\vert \bfA\vert})^*,\Omega}(z-\Pi_h^Q z)&\leq c\,\big(h^m+\rho_{(\varphi_{\vert\bfA\vert})^*,\Omega}(h^{\gamma}\vert \nabla^{\gamma} z\vert)\big)\\&\quad+
				c\,\sum_{T\in \mathcal{T}_h}{\|\bfF(\cdot,\bfA)-\bfF(\cdot,\langle\bfA\rangle_{\omega_T})\|_{2,\omega_T}^2}\,.  
			\end{aligned}
		\end{align}
	\end{lemma}
	
	\begin{proof}
		\textit{ad \eqref{lem:Pi_Q.1}.} Using the shift change Lemma \ref{lem:shift-change}\eqref{lem:shift-change.3}, for every $T\in \mathcal{T}_h$, we find that\enlargethispage{4mm}
		\begin{align}\label{eq:Pi_Q.1}
			\begin{aligned}
				\rho_{(\varphi_{\vert \bfA\vert})^*,T}(z-\Pi_h^Q z)&\leq c\, \rho_{(\varphi_{\vert \langle\bfA\rangle_{\omega_T}\vert})^*,T}(z-\Pi_h^Q z)\\&\quad+c\,\|\bfF(\cdot,\bfA)-\bfF(\cdot,\langle\bfA\rangle_{\omega_T})\|_{2,T}^2
				\,.\end{aligned}
		\end{align}
		Since \hspace*{-0.1mm}$\vert \langle\bfA\rangle_{\omega_T}\vert +\langle \vert z-\langle z\rangle_{\omega_T}\vert \rangle_T\hspace*{-0.1em}\leq\hspace*{-0.1em} c\,\vert T\vert^{-1}\hspace*{-0.1em}\leq\hspace*{-0.1em} c\,\vert T\vert^{-m}$, \hspace*{-0.1mm}where \hspace*{-0.1mm}$c\hspace*{-0.1em}>\hspace*{-0.1em}0$ \hspace*{-0.1mm}depends \hspace*{-0.1mm}only \hspace*{-0.1mm}on 
	\hspace*{-0.1mm}$d$, 	\hspace*{-0.1mm}$k$, \hspace*{-0.1mm}$\textcolor{black}{\ell}$, \hspace*{-0.1mm}$p^-\!$,~\hspace*{-0.1mm}$p^+\!$,~\hspace*{-0.1mm}$\omega_0$,~\hspace*{-0.1mm}$\|\bfA\|_{p(\cdot),\Omega}$, and $\|z\|_{p'(\cdot),\Omega}$, 
		using  Lemma \ref{lem:stab_Pi_Q}\eqref{lem:stab_Pi_Q.local} (with $a=\vert \langle\bfA\rangle_{\omega_T}\vert$) and $\Pi_h^Q\langle z\rangle_{\omega_T}=\langle z\rangle_{\omega_T}$,~for~every~${ T\in \mathcal{T}_h} $,~we~get
		\begin{align}\label{eq:Pi_Q.2}
			\begin{aligned}
				\rho_{(\varphi_{\vert \langle\bfA\rangle_{\omega_T}\vert})^*,T}(z-\Pi_h^Q z)&\leq c\,\rho_{(\varphi_{\vert \langle\bfA\rangle_{\omega_T}\vert})^*,T}(z-\langle z\rangle_{\omega_T})
				\\&\quad+c\,\rho_{(\varphi_{\vert \langle\bfA\rangle_{\omega_T}\vert})^*,T}(\Pi_h^Q(z-\langle z\rangle_{\omega_T}))
				\\&\leq c\,\rho_{(\varphi_{\vert \langle\bfA\rangle_{\omega_T}\vert})^*,\omega_T}(z-\langle z\rangle_{\omega_T})+c\,h_T^{m}\,.
			\end{aligned}
		\end{align}
		Thus, using \eqref{eq:Pi_Q.2} together with  \eqref{eq:Pi_Q_norm.2}, the key estimate (\textit{cf}.\  Lemma \ref{lem:key-estimate} (with $a=\vert \langle\bfA\rangle_{\omega_T}\vert$)), since $\vert \langle\bfA\rangle_{\omega_T}\vert +h_T^{\gamma} \langle\vert  \nabla^\gamma z\vert \rangle_{\omega_T}\leq c\,\vert T\vert^{-1}\leq c\,\vert T\vert ^{-m}$, where $c>0$ depends only on $d$, $k$, $\textcolor{black}{\ell}$, $p^-$, $p^+$, $\omega_0$, $\|\bfA\|_{p(\cdot),\Omega}$, and $\|\vert \nabla^{\gamma}z\vert\|_{p'(\cdot),\Omega}$, and the shift change Lemma \ref{lem:shift-change}\eqref{lem:shift-change.3}, for every $T\in \mathcal{T}_h$, we find that
		\begin{align}\label{eq:Pi_Q.4}
			\begin{aligned}
				\rho_{(\varphi_{\vert \langle\bfA\rangle_{\omega_T}\vert})^*,T}(z-\Pi_h^Q z)&\leq c\,\rho_{(\varphi_{\vert \langle\bfA\rangle_{\omega_T}\vert})^*,\omega_T}(h_T^{\gamma} \vert\nabla^\gamma z\vert)
				\\&\quad +c\,\rho_{(\varphi_{\vert \langle\bfA\rangle_{\omega_T}\vert})^*,\omega_T}(h_T^{\gamma} \langle\vert  \nabla^\gamma z\vert \rangle_{\omega_T})+c\,h_T^m
				\\
				&\leq c\,\rho_{(\varphi_{\vert \langle\bfA\rangle_{\omega_T}\vert})^*,\omega_T}(h_T^{\gamma} \vert\nabla^\gamma z\vert)
				+c\,h_T^m
				\\
				&\leq c\,\rho_{(\varphi_{\vert \bfA
						\vert})^*,\omega_T}(h_T^{\gamma} \vert\nabla^\gamma z\vert)
				\\&\quad +
				c\,\|\bfF(\cdot,\bfA)-\bfF(\cdot,\langle\bfA\rangle_{\omega_T})\|_{2,\omega_T}^2 +c\,h_T^m\,.
			\end{aligned}
		\end{align}
		Eventually, using \eqref{eq:Pi_Q.4} in \eqref{eq:Pi_Q.1}, we arrive at the first claimed  interpolation error estimate \eqref{lem:Pi_Q.1}. 
		
		\textit{ad \eqref{lem:Pi_Q.2}.} The second  interpolation error estimate \eqref{lem:Pi_Q.2} follows from the first via summation.
	\end{proof}

	\newpage 
	\section{A priori error estimates}\label{sec:a_priori}
	
	\hspace{5mm}In this section, we derive \textit{a priori} error estimates for the approximation of the steady $p(\cdot)$-Navier--Stokes equations  \eqref{eq:p-navier-stokes} (\textit{i.e.}, Problem (\hyperlink{Q}{Q}) and Problem (\hyperlink{P}{P}), respectively)
	through the discrete  steady $p(\cdot)$-Navier--Stokes equations (\textit{i.e.}, Problem (\hyperlink{Qh}{Q$_h$}) and  Problem (\hyperlink{Ph}{P$_h$}), respectively).\enlargethispage{8mm}

	\begin{theorem}
		\label{thm:error_FE}
		Let 
		$p\in C^{0,\alpha}(\overline{\Omega})$ with $p^-\ge \frac{3d}{d+2}$ and $\alpha\in (0,1]$, let $\delta> 0$, let  $\bfF(\cdot,\bfD\bfv)\in N^{\beta,2}(\Omega;\mathbb{R}^{d\times d})$ with $\beta\in (0,1]$ and let $q\in H^{\gamma,p'(\cdot)}(\Omega)$ with $\gamma\in (\frac{\alpha}{\min\{2,(p^+)'\}},1]$. Moreover, let $h\sim h_T$~for~all~$T\in \mathcal{T}_h$.
		Then, there exists a~constant~$s>1$, which can chosen to be close to $1$ if $h>0$ is close to $0$, and a constant $c_0 >0$, depending only on
		$d$, $k$, $\textcolor{black}{\ell}$, $p^- $, $p^+$, $[p]_{\alpha,\Omega}$, $\delta^{-1}$, $\omega_0$,  $\Omega$, and $s$, such that if $\|\bfD
		\bfv\|_{r(\cdot),\Omega}\le c_0$,~where $r\coloneqq \min\{2,p\}\in C^{0,\alpha}(\overline{\Omega})$, then
		\begin{align*}
			\|\bfF_h(\cdot,\bfD \bfv_h)-\bfF_h(\cdot,\bfD
			\bfv)\|_{2,\Omega}^2+ \smash{\| q_h-q\|_{p'(\cdot),\Omega}^{(r^-)'}}&\lesssim  h^{2\alpha}\,\big(1+\rho_{p(\cdot)s,\Omega}(\bfD\bfv)\big)\\&\quad+
			h^{2\beta}\,[\bfF(\cdot,\bfD\bfv)]_{N^{\beta,2}(\Omega)}^2\\&\quad+
			\smash{\rho_{(\varphi_{\vert \bfD\bfv\vert})^*,\Omega}(h^{\gamma}\vert \nabla^{\gamma}q\vert )+h^{(r^-)'\gamma}\|\vert \nabla^{\gamma}q\vert \|_{p'(\cdot),\Omega}^{(r^-)'}}\,,
		\end{align*}
		where $\lesssim $ depends on 
		$d$, $k$, $\textcolor{black}{\ell}$,  $p^- $, $p^+$, $[p]_{\alpha,\Omega}$, $\delta^{-1}$, $\omega_0$, $\Omega$,  $s$, $c_0$, and $\|q\|_{\gamma,p'(\cdot),\Omega}$.
	\end{theorem}
	
	\begin{remark}[Smallness condition in Theorem \ref{thm:error_FE}]
		The smallness condition $\|\bfD
		\bfv\|_{r(\cdot),\Omega}\le c_0$ for some $c_0>0$ in Theorem \ref{thm:error_FE} is no further restriction of the assumptions, due to the following two aspects:
		\begin{itemize}[leftmargin=!,topsep=2pt,noitemsep,labelwidth=\widthof{(ii)}]
			\item[(i)] \textup{Uniqueness:} The smallness condition ensures the uniqueness of a solution to Problem (\hyperlink{Q}{Q}) (and Problem~(\hyperlink{P}{P}), respectively) as well as of a discrete solution to Problem (\hyperlink{Qh}{Q$_h$}) (and Problem (\hyperlink{Ph}{P$_h$}), respectively) and, consequently, the well-posedness of the \textit{a priori} error estimate in Theorem~\ref{thm:error_FE};
			
			\item[(ii)] \textup{Regularity:} The smallness condition  ensures higher regularity  of solutions to Problem (\hyperlink{Q}{Q}) (or~Problem~(\hyperlink{P}{P})), \textit{i.e.}, for $\bfv\in W^{2,2}(\Omega;\mathbb{R}^d)\cap C^{1,\gamma}(\overline{\Omega};\mathbb{R}^d)$~for~some~${\gamma>0}$ and $q\in W^{1,2}(\Omega)\cap C^{0,\gamma}(\overline{\Omega})$~(\textit{cf}.~\cite{CG09}), so \hspace*{-0.1mm}that \hspace*{-0.1mm}we \hspace*{-0.1mm}expect \hspace*{-0.1mm}it \hspace*{-0.1mm}to \hspace*{-0.1mm}be \hspace*{-0.1mm}equally \hspace*{-0.1mm}needed \hspace*{-0.1mm}to \hspace*{-0.1mm}ensure \hspace*{-0.1mm}the \hspace*{-0.1mm}fractional \hspace*{-0.1mm}regularity \hspace*{-0.1mm}assumptions \hspace*{-0.1mm}in \hspace*{-0.1mm}Theorem~\hspace*{-0.1mm}\ref{thm:error_FE}.
		\end{itemize} 
	\end{remark}
	
	As an immediate consequence of Theorem \ref{thm:error_FE}, we obtain two error estimates with explicit~decay~rates.
	
	\begin{corollary}\label{cor:error_FE}
		Let the assumptions of Theorem \ref{thm:error_FE} be satisfied. Then, there exists a constant  $s>1$, which can chosen to be close to $1$ if $h>0$ is close to $0$, and a constant $c_0 >0$, depending only on
		$d$, $k$, $\textcolor{black}{\ell}$,  $p^- $, $p^+$, $[p]_{\alpha,\Omega}$, $\delta^{-1}$, $\omega_0$, $\Omega$,  and $s$, such that if $\|\bfD
		\bfv\|_{r(\cdot),\Omega}\le c_0$, where $r\coloneqq \min\{2,p\}\in C^{0,\alpha}(\overline{\Omega})$,~then
		\begin{align}
			\|\bfF_h(\cdot,\bfD \bfv_h)-\bfF_h(\cdot,\bfD
			\bfv)\|_{2,\Omega}^2+ \smash{\| q_h-q\|_{p'(\cdot),\Omega}^{(r^-)'}}&\lesssim h^{2\alpha}\,\big(1+\rho_{p(\cdot)s,\Omega}(\bfD\bfv)\big)\notag\\&\quad+
			h^{2\beta}\,[\bfF(\cdot,\bfD\bfv)]_{N^{\beta,2}(\Omega)}^2\label{cor:error_FE.1}\\&\quad+
			\smash{h^{\min\{2,(p^+)'\}\gamma}}\big(\rho_{(\varphi_{\vert \bfD\bfv\vert})^*,\Omega}(\vert\nabla^{\gamma}q\vert)
			+ \smash{\|\vert \nabla^{\gamma}q\vert \|_{p'(\cdot),\Omega}^{(r^-)'}}\big)\,,\notag
		\end{align}
		where $\lesssim $ depends on 
		$d$, $k$, $\textcolor{black}{\ell}$, $p^- $, $p^+$, $[p]_{\alpha,\Omega}$, $\delta^{-1}$, $\omega_0$,  $\Omega$,  $s$, $c_0$, and $\|q\|_{\gamma,p'(\cdot),\Omega}$. If, in addition, $p^-\ge 2$ and $(\delta+\vert\bfD\bfv\vert)^{\smash{\frac{2-p(\cdot)}{2}}}\vert \nabla^{\gamma} q\vert \in L^2(\Omega)$, then
		\begin{align}
			\|\bfF_h(\cdot,\bfD \bfv_h)-\bfF_h(\cdot,\bfD
			\bfv)\|_{2,\Omega}^2+ \smash{\| q_h-q\|_{p'(\cdot),\Omega}^2}
			&\lesssim  h^{2\alpha}\,\big(1+\rho_{p(\cdot)s,\Omega}(\bfD\bfv)\big)\notag\\&\quad+
			h^{2\beta}\,[\bfF(\cdot,\bfD\bfv)]_{N^{\beta,2}(\Omega)}^2\label{cor:error_FE.2}\\&\quad +h^{2\gamma}\,\big( \smash{\|(\delta\hspace*{-0.05em}+\hspace*{-0.05em}\vert\bfD\bfv\vert)^{\smash{\frac{2-p(\cdot)}{2}}}\vert\nabla^{\gamma} q\vert\|_{2,\Omega}^2 \hspace*{-0.05em}+\hspace*{-0.05em}\|\vert \nabla^{\gamma}q\vert \|_{p'(\cdot),\Omega}^2} \big)\,,\notag
		\end{align}
		where $\lesssim $ depends on 
		$d$, $k$, $\textcolor{black}{\ell}$, $p^- $, $p^+$, $[p]_{\alpha,\Omega}$, $\delta^{-1}$, $\omega_0$,  $\Omega$,  $s$, $c_0$, and $\|q\|_{\gamma,p'(\cdot),\Omega}$.
	\end{corollary}
	
	Essential ingredient in the treatment of the (discrete) convective term (\textit{cf}.\  \eqref{def:bh}) in the proof of Theorem \ref{thm:error_FE} is the following lemma that relates the $W^{1,r(\cdot)}(\Omega;\mathbb{R}^{d\times d})$-semi-norm of the discrete and the exact velocity to their discrete natural distance. It is decisively based on 
	Lemma \ref{lem:norm_equiv} and Lemma \ref{lem:sobolev2F}.
	
	\begin{lemma}\label{lem:sobolev2F_discrete}
		Let the assumptions of Theorem \ref{thm:error_FE} be satisfied. Then, there exists a constant  $s>1$, which can chosen to be close to $1$ if $h>0$ is close to $0$, such that 
		\begin{align*}
			\|\bfD\bfv_h-\bfD\bfv\|_{r(\cdot),\Omega}^2&\lesssim \|\bfF_h(\cdot,\bfD\bfv_h)-\bfF_h(\cdot,\bfD\bfv)\|_{2,\Omega}^2\\&\quad+ h^{2\alpha}\,\big(1+\rho_{p(\cdot)s,\Omega}(\bfD\bfv)\big)+h^{2\beta}\,[\bfF(\cdot,\bfD\bfv)]_{N^{\beta,2}(\Omega)}^2 \,,
		\end{align*}
		where $\lesssim$ depends only on $d$, $k$, $\textcolor{black}{\ell}$, $p^- $, $p^+$, $[p]_{\alpha,\Omega}$, $\delta^{-1}$, $\omega_0$, $\Omega$,  $s$, and $\|\bfD
		\bfv\|_{r(\cdot),\Omega}$.
	\end{lemma}
	\begin{proof}
		Lemma \ref{lem:norm_equiv}, Lemma \ref{lem:sobolev2F}, the estimate \eqref{eq:stability}, and
		$\rho_{p_h(\cdot),\Omega}(\bfD\Pi_h^V\bfv)\leq 1+\|\bfD\Pi_h^V\bfv\|_{p_h(\cdot),\Omega}^{p^+}\lesssim 1+\smash{\|\bfD\Pi_h^V\bfv\|_{p(\cdot),\Omega}^{p^+}}\lesssim 1+\smash{\|\bfD\bfv\|_{p(\cdot),\Omega}^{p^+}}$ (\textit{cf}.\  \cite[Lem.\  3.2.5]{dhhr}, Lemma \ref{lem:norm_equiv}, Lemma \ref{lem:stab_Pi_div_V_norm.2},~and~Theorem~\ref{thm:korn})~yield 
		\begin{align}\label{eq:sobolev2F_discrete.1}
			\begin{aligned}
				\|\bfD\bfv_h-\bfD\Pi_h^V\bfv\|_{r(\cdot),\Omega}^{\textcolor{black}{2}}&\lesssim  \|\bfD\bfv_h-\bfD\Pi_h^V\bfv\|_{r_h(\cdot),\Omega}^{\textcolor{black}{2}}\\&\lesssim
				\|\bfF_h(\cdot,\bfD\bfv_h)-\bfF_h(\cdot,\bfD\Pi_h^V\bfv)\|_{2,\Omega}^{\textcolor{black}{2}}\\&\quad \times\big((1+\rho_{p_h(\cdot),\Omega}(\vert \bfD\bfv_h\vert+\vert \bfD\Pi_h^V\bfv\vert))^{\textcolor{black}{\smash{\frac{2}{p^-}}}}+(\min\{1,\delta\})^{\smash{2-p^+}}\big)
				\\&\lesssim  \|\bfF_h(\cdot,\bfD\bfv_h)-\bfF_h(\cdot,\bfD\Pi_h^V\bfv)\|_{2,\Omega}^{\textcolor{black}{2}}\,.
			\end{aligned}
		\end{align}
		Therefore, using \eqref{eq:sobolev2F_discrete.1}, Lemma \ref{lem:Pi_div_F}\eqref{lem:Pi_div_F_2}, and Corollary \ref{cor:Pi_div_F}\eqref{cor:Pi_div_F.2}, we have that
		\begin{align*}
			\|\bfD\bfv_h-\bfD\bfv\|_{r(\cdot),\Omega}^2&\leq \|\bfD\bfv_h-\bfD\Pi_h^V\bfv\|_{r(\cdot),\Omega}^2+\|\bfD\bfv-\bfD\Pi_h^V\bfv\|_{r(\cdot),\Omega}^2\\
			&\lesssim \|\bfF_h(\cdot,\bfD\bfv_h)-\bfF_h(\cdot,\bfD\Pi_h^V\bfv)\|_{2,\Omega}^2+\|\bfD\bfv-\bfD\Pi_h^V\bfv\|_{r(\cdot),\Omega}^2\
			\\
			&\lesssim \|\bfF_h(\cdot,\bfD\bfv_h)-\bfF_h(\cdot,\bfD\bfv)\|_{2,\Omega}^2+\|\bfF_h(\cdot,\bfD\bfv)-\bfF_h(\cdot,\bfD\Pi_h^V\bfv)\|_{2,\Omega}^2\\&\quad+\|\bfD\bfv-\bfD\Pi_h^V\bfv\|_{r(\cdot),\Omega}^2\
			\\
			&\lesssim \|\bfF_h(\cdot,\bfD\bfv_h)-\bfF_h(\cdot,\bfD\bfv)\|_{2,\Omega}^2\\&\quad +h^{2\alpha}\,\big(1+\rho_{p(\cdot)s,\Omega}(\bfD\bfv)\big)+h^{2\beta}\,[\bfF(\cdot,\bfD\bfv)]_{N^{\beta,2}(\Omega)}^2\,,
		\end{align*}
		which is the claimed estimate.
	\end{proof}\vspace*{-1mm}
	
	We have now everything at our disposal to prove Theorem \ref{thm:error_FE}.\vspace*{-1mm}
	
	\begin{proof}[Proof (of Theorem \ref{thm:error_FE}).]
		We split the proof of Theorem \ref{thm:error_FE} into two main steps:
		
		\textit{1. Error estimate for the velocity vector field:}  First, we abbreviate $\bfe_h\coloneqq \bfv_h-\bfv\in \Vo$. Then, using \eqref{eq:hammera} and the decomposition
		$\bfe_h = \Pi_h^V\bfe_h + \Pi_h^V \bfv -\bfv$ in $\Vo$, we find that
		\begin{align}\label{thm:error_FE.1}
			\begin{aligned}
				c\, \|\bfF_h(\cdot,\bfD \bfv_h) - \bfF_h(\cdot,\bfD
				\bfv)\|_{2,\Omega}^2&\leq (\bfS _h(\cdot,\bfD \bfv_h) - \bfS _h(\cdot,\bfD
				\bfv),\bfD \bfe_h)_\Omega
				  \\&= (\bfS _h(\cdot,\bfD \bfv_h) - \bfS(\cdot,\bfD
							\bfv),\bfD \bfe_h)_\Omega
				      \\&\quad+(\bfS(\cdot,\bfD
								\bfv)-\bfS _h(\cdot,\bfD \bfv),\bfD \bfe_h)_\Omega
						\\
							&= (\bfS _h(\cdot,\bfD \bfv_h) - \bfS(\cdot,\bfD
								\bfv),\bfD \Pi_h^V\bfe_h)_\Omega
						\\&\quad+ (\bfS _h(\cdot,\bfD \bfv_h) - \bfS(\cdot,\bfD
							\bfv),\bfD\Pi_h^V\bfv-\bfD\bfv)_\Omega
				 \\&\quad+(\bfS(\cdot,\bfD
				 \bfv)-\bfS _h(\cdot,\bfD \bfv),\bfD \bfe_h)_\Omega
				\\&= (\bfS _h(\cdot,\bfD \bfv_h) - \bfS(\cdot,\bfD
				\bfv),\bfD \Pi_h^V\bfe_h)_\Omega
				\\&\quad+(\bfS _h(\cdot,\bfD \bfv_h) - \bfS_h(\cdot,\bfD
				\bfv),\bfD\Pi_h^V\bfv-\bfD\bfv)_\Omega
				\\&\quad+ (\bfS _h(\cdot,\bfD \bfv) - \bfS(\cdot,\bfD
				\bfv),\bfD  \Pi_h^V\bfv-\bfD\bfv)_\Omega
				\\&\quad+(\bfS(\cdot,\bfD
				\bfv)-\bfS _h(\cdot,\bfD \bfv),\bfD \bfe_h)_\Omega
				\\&\eqqcolon I_h^1+ I_h^2 +I_h^3 +I_h^4\,.
			\end{aligned}
		\end{align}
		The $\varepsilon$-Young inequality \eqref{ineq:young} with $\psi=(\varphi_h)_{\smash{\vert \bfD\bfv\vert}}$, \eqref{eq:hammera}, \eqref{eq:hammerh},  Lemma \ref{lem:Pi_div_F}\eqref{lem:Pi_div_F_2}, and Lemma~\ref{lem:A-Ah}\eqref{eq:Ah-A} yield that
		\begin{align}\label{thm:error_FE.2}
			I_h^2&\leq \varepsilon\,\|\smash{\bfF_h(\cdot,\bfD \bfv_h)}-\bfF_h(\cdot,\bfD
			\bfv)\|_{2,\Omega}^2+c_\varepsilon\, \|\bfF_h(\cdot,\bfD\bfv)-\bfF_h(\cdot,\bfD\Pi_h^V\bfv)\|_{2,\Omega}^2\\&\leq
			\varepsilon\,\|\bfF_h(\cdot,\bfD \bfv_h)-\bfF_h(\cdot,\bfD
			\bfv)\|_{2,\Omega}^2+c_\varepsilon \,\big(h^{2\alpha}\big(1+\rho_{p(\cdot)s,\Omega}(\bfD\bfv)\big)+h^{2\beta}\, [\bfF(\cdot,\bfD \bfv)]_{N^{\beta,2}(\Omega)}^2\big)\,;\notag\\[1mm]\label{thm:error_FE.3}
			I_h^3&\leq c\,\|\smash{\bfF_h^*(\cdot,\bfS _h(\cdot,\bfD \bfv)))}-\bfF_h^*(\cdot,\bfS(\cdot,\bfD \bfv))\|_{2,\Omega}^2+c\,\|\bfF_h(\cdot,\bfD\bfv)-\bfF_h(\cdot,\bfD\Pi_h^V\bfv)\|_{2,\Omega}^2\\&\leq
			c\,h^{2\alpha}\,\big(1+\rho_{p(\cdot)s,\Omega}(\bfD\bfv)\big)+c\,h^{2\beta}\, [\bfF(\cdot,\bfD \bfv)]_{N^{\beta,2}(\Omega)}^2\,;\notag\\[1mm]\label{thm:error_FE.3.1}
			I_h^4&\leq c_\varepsilon\,\|\smash{\bfF_h^*(\cdot,\bfS _h(\cdot,\bfD \bfv)))}-\bfF_h^*(\cdot,\bfS(\cdot,\bfD \bfv))\|_{2,\Omega}^2+\varepsilon\,\|\bfF_h(\cdot,\bfD\bfv_h)-\bfF_h(\cdot,\bfD\bfv)\|_{2,\Omega}^2\\&\leq
			c_\varepsilon \,h^{2\alpha}\big(1+\rho_{p(\cdot)s,\Omega}(\bfD\bfv)\big)+\varepsilon\,\|\bfF_h(\cdot,\bfD\bfv_h)-\bfF_h(\cdot,\bfD\bfv)\|_{2,\Omega}^2\,.\notag
		\end{align}
		
		Testing the first lines of Problem (\hyperlink{Q}{Q}) and Problem (\hyperlink{Qh}{Q$_h$}) with  $\bfz_h=\Pi_h^V\bfe_h\in \textcolor{black}{\Vo_{h,0}}$,  then, subtracting the resulting equations, and using that $(q,\divo \Pi_h^V\bfe_h)_\Omega=(\textcolor{black}{\Pi}_h^{k-1}q,\divo \Pi_h^V\bfe_h)_\Omega$, where $\textcolor{black}{\Pi}_h^{k-1}\colon \hspace*{-0.1em}L^1(\Omega)\hspace*{-0.1em}\to \hspace*{-0.1em}\mathbb{P}^{\textcolor{black}{k-1}}(\mathcal{T}_h)$ denotes the \textit{(local) $L^2$-projection operator},~as~well~as that $(q_h,\divo \Pi_h^V\bfe_h)_\Omega=0= (z_h,\divo \Pi_h^V\bfe_h)_\Omega$ (\textit{cf}.\ Assumption \ref{ass:proj-div}(ii)) for all  $z_h\in   \smash{\Qo_h}$,
		for~every~${z_h\in Q_h}$,~we~find~that
		\begin{align}\label{thm:error_FE.4}
			\begin{aligned}
				I_h^1&=(z_h-\textcolor{black}{\Pi}_h^{k-1}q,\divo \Pi_h^V\bfe_h)_\Omega \\&\quad+[b(\bfv,\bfv,\Pi_h^V\bfe_h)-b(\bfv_h,\bfv_h,\Pi_h^V\bfe_h)]\\&\eqqcolon  J_h^1+ J_h^2 \,.
			\end{aligned}
		\end{align}
		Here, $b\colon [\Vo]^3\to \mathbb{R}$  denotes the \textit{discrete convective term}, for every 
		$(\bfu,\bfw,\bfz)^\top\in [\Vo]^3$ defined by 
		\begin{align}\label{def:bh}
			\smash{b(\bfu,\bfw,\bfz)\coloneqq \tfrac{1}{2}(\bfz\otimes \bfu,\nabla\bfw)_{\Omega}-\tfrac{1}{2}(\bfw\otimes \bfu,\nabla\bfz)_{\Omega}\,.}
		\end{align}
		
		So, let us next estimate the terms $J_h^1$ and $J_h^2$:
		
		\textit{ad $J_h^1$.} Using, again, that $\Pi_h^V \bfe_h=\bfe_h+\Pi_h^V\bfv -\bfv
		$,  the $\varepsilon$-Young inequality~\eqref{ineq:young} for 
		$\psi =(\phi_h)_{\vert\bfD \bfv\vert}$, \eqref{eq:hammera}, and Lemma \ref{lem:Pi_div_F}\eqref{lem:Pi_div_F_2}, for every $z_h\in Q_h$, we find that
		\begin{align}\label{thm:error_FE.5}
			\begin{aligned}
				J_h^1&=((z_h-\textcolor{black}{\Pi}_h^{k-1}q)\mathbf{I}_d,\bfD \bfe_h)_{\Omega}\\&\quad((z_h-\textcolor{black}{\Pi}_h^{k-1}q)\mathbf{I}_d,\bfD\Pi_h^V \bfv-\bfD\bfv)_{\Omega}
				\\&\leq c_\varepsilon\,\rho_{((\varphi_h)_{\vert\bfD\bfv\vert})^*,\Omega}(z_h-\textcolor{black}{\Pi}_h^{k-1}q)+
				\varepsilon\, c\,\|\bfF_h(\cdot,\bfD\bfv)-\bfF_h(\cdot,\bfD\bfv_h)\|_{2,\Omega}^2\\&\quad+c\,\rho_{((\varphi_h)_{\vert\bfD\bfv\vert})^*,\Omega}(z_h-\textcolor{black}{\Pi}_h^{k-1}q)+c\,\|\bfF_h(\cdot,\bfD\bfv)-\bfF_h(\cdot,\bfD\Pi_h^V\bfv)\|_{2,\Omega}^2
				\\&\leq c_\varepsilon\,\rho_{((\varphi_h)_{\vert\bfD\bfv\vert})^*,\Omega}(z_h-\textcolor{black}{\Pi}_h^{k-1}q)\\&\quad+c\,\big(h^{2\alpha}\big(1+\rho_{p(\cdot)s,\Omega}(\bfD\bfv)\big)+h^{2\beta}\, [\bfF(\cdot,\bfD \bfv)]_{N^{\beta,2}(\Omega)}^2\big)\\&\quad+\varepsilon\,c\,
				\|\bfF_h(\cdot,\bfD\bfv)-\smash{\bfF_h(\cdot,\bfD\bfv_h)}\|_{2,\Omega}^2\,.
			\end{aligned}
		\end{align}
		Using Lemma \ref{lem:A-Ah}\eqref{eq:phih-phi} (with $\lambda\hspace*{-0.1em}=\hspace*{-0.1em}h^{\widetilde{\gamma}}$, where $\widetilde{\gamma}\hspace*{-0.1em}\coloneqq\hspace*{-0.1em} \alpha/\min\{2,(p^+)'\}$, $g\hspace*{-0.1em}=\hspace*{-0.1em}h^{-\widetilde{\gamma}}(\Pi_h^Q q-\textcolor{black}{\Pi}_h^{k-1}q)$,~and~${\bfA\hspace*{-0.1em}=\hspace*{-0.1em}\bfD\bfv}$), 
		and Lemma \ref{lem:Pi_Q}\eqref{lem:Pi_Q.2} together with Lemma~\ref{lem:poincare_F}\eqref{lem:poincare_F.4}, 
		we obtain
		\begin{align}\label{thm:error_FE.6}
			\begin{aligned}
				\rho_{((\varphi_h)_{\vert\bfD\bfv\vert})^*,\Omega}(\Pi_h^Q q-\textcolor{black}{\Pi}_h^{k-1}q)
				&\leq c\,\rho_{(\varphi_{\vert\bfD\bfv\vert})^*,\Omega}(\Pi_h^Q q-\textcolor{black}{\Pi}_h^{k-1}q
				)\\&\quad+c\,h^{2\alpha}\,\big(1+\rho_{p(\cdot)s,\Omega}(\bfD\bfv)\big)
				\\&\quad + c\,h^{\widetilde{\gamma}\min\{2,(p^+)'\}+\alpha}\,\rho_{p'(\cdot)s,\Omega}
				(h^{-\widetilde{\gamma}}(\Pi_h^Q q-\textcolor{black}{\Pi}_h^{k-1}q))
				\\&\leq c\,h^{2\alpha}\,\big(1+\rho_{p(\cdot)s,\Omega}(\bfD\bfv)\big)
				\\&\quad+c\,h^{2\beta}\,[\bfF(\cdot,\bfD\bfv)]_{N^{\beta,2}(\Omega)}^2+ c\,\rho_{(\varphi_{\vert \bfD\bfv\vert})^*,\Omega}(h^{\gamma}\vert\nabla^{\gamma}q\vert)
				\\&\quad + c\,h^{2\alpha}\,\rho_{p'(\cdot)s,\Omega}
				(h^{-\widetilde{\gamma}}(\Pi_h^Q q-\textcolor{black}{\Pi}_h^{k-1}q))\,,
			\end{aligned}
		\end{align}
		so that it is left to estimate the last term on the right-hand side of \eqref{thm:error_FE.6}.  To this end, note first that   $t^{p'(x)s}\hspace*{-0.1em}\leq \hspace*{-0.1em} (1+t)^{p'(\xi_T)p'(x)/p'(\xi_T)s}\hspace*{-0.1em}\leq\hspace*{-0.1em} c\, (1+t^{p'(\xi_T)\widetilde{s}})$ for all $t\hspace*{-0.1em}\ge \hspace*{-0.1em} 0$, $T\hspace*{-0.1em}\in \hspace*{-0.1em} \mathcal{T}_h$, and $x\hspace*{-0.1em}\in\hspace*{-0.1em} T$,~where~${\tilde{s}\hspace*{-0.1em}>\hspace*{-0.1em}1}$~is~a~constant that \hspace*{-0.1mm}can \hspace*{-0.1mm}chosen \hspace*{-0.1mm}to \hspace*{-0.1mm}be \hspace*{-0.1mm}close \hspace*{-0.1mm}to \hspace*{-0.1mm}$1$ \hspace*{-0.1mm}if \hspace*{-0.1mm}$h\hspace*{-0.1em}>\hspace*{-0.1em}0$ \hspace*{-0.1mm}is \hspace*{-0.1mm}close \hspace*{-0.1mm}to \hspace*{-0.1mm}$0$, discrete local norm equivalences
		 \hspace*{-0.1mm}(\textit{cf}.\ \hspace*{-0.1mm}\mbox{\cite[Lem.\ \hspace*{-0.1mm}12.1]{EG21}}), \hspace*{-0.1mm}that \hspace*{-0.1mm}$h\hspace*{-0.1em}\sim \hspace*{-0.1em} h_T$ \hspace*{-0.1mm}for \hspace*{-0.1mm}all \hspace*{-0.1mm}$T\in \mathcal{T}_h$, \hspace*{-0.1mm}and~\hspace*{-0.1mm}that~\hspace*{-0.1mm}${\sum_{i\in \mathbb{N}}{\vert a_i\vert^s}\hspace*{-0.1em}\leq\hspace*{-0.1em} (\sum_{i\in \mathbb{N}}{\vert a_i\vert})^s}$  
		for all~$(a_i)_{i\in \mathbb{N}}\subseteq\ell^1(\mathbb{N})$,~yield~that
		\begin{align}\label{thm:error_FE.6.1}
			\begin{aligned}
			\rho_{p'(\cdot)s,\Omega}
			(h^{-\widetilde{\gamma}}(\Pi_h^Q q-\textcolor{black}{\Pi}_h^{k-1}q))&\leq c\,(1+ \rho_{p'_h(\cdot)\tilde{s},\Omega}
			(h^{-\widetilde{\gamma}}(\Pi_h^Q q-\textcolor{black}{\Pi}_h^{k-1}q))\big)
			\\&\leq  c\,\big(1+ 
			h^{d(1-\tilde{s})}(\rho_{p'_h(\cdot),\Omega}
			(h^{-\widetilde{\gamma}}(\Pi_h^Q q-\textcolor{black}{\Pi}_h^{k-1}q))^{\tilde{s}}\big)\,.
		\end{aligned}
		\end{align}
		In \hspace*{-0.1mm}addition, \hspace*{-0.1mm}using \hspace*{-0.1mm}Lemma \hspace*{-0.1mm}\ref{lem:norm_equiv} \hspace*{-0.1mm}and  \hspace*{-0.1mm}Lemma \hspace*{-0.1mm}\ref{lem:Pi_Q_norm}\eqref{lem:Pi_Q_norm.global} \hspace*{-0.1mm}($\textcolor{black}{\Pi}_h^{k-1}$ \hspace*{-0.1mm}satisfies \hspace*{-0.1mm}Assumption~\hspace*{-0.1mm}\ref{ass:PiY}~\hspace*{-0.1mm}with~\hspace*{-0.1mm}${Q_h\hspace*{-0.15em}=\hspace*{-0.15em}\mathbb{P}^{k-1}(\mathcal{T}_h))}$, we~find~that\enlargethispage{8mm}
		\begin{align*}
			\begin{aligned}
				\|h^{-\widetilde{\gamma}}(\Pi_h^Q q-\textcolor{black}{\Pi}_h^{k-1}q)\|_{p'_h(\cdot),\Omega}&\leq  c\, h^{-\widetilde{\gamma}} \,\|q-\Pi_h^Q q\|_{p'(\cdot),\Omega}
		 +c\,h^{-\widetilde{\gamma}}\,	\|q-\textcolor{black}{\Pi}_h^{k-1}q\|_{p'(\cdot),\Omega}
					\\&\leq c\, h^{\gamma-\widetilde{\gamma}}\,\|\vert \nabla^\gamma q\vert \|_{p'(\cdot),\Omega} 
					\,,
			\end{aligned}
		\end{align*}
		which, appealing to \cite[Lem.\ 3.2.5]{dhhr}, implies that
		\begin{align}	\label{thm:error_FE.6.3} 
			\rho_{p'_h(\cdot),\Omega}
			(h^{-\widetilde{\gamma}}(\Pi_h^Q q-\textcolor{black}{\Pi}_h^{k-1}q))
			\leq c\, h^{(p^+)'(\gamma-\widetilde{\gamma})}\,.
		\end{align}
		Using \eqref{thm:error_FE.6.3} in  \eqref{thm:error_FE.6.1} and choosing $\tilde{s}>1$ sufficiently close to $1$  such that 
		$$d(1-\tilde{s})+\tilde{s}(p^+)'(\gamma-\widetilde{\gamma})\ge 0\,,$$
		which is possible since $\gamma>\widetilde{\gamma}=\frac{\alpha}{\min\{2,(p^+)'\}}$, we find that
		\begin{align}	\label{thm:error_FE.6.4}
			 \begin{aligned}
		\rho_{p'(\cdot)s,\Omega}
				(h^{-\widetilde{\gamma}}(\Pi_h^Q q-\textcolor{black}{\Pi}_h^{k-1}q))&\leq c\,\big(1+ h^{d(1-\tilde{s})+\tilde{s}(p^+)'(\gamma-\widetilde{\gamma})}\big)\\&\leq 
			 c\,.
			 \end{aligned}
		\end{align}
		Eventually, using \eqref{thm:error_FE.6.4} in   \eqref{thm:error_FE.6}, we arrive at
		\begin{align}\label{thm:error_FE.6.5}
			\begin{aligned}
				\rho_{((\varphi_h)_{\vert\bfD\bfv\vert})^*,\Omega}(\Pi_h^Q q-\textcolor{black}{\Pi}_h^{k-1}q)&\leq  c\,h^{2\alpha}\,\big(1+\rho_{p(\cdot)s,\Omega}(\bfD\bfv)\big)
					\\&\quad+c\,h^{2\beta}\,[\bfF(\cdot,\bfD\bfv)]_{N^{\beta,2}(\Omega)}^2+ c\,\rho_{(\varphi_{\vert \bfD\bfv\vert})^*,\Omega}(h^{\gamma}\vert\nabla^{\gamma}q\vert)\,.
				\end{aligned}
		\end{align}
		Thus, choosing $z_h =\Pi_h^Q q\in Q_h$ in \eqref{thm:error_FE.5} and resorting to \eqref{thm:error_FE.6.5}, we deduce that
		\begin{align}\label{thm:error_FE.5a}
			\begin{aligned}
				\vert J_h^1\vert &\leq c_\varepsilon\,h^{2\alpha}\,\big(1+\rho_{p(\cdot)s,\Omega}(\bfD\bfv)\big)\\&\quad +
				c_\varepsilon\,h^{2\beta}\,[\bfF(\cdot,\bfD\bfv)]_{N^{\beta,2}(\Omega)}^2+ c_\varepsilon\,\rho_{(\varphi_{\vert \bfD\bfv\vert})^*,\Omega}(h^{\gamma}\vert\nabla^{\gamma}q\vert)
				\\&\quad +
				\varepsilon\,c\,
				\|\bfF_h(\cdot,\bfD\bfv)-\bfF_h(\cdot,\bfD\bfv_h)\|_{2,\Omega}^2\,.
			\end{aligned}
		\end{align}
		
		\textit{ad $J_h^2$.} By definition of $b\colon [\Vo]^3\to \mathbb{R}$ (\textit{cf}.\ \eqref{def:bh}), we have that $b(\bfv_h,\Pi_h^V\bfe_h,\Pi_h^V\bfe_h)=0$,~which~yields~that
		\begin{align}\label{thm:error_FE.7}
			\hspace*{-3mm}\begin{aligned}
				J_h^2&= b(\bfv,\bfv,\Pi_h^V\bfe_h)\pm b(\bfv,\Pi_h^V\bfv,\Pi_h^V\bfe_h)\\&\quad-b(\bfv_h,\bfv_h,\Pi_h^V\bfe_h)\pm     b(\bfv_h,\bfv,\Pi_h^V\bfe_h)
				\\
				&= b(\bfv,\bfv-\Pi_h^V\bfv,\Pi_h^V\bfe_h)+ b(\bfv,\Pi_h^V\bfv,\Pi_h^V\bfe_h)\\&\quad-b(\bfv_h,\Pi_h^V\bfe_h,\Pi_h^V\bfe_h)-     b(\bfv_h,\Pi_h^V\bfv,\Pi_h^V\bfe_h)
				\\
				&= b(\bfv,\bfv-\Pi_h^V\bfv,\Pi_h^V\bfe_h)- b(\bfe_h,\Pi_h^V\bfv,\Pi_h^V\bfe_h)
				\\&
				\eqqcolon  J_h^{21}-J_h^{22}\,.
			\end{aligned}
		\end{align}
		First, we have that\enlargethispage{1mm}
		\begin{align}\label{thm:error_FE.8}
			\begin{aligned}
				2\, J_h^{21}&=(\Pi_h^V\bfe_h\otimes \bfv,\nabla\bfv-\nabla\Pi_h^V\bfv)_{\Omega}-((\bfv-\Pi_h^V\bfv)\otimes \bfv,\nabla\Pi_h^V\bfe_h)_{\Omega}\\&\eqqcolon J_{h,1}^{21}+J_{h,2}^{21}\,.
			\end{aligned}
		\end{align}
		Using the generalized Hölder inequality \eqref{eq:gen_hoelder}, the 
		$\varepsilon$-Young inequality \eqref{ineq:young}~with $\smash{\psi=\frac{1}{2}\vert \cdot\vert ^2}$, the Sobolev embedding theorem (\textit{cf}.\  Theorem \ref{thm:sobolev} together with $2r'\leq r^*$ in $\Omega$ for $p^-\ge \frac{3d}{d+2}$), 
		Lemma \ref{lem:stab_Pi_div_V_norm}\eqref{lem:stab_Pi_div_V_norm.2}, Korn's inequality (\textit{cf}.\  Theorem \ref{thm:korn}), Lemma \ref{lem:sobolev2F_discrete}, and~Corollary~\ref{cor:Pi_div_F}\eqref{cor:Pi_div_F.2},~we~find~that
		\begin{align}\label{thm:error_FE.9}
			\begin{aligned}
				\vert J_{h,1}^{21}\vert &\leq \varepsilon\,c\,\|\Pi_h^V\bfe_h\|_{2r'(\cdot),\Omega}^2\|\bfv\|_{2r'(\cdot),\Omega}^2+c_\varepsilon\,\| \nabla\bfv-\nabla\Pi_h^V\bfv\|_{r(\cdot),\Omega}^2
				\\&\leq 
				\varepsilon\,c\,\|\bfD\bfe_h\|_{r(\cdot),\Omega}^2\|\bfD \bfv\|_{r(\cdot),\Omega}^2+c_{\varepsilon}\,\|\bfD\bfv-\bfD\Pi_h^V\bfv\|_{r(\cdot),\Omega}^2
				\\&\leq 
				\varepsilon\,c\,\|\bfF_h(\cdot,\bfD\bfv_h)-\bfF_h(\cdot,\bfD\bfv)\|_{2,\Omega}^2\\&\quad+c_{\varepsilon}\,\big( h^{2\alpha}\,\big(1+\rho_{p(\cdot)s,\Omega}(\bfD\bfv)\big)+h^{2\beta}\,[\bfF(\cdot,\bfD\bfv)]_{N^{\beta,2}(\Omega)}^2\big)\,.
			\end{aligned}
		\end{align}
		Similarly, using the generalized Hölder inequality \eqref{eq:gen_hoelder}, the $\varepsilon$-Young inequality \eqref{ineq:young}~with $\psi=\frac{1}{2}\vert\cdot\vert^2$, the Sobolev embedding theorem (\textit{cf}.\  Theorem \ref{thm:sobolev} together with $2r'\leq r^*$ in $\Omega$ for $p^-\ge \frac{3d}{d+2}$), 
		Lemma~\ref{lem:stab_Pi_div_V_norm} \eqref{lem:stab_Pi_div_V_norm.2}, Korn's inequality (\textit{cf}.\  Theorem~\ref{thm:korn}), Lemma \ref{lem:sobolev2F_discrete}, and Corollary~\ref{cor:Pi_div_F}\eqref{cor:Pi_div_F.2}, we~find~that 
		\begin{align}\label{thm:error_FE.10}
			\begin{aligned}
				\vert J_{h,2}^{21}\vert &\leq c_\varepsilon\, \|\bfv-\Pi_h^V\bfv\|_{2r'(\cdot),\Omega}^2\| \bfv\|_{2r'(\cdot),\Omega}^2+\varepsilon\,c\,\|\smash{\nabla\Pi_h^V\bfe_h}\|_{r(\cdot),\Omega}^2
			%	\\&\leq c_{\varepsilon}\
			%	\|\bfv-\Pi_h^V\bfv\|_{2r'(\cdot),\Omega}^2,\|\bfv\|_{2r'(\cdot),\Omega}^2+ \varepsilon\,c\,\|\bfD\bfe_h\|_{r(\cdot),\Omega}^2
				\\&\leq c_{\varepsilon}\,
				\|\bfD \bfv-\bfD\Pi_h^V\bfv\|_{r(\cdot),\Omega}^2\|\bfD\bfv\|_{r(\cdot),\Omega}^2+ \varepsilon\,c\,\|\bfD\bfe_h\|_{r(\cdot),\Omega}^2
				\\&\leq 
				\varepsilon\,c\,\|\bfF_h(\cdot,\bfD\bfv_h)-\bfF_h(\cdot,\bfD\bfv)\|_{2,\Omega}^2\\&\quad+c_{\varepsilon}\,\big( h^{2\alpha}\,\big(1+\rho_{p(\cdot)s,\Omega}(\bfD\bfv)\big)+h^{2\beta}\,[\bfF(\cdot,\bfD\bfv)]_{N^{\beta,2}(\Omega)}^2\big)\,.
			\end{aligned}
		\end{align}
		Second, we have that
		\begin{align}\label{thm:error_FE.11}
			\begin{aligned}
				2\, J_h^{22}&=\hskp{\Pi_h^V\bfe_h\otimes\bfe_h}{\nabla\Pi_h^V\bfv}_{\Omega}-\hskp{\Pi_h^V\bfv\otimes \bfe_h}{\nabla\Pi_h^V\bfe_h}_{\Omega}\\&\eqqcolon  J_{h,1}^{22}+J_{h,2}^{22}\,.
			\end{aligned}
		\end{align}
		Using the generalized Hölder inequality \eqref{eq:gen_hoelder}, 
		the Sobolev embedding theorem (\textit{cf}.\  Theorem \ref{thm:sobolev} together with $2r'\leq r^*$ in $\Omega$ for $p^-\ge \frac{3d}{d+2}$), Lemma \ref{lem:stab_Pi_div_V_norm}\eqref{lem:stab_Pi_div_V_norm.2}, Korn's inequality (\textit{cf}.\  Theorem~\ref{thm:korn}),  and Lemma~\ref{lem:sobolev2F_discrete},  we~obtain
			\begin{align}\label{thm:error_FE.12}
				\begin{aligned} 
			\vert  J_{h,1}^{22}\vert &\leq c\,\|\Pi_h^V\bfe_h\|_{2r'(\cdot),\Omega}\|\bfe_h\|_{2r'(\cdot),\Omega}\|\nabla\Pi_h^V\bfv\|_{r(\cdot),\Omega}
			\\
			& \leq
			c\, \|\bfD\bfe_h\|_{r(\cdot),\Omega}^2\|\bfD\bfv\|_{r(\cdot),\Omega}
			\\
			&  \leq  
			c_0\,c\,\|\bfF_h(\cdot,\bfD\bfv_h)-\bfF_h(\cdot,\bfD\bfv)\|_{2,\Omega}^2\\&\quad+c\,\big( h^{2\alpha}\,\big(1+\rho_{p(\cdot)s,\Omega}(\bfD\bfv)\big)+h^{2\beta}\,[\bfF(\cdot,\bfD\bfv)]_{N^{\beta,2}(\Omega)}^2\big)\,.
		\end{aligned}
		\end{align}
		Similarly, using the generalized Hölder inequality \eqref{eq:gen_hoelder}, 
		the Sobolev embedding theorem (\textit{cf}.\  Theorem~\ref{thm:sobolev} together with $2r'\leq r^*$ in $\Omega$ for $p^-\ge \frac{3d}{d+2}$), Lemma \ref{lem:stab_Pi_div_V_norm}\eqref{lem:stab_Pi_div_V_norm.2}, Korn's inequality (\textit{cf}.\  Theorem~\ref{thm:korn}), and Lemma \ref{lem:sobolev2F_discrete}, we~obtain
		\begin{align}\label{thm:error_FE.13}
			\begin{aligned}
							\vert J_{h,2}^{22}\vert &\leq c\,\|\Pi_h^V\bfv\|_{2r'(\cdot),\Omega}\|\bfe_h\|_{2r'(\cdot),\Omega}\|\nabla\Pi_h^V\bfe_h\|_{r(\cdot),\Omega}		\\	& \leq
							c\, \|\bfD\bfv\|_{r(\cdot),\Omega}\|\bfD\bfe_h\|_{r(\cdot),\Omega}^2\\
							&  \leq  
								c_0\,c\,\|\bfF_h(\cdot,\bfD\bfv_h)-\bfF_h(\cdot,\bfD\bfv)\|_{2,\Omega}^2\\&\quad+c\,\big( h^{2\alpha}\,\big(1+\rho_{p(\cdot)s,\Omega}(\bfD\bfv)\big)+h^{2\beta}\,[\bfF(\cdot,\bfD\bfv)]_{N^{\beta,2}(\Omega)}^2\big)\,.
			\end{aligned}
		\end{align}
		Using \eqref{thm:error_FE.9} and \eqref{thm:error_FE.10} in
		\eqref{thm:error_FE.8} as well as  \eqref{thm:error_FE.12} and 
		\eqref{thm:error_FE.13} in \eqref{thm:error_FE.11}, from \eqref{thm:error_FE.7}, it follows that 
		\begin{align}\label{thm:error_FE.13.2}
			\begin{aligned}
				\vert J_h^2\vert &\leq 
				(\varepsilon+c_0)\,c\,\|\bfF_h(\cdot,\bfD\bfv_h)-\bfF_h(\cdot,\bfD\bfv)\|_{2,\Omega}^2
				\\&\quad+c_{\varepsilon} \,\big(h^{2\alpha}\,\big(1+\rho_{p(\cdot)s,\Omega}(\bfD\bfv)\big)+h^{2\beta}\,
				[\bfF(\cdot,\bfD\bfv)]_{N^{\beta,2}(\Omega)}^2\big)
				\,.
			\end{aligned}
		\end{align}
		Using, in turn, \eqref{thm:error_FE.5a} and \eqref{thm:error_FE.13.2} in \eqref{thm:error_FE.4}, we obtain\enlargethispage{9mm}
		\begin{align}\label{thm:error_FE.16.0}
			\begin{aligned}
				\vert I_h^1\vert 
				&\le c_{\varepsilon} \,\big(h^{2\alpha}\,\big(1+\rho_{p(\cdot)s,\Omega}(\bfD\bfv)\big)
				\\&\quad+h^{2\beta}\,[\bfF(\cdot,\bfD\bfv)]_{N^{\beta,2}(\Omega)}^2+ \rho_{(\varphi_{\vert \bfD\bfv\vert})^*,\Omega}(h^{\gamma}\vert\nabla^{\gamma}q\vert)\big)\\&\quad+ (\varepsilon+c_0)\,c\,\|\bfF_h(\cdot,\bfD\bfv_h)-\bfF_h(\cdot,\bfD\bfv)\|_{2,\Omega}^2 \,.
			\end{aligned}
		\end{align}
		
		Putting everything together, \textit{i.e.},
		\eqref{thm:error_FE.2}, \eqref{thm:error_FE.3}, \eqref{thm:error_FE.3.1}, and
		\eqref{thm:error_FE.16.0} in \eqref{thm:error_FE.1},  we arrive at
		\begin{align}\label{thm:error_FE.16}
			\begin{aligned}
				\|\bfF_h(\cdot,\bfD \bfv_h) - \bfF(\cdot,\bfD
				\bfv)\|_{2,\Omega}^2 
				&\le c\, (\varepsilon+c_0)\,\|\bfF_h(\cdot,\bfD\bfv_h)-\bfF_h(\cdot,\bfD\bfv)\|_{2,\Omega}^2
				\\
				&\quad + c_{\varepsilon} \,\big(h^{2\alpha}\,\big(1+\rho_{p(\cdot)s,\Omega}(\bfD\bfv)\big)
				\\&\quad+h^{2\beta}\,[\bfF(\cdot,\bfD\bfv)]_{N^{\beta,2}(\Omega)}^2+ \rho_{(\varphi_{\vert \bfD\bfv\vert})^*,\Omega}(h^{\gamma}\vert\nabla^{\gamma}q\vert)\big)\,.
			\end{aligned}
		\end{align}
		
		Eventually, choosing  $c_0>0$ and $\varepsilon>0$ sufficiently small,  we can absorb the first term~on~the~right-hand side of \eqref{thm:error_FE.16} in the left-hand side and conclude 
		the existence of a
		constant $c>0$, depending~only~on  $d$, $k$, $\textcolor{black}{\ell}$,  $p^- $, $p^+$, $[p]_{\alpha,\Omega}$, $\delta^{-1}$, $\omega_0$, $s$, $c_0$, and $\|q\|_{\gamma,p'(\cdot),\Omega}$, such that
		\begin{align}\label{thm:error_FE.17}
			\begin{aligned}
				\|\bfF_h(\cdot,\bfD \bfv_h)-\bfF_h(\cdot,\bfD
				\bfv)\|_{2,\Omega}^2&\leq c\,h^{2\alpha}\,\big(1+\rho_{p(\cdot)s,\Omega}(\bfD\bfv)\big)
				\\&\quad+c\,h^{2\beta}\,[\bfF(\cdot,\bfD\bfv)]_{N^{\beta,2}(\Omega)}^2+ c\,\rho_{(\varphi_{\vert \bfD\bfv\vert})^*,\Omega}(h^{\gamma}\vert\nabla^{\gamma}q\vert) \,,
			\end{aligned}
		\end{align}
		which is the claimed \textit{a priori} error estimate for the velocity vector field.\enlargethispage{2mm}
		
		\textit{2. Error estimates for the pressure:} We employ the abbreviation
		\begin{align*}
			\Xi^2_h(\bfv,q)\coloneqq h^{2\alpha}\,\big(1+\rho_{p(\cdot)s,\Omega}(\bfD\bfv)\big)
		+h^{2\beta}\,[\bfF(\cdot,\bfD\bfv)]_{N^{\beta,2}(\Omega)}^2+ \rho_{(\varphi_{\vert \bfD\bfv\vert})^*,\Omega}(h^{\gamma}\vert\nabla^{\gamma}q\vert) \,. 
		\end{align*}
		Due to Lemma \ref{lem:ismd}, there exists a constant $c>0$ such
		that for every $z_h\in \Qo_h$, it holds that
		\begin{align}
			\label{thm:error_pressure_FE.1}
			c\,\|z_h\|_{p'_h(\cdot),\Omega}\leq \sup_{\bfz_h\in \Vo_h\,:\,{\|\nabla\bfz_h\|_{p_h(\cdot),\Omega}\leq 1}}{( z_h,\divo\bfz_h)_{\Omega}}\,.
		\end{align}
		On the other hand,  testing the first lines of Problem (\hyperlink{Q}{Q}) and Problem (\hyperlink{Qh}{Q$_h$}) with an arbitrary $\bfz_h\in \Vo_h$~and, then, subtracting the resulting equations, 
		for every $\bfz_h\in \smash{\Vo_h}$, we find that
		\begin{align}\label{thm:error_pressure_FE.2}
			\begin{aligned}
				(q_h-q,\divo \bfz_h)_{\Omega}
				&=(\bfS_h (\cdot,\bfD \bfv_h) - \bfS_h (\cdot,\bfD
				\bfv),\bfD \bfz_h)_{\Omega}
				\\&\quad+ (\bfS_h (\cdot,\bfD \bfv) - \bfS (\cdot,\bfD
				\bfv),\bfD \bfz_h)_{\Omega}
				\\&\quad+[b(\bfv_h,\bfv_h,\bfz_h)-b(\bfv,\bfv,\bfz_h)]
				\\&\eqqcolon I_h^1+I_h^2+I_h^3\,.
			\end{aligned}
		\end{align}
		
		So, let us next estimate $I_h^1$, $I_h^2$, and $I_h^3$ for an arbitrary $\bfz_h\in \smash{\Vo_h}$ with $\|\nabla\bfz_h\|_{p_h(\cdot),\Omega}\leq 1$:
		
		\textit{ad $I_h^1$}.
		Using the generalized Hölder inequality \eqref{eq:gen_hoelder} with $\psi=(\varphi_h)_{\smash{\vert\bfD\bfv\vert}}$, we find that
		\begin{align}\label{thm:error_pressure_FE.3}
			\begin{aligned}
				\vert I_h^1\vert &\leq 2\,\|\bfS_h(\cdot,\bfD\bfv_h)-\bfS_h(\cdot,\bfD\bfv)\|_{\smash{((\varphi_h)_{\smash{\vert\bfD\bfv\vert}})^*}}\|\bfD\bfz_h\|_{\smash{(\varphi_h)_{\smash{\vert \bfD\bfv\vert}}}}
				\\&\eqqcolon  2\,I_h^{11}\times I_h^{12}\,.
			\end{aligned}
		\end{align}
		Appealing to \eqref{eq:hammera} and \eqref{thm:error_FE.17}, we have that 
		\begin{align}\label{thm:error_pressure_FE.4}
			\begin{aligned}
				\rho_{((\varphi_h)_{\smash{\vert \bfD\bfv\vert}})^*,\Omega}(\bfS_h(\cdot,\bfD\bfv_h)-\bfS_h(\cdot,\bfD\bfv))&\leq c\,\|\bfF_h(\cdot,\bfD\bfv_h)-\bfF_h(\cdot,\bfD\bfv)\|_{2,\Omega}^2
				\\&
				\leq c\,\Xi^2_h(\bfv,q)\,.
			\end{aligned}
		\end{align}
		%For $c_0\coloneqq \max\{1,c\,\Xi^2(\bfv,q;1)\}\ge 1$ and $ \gamma\coloneqq c_0^{-1}c\,\Xi^2_h(\bfv,q)\leq 1$, 
		Thus, Lemma \ref{lem:just_to_mighty}(ii) applied to \eqref{thm:error_pressure_FE.4} yields a constant $c>0$ such that 
		\begin{align}\label{thm:error_pressure_FE.5}
			\begin{aligned}
				\smash{I_h^{11}\leq c\,\Xi^2_h(\bfv,q)^{\smash{\frac{1}{(r^-)'}}}}\,.
			\end{aligned}
		\end{align}
		Resorting to  the shift change Lemma \ref{lem:shift-change}\eqref{lem:shift-change.1}, that $\varphi(x,t)\sim \delta^{p(x)}+t^{p(x)}$ for all $t\ge 0$ and $x\in \Omega$, and 
		$\rho_{p(\cdot),\Omega}(\bfD\bfz_h)\hspace*{-0.1em} \leq \hspace*{-0.1em}  1+{\|\bfD\bfz_h\|_{p(\cdot),\Omega}^{p^+}}\hspace*{-0.1em} \leq\hspace*{-0.1em}  c\,(1+{\|\bfD\bfz_h\|_{p_h(\cdot),\Omega}^{p^+}})\hspace*{-0.1em} \leq\hspace*{-0.1em} c$ (\textit{cf}.\  \cite[Lem.\  3.2.5]{dhhr} and Lemma~\ref{lem:norm_equiv}),~we~obtain
		\begin{align}\label{thm:error.6}
			\begin{aligned}
				\rho_{\varphi_{\vert \bfD\bfv\vert},\Omega}(\bfD\bfz_h)&\leq c\,\rho_{\varphi,\Omega}(\bfD\bfz_h)+c\,\rho_{\varphi,\Omega}(\bfD\bfv)
				\\&\leq c\,\rho_{p(\cdot),\Omega}(\bfD\bfz_h)+c\,\rho_{p(\cdot),\Omega}(\bfD\bfv)+c\,(\max\{1,\delta\})^{\smash{p^+}}\vert \Omega\vert\\&\leq c\,.
			\end{aligned}
		\end{align}
		Thus, Lemma \ref{lem:just_to_mighty}(i) applied to \eqref{thm:error.6}  
		yields a constant $c>0$   such that
		\begin{align}\label{thm:error.7}
			\smash{I_h^{12}\leq c\,.}
		\end{align}
		Then, combining \eqref{thm:error_pressure_FE.5} and \eqref{thm:error.7} in \eqref{thm:error_pressure_FE.3}, we arrive that
		\begin{align}\label{thm:error.I1}
			\smash{\vert I_h^1\vert 
				\leq c\,\Xi^2_h(\bfv,q)^{\smash{\frac{1}{(r^-)'}}}\,.}
		\end{align}
		
		\textit{ad $I_h^2$}. Using the generalized Hölder inequality \eqref{eq:gen_hoelder} with $\psi=(\varphi_h)_{\smash{\vert\bfD\bfv\vert}}$, we find that\enlargethispage{3.5mm}
		\begin{align}\label{thm:error_pressure_FE.6}
			\begin{aligned}
				\vert I_h^2\vert &\leq 2\,\|\bfS_h(\cdot,\bfD\bfv)-\bfS(\cdot,\bfD\bfv)\|_{\smash{((\varphi_h)_{\smash{\vert\bfD\bfv\vert}})^*}}\|\bfD\bfz_h\|_{\smash{(\varphi_h)_{\smash{\vert \bfD\bfv\vert}}}}\\&\eqqcolon  2\,I_h^{21}\times I_h^{22}\,.
			\end{aligned}
		\end{align}
		Appealing to \eqref{eq:hammera} and Lemma \ref{lem:A-Ah}\eqref{eq:Ah-A}, we have that 
		\begin{align}\label{thm:error_pressure_FE.7}
			\begin{aligned}
				\rho_{((\varphi_h)_{\smash{\vert \bfD\bfv\vert}})^*,\Omega}(\bfS_h(\cdot,\bfD\bfv)-\bfS(\cdot,\bfD\bfv))&\leq c\,\|\bfF_h^*(\cdot,\bfS_h(\cdot,\bfD\bfv))-\bfF_h^*(\cdot,\bfS(\cdot,\bfD\bfv))\|_{2,\Omega}^2
				\\&
				\leq c\,h^{2\alpha}\,\big(1+\rho_{p(\cdot)s,\Omega}(\bfD\bfv)\big)
					\\&
				\leq c\,\Xi^2_h(\bfv,q)\,.
			\end{aligned}
		\end{align}
		Thus, Lemma \ref{lem:just_to_mighty}(ii) applied to \eqref{thm:error_pressure_FE.7} yields a constant $c>0$  such that
		\begin{align}\label{thm:error_pressure_FE.8}
			\begin{aligned}
				\smash{\vert I_h^{21}\vert \leq c\,\Xi^2_h(\bfv,q)^{\smash{\frac{1}{(r^-)'}}}}\,.
			\end{aligned}
		\end{align}
		Then, 
		combining \eqref{thm:error_pressure_FE.8} and \eqref{thm:error.7} in \eqref{thm:error_pressure_FE.6}, we arrive at\enlargethispage{7mm}
		\begin{align}\label{thm:error.I2}
			\smash{\vert I_h^2\vert 
				\leq c\,\Xi^2_h(\bfv,q)^{\smash{\frac{1}{(r^-)'}}}}\,.
		\end{align}
		
		\textit{ad $I_h^3$.} Using, again, the abbreviation $\bfe_h\coloneqq  \bfv_h-\bfv\in \smash{\Vo}$, the term $I_h^3$ can be re-written~as
		\begin{align}\label{eq:i3_pressure_FE}
			\begin{aligned}
				I_h^3 &= -b(\bfv,\bfv-\Pi_h^V\bfv,\bfz_h)+b(\bfe_h,\Pi_h^V\bfv,\bfz_h)+b(\bfv_h,\Pi_h^V\bfe_h,\bfz_h)
				\\&\eqqcolon  I_h^{31}+I_h^{32}+I_h^{33}\,.
			\end{aligned}
		\end{align}
		So, let us next estimate $ I_h^{3i}$, $i=1,2,3$:
		
		\textit{ad $I_h^{31}$}.  The definition of $b\colon [\Vo]^3\to \mathbb{R}$ (\textit{cf}.\ \eqref{def:bh})
		yields that
		\begin{align}\label{eq:i31_pressure_FE}
			\begin{aligned}
				2\,I_h^{31}&=-(\bfz_h\otimes \bfv,\nabla\bfv-\nabla\Pi_h^V\bfv)_{\Omega}+((\bfv-\Pi_h^V\bfv)\otimes \bfv,\nabla\bfz_h)_{\Omega}
				\\&\eqqcolon  I_{h,1}^{31}+ I_{h,2}^{31}\,.
			\end{aligned}
		\end{align}
		Using the generalized Hölder inequality \eqref{eq:gen_hoelder}, the Sobolev embedding theorem (\textit{cf}.\  Theorem~\ref{thm:sobolev}~together with $2r'\leq r^*$ in $\Omega$ for $p^-\ge \frac{3d}{d+2}$),  Korn's inequality (\textit{cf}.\  Theorem \ref{thm:korn}), that $\|\nabla\bfz_h\|_{r(\cdot),\Omega}\leq 2\,(1+\vert \Omega\vert )$ $\times\|\nabla\bfz_h\|_{p(\cdot),\Omega}$
		(since $r\leq p$ in $\Omega$, \textit{cf}.\ \cite[Cor.\ 3.3.4]{dhhr}), Lemma \ref{lem:norm_equiv}, and Lemma~\ref{cor:Pi_div_F}\eqref{cor:Pi_div_F.2},~we~find~that
		\begin{align}
			\label{eq:i311_pressure_FE}
			\begin{aligned}
				\vert I_{h,1}^{31}\vert &\leq c\,\|\bfz_h\|_{2r'(\cdot),\Omega}\|\bfv\|_{2r'(\cdot),\Omega}\|\nabla\bfv-\nabla\Pi_h^V\bfv\|_{r(\cdot),\Omega}\\
				&\leq c\,\|\nabla\bfz_h\|_{r(\cdot),\Omega}\|\bfD\bfv\|_{r(\cdot),\Omega}\|\bfD\bfv-\bfD\Pi_h^V\bfv\|_{r(\cdot),\Omega}
				\\
				&\leq c\,\|\nabla\bfz_h\|_{p_h(\cdot),\Omega}\|\bfD\bfv\|_{r(\cdot),\Omega}\big( h^{2\alpha}\,\big(1+\rho_{p(\cdot)s,\Omega}(\bfD\bfv)\big)+h^{\beta}\,[\bfF(\cdot,\bfD\bfv)]_{N^{\beta,2}(\Omega)}^2\big)^{\smash{\frac{1}{2}}}
				\\
				&\leq c\,\Xi_h^2(\bfv,q)^{\smash{\frac{1}{2}}}\,.
			\end{aligned}
		\end{align}
		Similarly, using the generalized Hölder inequality \eqref{eq:gen_hoelder}, the Sobolev embedding theorem (\textit{cf}.\  Theorem~\ref{thm:sobolev} together \hspace*{-0.1mm}with \hspace*{-0.1mm}$2r'\hspace*{-0.15em}\leq\hspace*{-0.15em}  r^*$ \hspace*{-0.1mm}in \hspace*{-0.1mm}$\Omega$ \hspace*{-0.1mm}for \hspace*{-0.1mm}$p^-\hspace*{-0.15em}\ge\hspace*{-0.15em} \frac{3d}{d+2}$), \hspace*{-0.1mm}Korn's \hspace*{-0.1mm}inequality \hspace*{-0.1mm}(\textit{cf}.\  \hspace*{-0.1mm}Theorem~\hspace*{-0.1mm}\ref{thm:korn}),~\hspace*{-0.1mm}that~\hspace*{-0.1mm}${\|\nabla\bfz_h\|_{r(\cdot),\Omega}\hspace*{-0.15em}\leq\hspace*{-0.15em} 2\,(1\hspace*{-0.15em}+\hspace*{-0.15em}\vert \Omega\vert )}$ $\times\|\nabla\bfz_h\|_{p(\cdot),\Omega}$
		(since $r\leq p$ in $\Omega$, \textit{cf}.\ \cite[Cor.\ 3.3.4]{dhhr}), Lemma \ref{lem:norm_equiv}, and Lemma~\ref{cor:Pi_div_F}\eqref{cor:Pi_div_F.2},~we~find~that
		\begin{align}
			\label{eq:i312_pressure_FE}
			\begin{aligned}
				\vert I_{h,2}^{31}\vert &\leq c\,\|\bfv-\Pi_h^V\bfv\|_{2r'(\cdot),\Omega}\|\bfv\|_{2r'(\cdot),\Omega}\|\nabla\bfz_h\|_{r(\cdot),\Omega}\\
				&\leq c\,\|\bfD\bfv-\bfD\Pi_h^V\bfv\|_{r(\cdot),\Omega}\|\bfD\bfv\|_{r(\cdot),\Omega}\|\nabla\bfz_h\|_{p_h(\cdot),\Omega}\\
				&\leq c\,\|\bfD\bfv\|_{r(\cdot),\Omega}\|\nabla\bfz_h\|_{r_h(\cdot),\Omega}\big( h^{2\alpha}\,\big(1+\rho_{p(\cdot)s,\Omega}(\bfD\bfv)\big)+h^{\beta}\,[\bfF(\cdot,\bfD\bfv)]_{N^{\beta,2}(\Omega)}^2\big)^{\smash{\frac{1}{2}}}
				\\
				&\leq c\,\Xi_h^2(\bfv,q)^{\smash{\frac{1}{2}}}\,.
			\end{aligned}
		\end{align}
		
		\textit{ad $I_h^{32}$}. The definition of $b\colon [\Vo]^3\to \mathbb{R}$ (\textit{cf}.\ \eqref{def:bh})
		yields that
		\begin{align}
			\label{eq:i32_pressure_FE}
			\begin{aligned}
				2\,I_h^{32}&=(\bfz_h\otimes \bfe_h,\nabla\Pi_h^V\bfv)_{\Omega}-(\Pi_h^V\bfv\otimes \bfe_h,\nabla\bfz_h)_{\Omega}
				\\&	\eqqcolon  I_{h,1}^{32}+ I_{h,2}^{32}\,.
			\end{aligned}
		\end{align}
		Using the generalized Hölder inequality \eqref{eq:gen_hoelder}, the Sobolev embedding theorem (\textit{cf}.\  Theorem~\ref{thm:sobolev}~together with $2r'\leq r^*$ in $\Omega$ for $p^-\ge \frac{3d}{d+2}$), Lemma \ref{lem:stab_Pi_div_V_norm}\eqref{lem:stab_Pi_div_V_norm.2}, 
		Korn's inequality (\textit{cf}.\  Theorem \ref{thm:korn}), that $\|\nabla\bfz_h\|_{r(\cdot),\Omega}\leq 2\,(1+\vert \Omega\vert )\,\|\nabla\bfz_h\|_{p(\cdot),\Omega}$
		(since $r\leq p$ in $\Omega$, \textit{cf}.\ \cite[Cor.\ 3.3.4]{dhhr}), Lemma \ref{lem:norm_equiv}, Lemma~\ref{lem:sobolev2F_discrete}, and \eqref{thm:error_pressure_FE.2},  we find that\enlargethispage{3.5mm}
		\begin{align}
			\label{eq:i321_pressure_FE}
			\begin{aligned}
				\vert I_{h,1}^{32}\vert &\leq
				c\,\|\bfz_h\|_{2r'(\cdot),\Omega}\|\bfe_h\|_{2r'(\cdot),\Omega}	\|\nabla\Pi_h^V\bfv\|_{r(\cdot),\Omega}
				\\
				&\leq
				c\,\|\nabla\bfz_h\|_{p_h(\cdot),\Omega}\|\bfD\bfe_h\|_{r(\cdot),\Omega}\|\bfD\bfv\|_{r(\cdot),\Omega}
				\\
				&\leq c\,\Xi_h^2(\bfv,q)^{\smash{\frac{1}{2}}}\,.
			\end{aligned}
		\end{align}
		Similarly, using the generalized Hölder inequality \eqref{eq:gen_hoelder}, the Sobolev embedding theorem (\textit{cf}.\  Theorem~\ref{thm:sobolev} together with $2r'\leq r^*$ in $\Omega$ for $p^-\ge \frac{3d}{d+2}$), Lemma \ref{lem:stab_Pi_div_V_norm}\eqref{lem:stab_Pi_div_V_norm.2}, 
		Korn's inequality (\textit{cf}.\ Theorem~\ref{thm:korn}), that $\|\nabla\bfz_h\|_{r(\cdot),\Omega}\leq 2\,(1+\vert \Omega\vert )\,\|\nabla\bfz_h\|_{p(\cdot),\Omega}$
		(since $r\leq p$ in $\Omega$, \textit{cf}.\ \cite[Cor.\ 3.3.4]{dhhr}), Lemma \ref{lem:norm_equiv}, Lemma~\ref{lem:sobolev2F_discrete}, and \eqref{thm:error_pressure_FE.2},  we find that
		\begin{align}
			\label{eq:i322_pressure_FE}
			\begin{aligned}
				\vert I_{h,2}^{32}\vert &\leq c\,\|\Pi_h^V\bfv\|_{2r'(\cdot),\Omega}\|\bfe_h\|_{2r'(\cdot),\Omega}\|\nabla\bfz_h\|_{r(\cdot),\Omega}\\
				&\leq c\,\|\bfD\bfv\|_{r(\cdot),\Omega}\|\bfD \bfe_h\|_{r(\cdot),\Omega}\|\nabla\bfz_h\|_{p_h(\cdot),\Omega}\\
				&\leq c\,\Xi_h^2(\bfv,q)^{\smash{\frac{1}{2}}}\,.
			\end{aligned}
		\end{align}

		\textit{ad $I_h^{33}$}.  The definition of $b\colon [\Vo]^3\to \mathbb{R}$ (\textit{cf}.\ \eqref{def:bh})
		yields that
		\begin{align}
			\label{eq:i33_pressure_FE}
			\begin{aligned}
				2\,I_h^{33}&=(\bfz_h\otimes \bfv_h,\nabla\Pi_h^V\bfe_h)_{\Omega}-(\Pi_h^V\bfe_h\otimes \bfv_h,\nabla\bfz_h)_{\Omega}
				\\&\eqqcolon  I_{h,1}^{33}+ I_{h,2}^{33}\,.
			\end{aligned}
		\end{align}
		Using the generalized Hölder inequality \eqref{eq:gen_hoelder}, the Sobolev embedding theorem (\textit{cf}.\  Theorem \ref{thm:sobolev} together with $2r'\leq r^*$ in $\Omega$ for $p^-\ge \frac{3d}{d+2}$),  Lemma \ref{lem:stab_Pi_div_V_norm}\eqref{lem:stab_Pi_div_V_norm.2}, 
		Korn's inequality (\textit{cf}.\  Theorem~\ref{thm:korn}), that $\|\nabla\bfz_h\|_{r(\cdot),\Omega}\leq 2\,(1+\vert \Omega\vert )\,\|\nabla\bfz_h\|_{p(\cdot),\Omega}$
		(since $r\leq p$ in $\Omega$, \textit{cf}.\ \cite[Cor.\ 3.3.4]{dhhr}), Lemma \ref{lem:norm_equiv}, Lemma~\ref{lem:sobolev2F_discrete}, and \eqref{thm:error_pressure_FE.2},  we find that
		\begin{align}
			\label{eq:i331_pressure_FE}
			\begin{aligned}
				\vert I_{h,1}^{33}\vert &\leq c\,
				\|\bfz_h\|_{2r'(\cdot),\Omega}\|\bfv_h\|_{2r'(\cdot),\Omega}\|\nabla\Pi_h^V\bfe_h\|_{r(\cdot),\Omega}\\
				&\leq
				c\,\|\nabla\bfz_h\|_{p_h(\cdot),\Omega}\|\bfD\bfv_h\|_{r(\cdot),\Omega}\|\bfD\bfe_h\|_{r(\cdot),\Omega}\\
				&\leq c\,\Xi_h^2(\bfv,q)^{\smash{\frac{1}{2}}}\,.
			\end{aligned}
		\end{align}
		Similarly, using  the generalized Hölder inequality \eqref{eq:gen_hoelder}, the Sobolev embedding theorem (\textit{cf}.\ Theorem~\ref{thm:sobolev} together with $2r'\leq  r^*$ in $\Omega$ for $p^-\ge \frac{3d}{d+2}$),  Lemma \ref{lem:stab_Pi_div_V_norm}\eqref{lem:stab_Pi_div_V_norm.2}, 
		Korn's inequality (\textit{cf}.\  Theorem~\ref{thm:korn}), that $\|\nabla\bfz_h\|_{r(\cdot),\Omega}\leq 2\,(1+\vert \Omega\vert )\,\|\nabla\bfz_h\|_{p(\cdot),\Omega}$
		(since $r\leq p$ in $\Omega$, \textit{cf}.\ \cite[Cor.\ 3.3.4]{dhhr}), Lemma \ref{lem:norm_equiv}, Lemma~\ref{lem:sobolev2F_discrete}, and \eqref{thm:error_pressure_FE.2}, we find that
		\begin{align}
			\label{eq:i332_pressure_FE}
			\begin{aligned}
				\vert I_{h,2}^{33}\vert &\leq c\, \|\Pi_h^V\bfe_h\|_{2r'(\cdot),\Omega}\|\bfv_h\|_{2r'(\cdot),\Omega}\|\nabla\bfz_h\|_{r(\cdot),\Omega}\\
				&\leq c\,\|\bfD\bfe_h\|_{r(\cdot),\Omega}\|\bfD\bfv_h\|_{r(\cdot),\Omega}\|\nabla\bfz_h\|_{p_h(\cdot),\Omega}\\
				&\leq c\,\Xi_h^2(\bfv,q)^{\smash{\frac{1}{2}}}\,.
			\end{aligned}
		\end{align}
		Eventually, combining \eqref{eq:i3_pressure_FE}--\eqref{eq:i332_pressure_FE}, we conclude that\vspace*{-0.5mm}
		\begin{align}
			\vert I_h^3\vert \leq 
			c\,\Xi_h^2(\bfv,q)^{\smash{\frac{1}{2}}}\,.\label{eq:i3fin_pressure_FE}
		\end{align}

		Putting it all together,  \textit{i.e.}, \eqref{thm:error.I1}, \eqref{thm:error.I2}, and \eqref{eq:i3fin_pressure_FE} in \eqref{thm:error_pressure_FE.2},  for every $\bfz_h\hspace*{-0.1em}\in\hspace*{-0.1em} \Vo_h$ with ${\|\nabla\bfz_h\|_{p_h(\cdot),\Omega}\hspace*{-0.1em}\leq\hspace*{-0.1em} 1}$, we conclude that\vspace*{-1.5mm}
		\begin{align}\label{eq:i3fin_pressure_FE.2}
			(q_h-q,\divo\bfz_h)_{\Omega}\leq  \Xi^2_h(\bfv,q)^{\smash{\frac{1}{(r^-)'}}}\,.
		\end{align}
		Therefore, using Lemma \ref{lem:norm_equiv}, \eqref{thm:error_pressure_FE.1}, and \eqref{eq:i3fin_pressure_FE.2}, for every $z_h\in\smash{ \Qo_h}$, we find that
		\begin{align}\label{eq:fast_pressure_FE}
			\begin{aligned}
				\|q_h-q\|_{p'(\cdot),\Omega}%&\leq \|q_h-z_h\|_{p'(\cdot),\Omega}+\|z_h-q\|_{p'(\cdot),\Omega} \\
				&\leq c\,\|q_h-z_h\|_{p'_h(\cdot),\Omega}+\|z_h-q\|_{p'(\cdot),\Omega} \\[-0.5mm]&\leq
				c\,\sup_{\bfz_h\in \Vo_h\,:\, {\|\nabla\bfz_h\|_{p_h(\cdot),\Omega}\leq
						1}}{(q_h-z_h,\divo\bfz_h)_{\Omega}+\|z_h-q\|_{p'(\cdot),\Omega}} \\[-0.5mm]&\leq
				c\,	\sup_{\bfz_h\in \Vo_h\,:\,  {\|\nabla\bfz_h\|_{p_h(\cdot),\Omega}\leq
						1}}{(q_h-q,\divo\bfz_h)_{\Omega}+c\,\|z_h-q\|_{p'(\cdot),\Omega}} \\[-0.5mm]&\leq
				c\,\Xi^2_h(\bfv,q)^{\smash{\frac{1}{(r^-)'}}}+c\,\|z_h-q\|_{p'(\cdot),\Omega}\,.
			\end{aligned}
		\end{align}
		Due to $\langle q\rangle_{\Omega}=0$,  $\| \cdot\|_{1,\Omega}\leq 2\,(1+\vert \Omega\vert )\,\|\cdot\|_{p'(\cdot),\Omega}$
		(\textit{cf}.\ \cite[Cor.\ 3.3.4]{dhhr}), and Lemma \ref{lem:Pi_Q_norm}\eqref{lem:Pi_Q_norm.global},~we~have~that\vspace*{-0.5mm}\enlargethispage{6.5mm}
		\begin{align}\label{eq:fast_pressure_FE.3}
			\begin{aligned}
				\|q-(\smash{\Pi_h^Q } q-\langle\smash{\Pi_h^Q }q\rangle_{\Omega})\|_{p'(\cdot),\Omega}&= \|q-\smash{\Pi_h^Q } q-\langle q-\smash{\Pi_h^Q } q\rangle_{\Omega}\|_{p'(\cdot),\Omega}
				\\&\leq \|q-\smash{\Pi_h^Q } q\|_{p'(\cdot),\Omega}+\vert\Omega\vert^{\smash{-1}}\| q-\smash{\Pi_h^Q }q\|_{1,\Omega}\|1\|_{p'(\cdot),\Omega}
			%	\\&\leq (1+\vert\Omega\vert^{-1}2(1+\vert \Omega\vert )\|1\|_{p'(\cdot),\Omega})\|q-\Pi_h^Q q\|_{p'(\cdot),\Omega}
				\\&\leq c\,h^{\gamma} \|\vert \nabla^{\gamma} q\vert\|_{p'(\cdot),\Omega}\,.
			\end{aligned}
		\end{align}
		Therefore, choosing $z_h\coloneqq \Pi_h^Q q-\langle\Pi_h^Q q\rangle_{\Omega}\in \Qo_h$ in \eqref{eq:fast_pressure_FE}, using \eqref{eq:fast_pressure_FE.3} in doing so, we conclude that
		\begin{align*}
			\smash{\|q_h-q\|_{p'(\cdot),\Omega}^{(r^-)'}\leq c\,\Xi^2_h(\bfv,q)+c\,h^{(r^-)'\gamma}\, \|\vert \nabla^{\gamma} q\vert\|_{p'(\cdot),\Omega}^{(r^-)'}\,,}
		\end{align*}
		which is the claimed error estimate for the kinematic pressure.
	\end{proof}
	
	\begin{proof}[Proof (of Corollary \ref{cor:error_FE}).]
		
		\textit{ad \eqref{cor:error_FE.1}.}
		Using that $(\varphi_a)^*(x,h\,t) \lesssim
		h^{\smash{\min\{2,p'(x)\}}} (\varphi_a)^*(x,t)$ for all $t,a\ge 0$, $h\in (0,1]$, and $x\in \Omega$,  we deduce that\vspace*{-0.5mm}
		\begin{align*}
			\smash{\rho_{(\varphi_{\vert \bfD\bfv\vert})^*,\Omega}(h^{\gamma}\vert\nabla^{\gamma}q\vert) \lesssim h^{\smash{\min\{2,(p^+)'\}\gamma}}\rho_{(\varphi_{\vert \bfD\bfv\vert})^*,\Omega}(\vert\nabla^{\gamma}q\vert) \,,}
		\end{align*}
		so that from Theorem \ref{thm:error_FE}, it follows the claimed \textit{a priori} error estimate \eqref{cor:error_FE.1}.
		
		\textit{ad \eqref{cor:error_FE.2}.} 
		Using that 	$(\varphi_a)^*(x,h\,t)\sim ( (\delta+a)^{p(x)-1}
		+h\,t\big )^{p'(x)-2}\, h^2\, t^2	\le (\delta+a)^{2-p(x)}\, h^2\, t^2 $~for~all~${t,a\ge 0}$, $x\in \Omega$, and $h\in (0,1]$ (\textit{cf}.\ \eqref{rem:phi_a.2}), due to $p^-\ge 2$, we deduce  that
		\begin{align*}
			\smash{\rho_{(\varphi_{\vert \bfD\bfv\vert})^*,\Omega}(h^{\gamma}\vert\nabla^{\gamma}q\vert) \lesssim h^{\smash{2\gamma}}\,
			\|(\delta+\vert \bfD\bfv\vert )^{\smash{\frac{2-p(\cdot)}{2}}}\vert \nabla^{\gamma} q\vert\|_{2,\Omega}^2 	\big )\,,}
		\end{align*}
		so that from Theorem \ref{thm:error_FE}, it follows the claimed \textit{a priori} error estimate \eqref{cor:error_FE.2}.
	\end{proof}\newpage
	
	\section{Numerical experiments}\label{sec:experiments}

	\hspace*{5mm}In this section, we complement the theoretical findings  of Section \ref{sec:a_priori} via numerical experiments:~first, we examine a purely academic example that is only meant to confirm the quasi-optimality of the derived error decay rates (for the velocity vector field); second, we consider a less academic example.%\enlargethispage{10mm}

	\subsection{Implementation details}
	
	\hspace*{5mm}All experiments were conducted deploying the finite element software package \texttt{FEniCS} (version 2019.1.0), \textit{cf}.\  \cite{LW10}. All triangulations were generated using  the \texttt{MatLab} (version R2022b, \textit{cf}.\  \cite{matlab}) library \texttt{DistMesh} (version 1.1, \textit{cf}.\  \cite{distmesh}).
	All graphics were generated using the \texttt{Matplotlib} library~(version~3.5.1, \textit{cf}.~\cite{Hun07}). The communication between \texttt{FEniCS} (\textit{i.e.}, \texttt{Numpy} (version 1.24.3, \textit{cf}.\  \cite{numpy})) and \texttt{MatLab}~relied~on the \texttt{mat4py} library (version 3.1.4).\enlargethispage{10mm}
	
	To keep the computational costs moderate, we restrict to the case $d=2$.  
	As quadrature points of the one-point~quadrature~rule~used~to discretize $p \in C^{0,\alpha}(\overline{\Omega})$, $\alpha \in (0,1]$, 
	we employ barycenters~of~elements, \textit{i.e.}, we set $\xi_T \coloneqq  \frac{1}{3}\sum_{\nu \in \mathcal{N}_h\cap T}{\nu}$ for all $T\in  \mathcal{T}_h$, where $\mathcal{N}_h$ denotes the set of vertices of $\mathcal{T}_h$.
	
	We approximate the discrete solution $(\bfv_h,q_h)^{\top}\in \textcolor{black}{\Vo_{h,0}}\times \Qo_h$ of the non-linear saddle point problem (\textit{i.e.}, Problem (\hyperlink{Qh}{Q$_h$})) using the Newton solver from \mbox{\texttt{PETSc}} (version~3.17.3, \textit{cf}.~\cite{LW10}),~with~an~\mbox{absolute} tolerance~of $\tau_{abs}= 1.0\times10^{-8}$ and a relative tolerance of $\tau_{rel}=1.0\times10^{-10}$.~The~linear~system emerging~in~each~Newton iteration is solved using a sparse direct solver from \texttt{MUMPS} (version~5.5.0,~\textit{cf}.~\cite{mumps}).~In~the \mbox{implementation}, the uniqueness of the pressure is enforced via~adding~a~zero~mean~condition.\enlargethispage{3.5mm}
	
	\subsection{Quasi-optimality of the derived error decay rates (for the velocity)}
	
	\hspace*{5mm}In this subsection, we confirm the quasi-optimality of the derived error decay rates (for~the~velocity).
	More precisely, we apply Problem (\hyperlink{Qh}{Q$_h$}) (or Problem (\hyperlink{Ph}{P$_h$}), respectively)  to approximate the system~\eqref{eq:p-navier-stokes} with  $\bfS\colon\Omega\times \mathbb{R}^{2\times 2}\to\mathbb{R}^{2\times 2}_{\mathrm{sym}}$, where $\Omega=(0,1)^2$, for every \textcolor{black}{$(x,\bfA)^\top\in \Omega\times\mathbb{R}^{2\times 2}$}~defined~by\vspace*{-0.5mm}
	\begin{align*}
		\bfS(x,\bfA) \coloneqq \mu_0\,(\delta+\vert \bfA^{\textup{sym}}\vert)^{p(x)-2}\bfA^{\textup{sym}}\,,
	\end{align*}  
	where $\mu_0 = \frac{1}{2}$, $\delta\coloneqq 1.0\times10^{-5}$, and $ p\in C^{0,\alpha}(\overline{\Omega})$, where $\alpha\in (0,1]$,  for every $x\in \overline{\Omega}$ is defined by\vspace*{-0.5mm}
	\begin{align*}
		p(x)\coloneqq \Big(1-\frac{\vert x\vert^{\alpha}}{2^{\alpha/2}}\Big)\, p^++\frac{\vert x\vert^{\alpha}}{2^{\alpha/2}}\,p^-\,,
	\end{align*}
	where $p^-,p^+>1$. As manufactured solutions serve the vector field $\bfv\in V$\footnote{The manufacture solution does not satisfy the homogeneous Dirichlet boundary condition \eqref{eq:p-navier-stokes}$_3$. 
		However, the error is concentrated around the singularity and, therefore, this small inconsistency with the setup of the theory does not have any influence on the results of this paper. Following the recent contribution \cite{JK23_inhom}, \textcolor{black}{which treats the case of a constant power-law index}, the results of this paper should be readily extendable to the case of homogeneous Dirichlet~boundary~data.\vspace*{-5.5mm}} and the function $q \in \Qo$, given $\rho_{\bfv},\rho_q\in C^{0,\alpha}(\overline{\Omega})$, 
	for every $x\coloneqq(x_1,x_2)^\top\in \Omega$ defined by\vspace*{-0.5mm}
	\begin{align}\label{solutions}
		\bfv(x)\coloneqq\vert x\vert^{\rho_{\bfv}(x)} (x_2,-x_1)^\top\,, \qquad q(x)\coloneqq \vert x\vert^{\rho_{q}(x)}-\langle\,\vert \!\cdot\!\vert^{\rho_{q}(\cdot)}\,\rangle_\Omega\,,
	\end{align}
	\textit{i.e.}, we choose $\bff\in L^{p'(\cdot)}(\Omega;\mathbb{R}^2)$ and boundary data $\bfv_{\partial\Omega}\in W^{\smash{1,1-\frac{1}{p^-}}}(\partial\Omega;\mathbb{R}^2)$ accordingly.
	
	Regarding the regularity of the velocity, for  $\beta\in (0,1]$, we choose
	${\rho_{\bfv}\coloneqq 2\frac{\beta-1}{p}\hspace*{-0.1em}+\hspace*{-0.1em}1.0\hspace*{-0.1em}\times\hspace*{-0.1em}10^{-4}\hspace*{-0.1em}\in\hspace*{-0.1em} C^{0,\alpha}(\overline{\Omega})}$, 
	 such 
	that $\bfF(\cdot,\bfD\bfv)\in N^{\beta,2}(\Omega;\mathbb{R}^{2\times 2})$. 
	Regarding the regularity of the pressure, for  $\gamma \in (0,1]$,~we~consider the two different cases  in Corollary~\ref{cor:error_FE}:\vspace{2mm}
	\begin{itemize}[noitemsep,topsep=0pt,leftmargin=!,labelwidth=\widthof{(Case 2)}]
		\item[(\hypertarget{case_1}{Case 1})] We choose $\rho_{q}\hspace*{-0.1em} =\hspace*{-0.1em}  \gamma-\frac{2}{p'}+1.0\times10^{-4}\hspace*{-0.1em} \in\hspace*{-0.1em}  C^{0,\alpha}(\overline{\Omega})$, 
		 such that  ${q\hspace*{-0.1em} \in\hspace*{-0.1em} 
		H^{\gamma,p'(\cdot)}(\Omega)\cap L^{p'(\cdot)}_0(\Omega)}$;\vspace*{1mm}
		\item[(\hypertarget{case_2}{Case 2})] We \hspace*{-0.1mm}choose \hspace*{-0.1mm}$\rho_{q}\hspace*{-0.175em} =\hspace*{-0.175em}  \rho_{\bfv}\frac{p-2}{2} +\gamma-1+1.0\times10^{-4}\hspace*{-0.175em} \in\hspace*{-0.175em}  C^{0,\alpha}(\overline{\Omega})$, 
	 such that $(\delta\hspace*{-0.175em} +\hspace*{-0.175em} \vert\bfD\bfv\vert)^{\smash{\frac{2-p(\cdot)}{2}}} \vert \nabla^{\gamma} q\vert  \in
		L^2(\Omega)$.\vspace{3mm}
	\end{itemize}
	The power-law index $p\in C^{0,\alpha}(\overline{\Omega})$ is  constructed precisely in such a way that $p(x)\approx p^+$ close to the origin, \textit{i.e.}, when $\vert x\vert \to 0$, where both the velocity vector field and the pressure field 
	have singularities~that~enforce the respective fractional regularity. This turned out to be crucial for creating a critical setup that enables us to confirm the quasi-optimality of the a  priori error estimates~(for~the~velocity)~in~Corollary~\ref{cor:error_FE}. 
	It proved particularly important to use $x$-dependent power functions $\rho_{\bfv},\rho_q\in C^{0,\alpha}(\overline{\Omega})$ in \eqref{solutions},~as,~then, the asymptotic behavior of the power-law $p\in C^{0,\alpha}(\overline{\Omega})$, \textit{i.e.}, $p(x)\approx p^+$  when $\vert x\vert \to 0$, transfers to these power functions, \textit{i.e.}, $\rho_{\bfv}(x)\approx 2\frac{\beta-1}{p^+}+1.0\times10^{-4}$ and $\rho_q(x)\approx \rho_{\bfv}(x)\frac{p^+-2}{2}+\gamma -1+1.0\times10^{-4}$~if~$\vert x\vert \to 0$.
	
	We construct an initial triangulation $\mathcal
	T_{h_0}$, where $h_0=1$, by subdividing the domain $\Omega=(0,1)^2$ along its diagonals into four triangles  with different orientations.  Finer triangulations~$\mathcal T_{h_i}$, $i=1,\ldots,9$, where $h_{i+1}=\frac{h_i}{2}$ for all $i=0,\ldots,9$, are, then,
	obtained by
	regular subdivision of the previous grid: Each \mbox{triangle} is subdivided
	into four equal triangles by connecting the midpoints of the edges. 
	
	Note that it is difficult to compute numerically the pressure errors measured in the Luxembourg~norm, \textit{i.e.}, $\|q_{h_i}-q\|_{p'(\cdot),\Omega}$, $i=0,\ldots,9$, or the discretized Luxembourg norm,~\textit{i.e.},~$\|q_{h_i}-q\|_{p'_h(\cdot),\Omega}$,~$i=0,\ldots,9$, as these error quantities 
	cannot be localized due to the particular structure~of~the~Luxembourg norm. In fact, appealing to \cite[Cor.\ 7.3.21]{dhhr}, one can only hope for the \textit{local-to-global~equivalence}
	\begin{align}
		e_{q,i}^{\textup{true}}\coloneqq\|q_{h_i}-q\|_{p'(\cdot),\Omega}\sim \bigg\|\sum_{T\in \mathcal{T}_h}{\chi_T\frac{\|q_{h_i}-q\|_{p'(\cdot),T}}{\|1\|_{p'(\cdot),T}}}\bigg\|_{p'(\cdot),\Omega}\,,\quad i=0,\ldots,9\,.\label{eq:pseudo-localized}
	\end{align}
	 To \hspace*{-0.1mm}have \hspace*{-0.1mm}at \hspace*{-0.1mm}least \hspace*{-0.1mm}an \hspace*{-0.1mm}upper \hspace*{-0.1mm}bound \hspace*{-0.1mm}for \hspace*{-0.1mm}the \hspace*{-0.1mm}asymptotic \hspace*{-0.1mm}behavior \hspace*{-0.1mm}of \hspace*{-0.1mm}the \hspace*{-0.2mm}\textit{``true''} \hspace*{-0.1mm}pressure~\hspace*{-0.1mm}errors~\hspace*{-0.1mm}$e_{q,i}^{\textup{true}}$,~\hspace*{-0.1mm}${i\hspace*{-0.1em}=\hspace*{-0.1em}0,\ldots,9}$, (\textit{cf}.~\eqref{eq:pseudo-localized}),
	in the numerical experiments, we compute \textit{``localized''} pressure errors $e_{q,i}$, $i\hspace*{-0.1em}=\hspace*{-0.1em}0,\ldots,9$,~(\textit{cf}.~\eqref{eq:errors}), which can be 
    computed numerically and are 
    motivated by the fact that, due to Lemma~\ref{lem:norm_equiv},~Lemma~\ref{lem:Pi_Q_norm} \eqref{lem:Pi_Q_norm.global}, and Minkowski's inequality, for every $i=0,\ldots,9$, we have that
	\begin{align}\label{eq:relation}
		\begin{aligned}
		e_{q,i}^{\textup{true}}&\leq \|q_{h_i}-\textcolor{black}{\Pi}_{h_i}^0 q\|_{p'(\cdot),\Omega}+\|q-\textcolor{black}{\Pi}_{h_i}^0 q\|_{p'(\cdot),\Omega}
			\\&	\leq \|q_{h_i}-\textcolor{black}{\Pi}_{h_i}^0 q\|_{p'_h(\cdot),\Omega}+c\,h^{\gamma}\,\|\vert \nabla^\gamma q\vert \|_{p'(\cdot),\Omega}
			\\[-0.5mm]&	\le \sum_{T\in \mathcal{T}_h}{\|q_{h_i}-\textcolor{black}{\Pi}_{h_i}^0 q\|_{p'(x_T),T}}+c\,h^{\gamma}\,\|\vert \nabla^\gamma q\vert \|_{p'(\cdot),\Omega}\,,
		\end{aligned}
	\end{align}
	where $\textcolor{black}{\Pi}_{h_i}^0\colon L^1(\Omega)\to \mathbb{P}^0(\mathcal{T}_{h_i})$, $i=0,\ldots,9$, denote the (local) $L^2$-projection operators.
	
	Therefore, for the resulting series of triangulations $\mathcal T_{h_i}$, $i=0,\ldots,9$, we apply the above Newton scheme to compute the corresponding discrete solutions $(\bfv_{h_i},q_{h_i})^\top\in V_{h_i}\times \Qo_{h_i}$, $i=0,\ldots,9$, 
	and the error quantities 
	\begin{align}\label{eq:errors}
		\left.\begin{aligned}
			e_{\bfv,i}&\coloneqq\|\bfF_{h_i}(\cdot,\bfD\bfv_{h_i})-\bfF_{h_i}(\cdot,\bfD\bfv)\|_{2,\Omega}\,,\\
			e_{q,i}&\coloneqq \sum_{T\in \mathcal{T}_h}{\|q_{h_i}-\textcolor{black}{\Pi}_{h_i}^0q\|_{p'(x_T),T}}\,,\\[-2mm]
		\end{aligned}\quad\right\}\quad i=0,\ldots,9\,.
	\end{align}

	As estimation of the convergence rates,  the experimental order of convergence~(EOC)
	\begin{align*}
		\texttt{EOC}_i(e_i)\coloneqq\frac{\log(e_i/e_{i-1})}{\log(h_i/h_{i-1})}\,, \quad i=1,\ldots,9\,,
	\end{align*}
	where for every $i= 1,\ldots,9$, we denote by $e_i$
	either 
	$e_{\bfv,i}$ or
	$e_{q,i}$,~respectively,~is~recorded.  
	
	Due to  Corollary~\ref{cor:error_FE}, in the Case \hyperlink{case_1}{1}, we can expect the
	convergence rate $\min\{\alpha,\beta,\gamma \min\{1,(p^+)'/2\}\}$ for error quantity $e_{\bfv,i}$, $i=1,\ldots,9$, while in the  Case \hyperlink{case_2}{2}, we can expect~the~convergence~rate~$\min\{\alpha,\beta,\gamma\}$.
	Moreover, due to Corollary \ref{cor:error_FE} and the relation \eqref{eq:relation}, in the Case \hyperlink{case_1}{1},  we
	can expect a convergence rate not larger than $2\min\{\alpha,\beta,\gamma \min\{1,(p^+)'/2\}\}/(r^-)'$ 
	 for error quantity $e_{q,i}$, $i=1,\ldots,9$, while~in~the~Case~\hyperlink{case_2}{2}, we can expect a convergence rate not larger than $\min\{\alpha,\beta,\gamma\}$.
	 
	  Motivated by Lemma \ref{lem:pres}, we restrict to the case $\alpha=\beta=\gamma$. In addition, we always
	set $p^+\coloneqq p^-+1$.\enlargethispage{4mm}
	
	For different values of $p^- \in \{1.5,1.75, 2, 2.25, 2.5, 2.75\}$, fractional exponents $\alpha =\beta=\gamma\in \{0.5,1.0\}$, and a series of triangulations $\mathcal{T}_{h_i}$,
	$i = 0, \ldots , 9$, obtained by  global refinement~as~described~above, the EOC is
	computed and presented~in~the~Tables~\ref{tab1}--\ref{tab4}, respectively:
	for both the Mini element and the Taylor--Hood element, for the velocity errors $e_{\bfv,i}$, $i=0,\ldots,9$, we report
	the expected a convergence rate of about $\texttt{EOC}_i(e_{\bfv,i}) \approx \alpha \min\{1,(p^+)'/2\}$ , $i=1,\ldots,9$,~in~Case~\hyperlink{case_1}{1} and $\texttt{EOC}_i(e_{\bfv,i}) \approx \alpha$,~${i=1,\ldots, 9}$,~in~Case~\hyperlink{case_2}{2}. For the pressure errors  $e_{q,i}$, $i=0,\ldots,9$, however, for both the Mini element and the Taylor--Hood element, we report experimental convergence rates that are larger than theoretically predicted maximal convergence rates of about 
	$\min\{2,(p^+)'\}/(r^-)'$ in Case~\hyperlink{case_1}{1} and $\min\{1,(p^+)'/2\}\}$  in Case~\hyperlink{case_2}{2}.~Since,~owing~to~\eqref{eq:relation},\linebreak the experimental convergence rates of the true pressure errors $e_{q,i}^{\textup{true}}$, $i=0,\ldots,9$, at least as high as the experimental convergence rates of the localized pressure error $e_{q,i}$, $i=0,\ldots,9$, this indicates that the \textit{a priori} error estimates derived for the pressure in Corollary \ref{cor:error_FE} are potentially sub-optimal.~In~the~case~of~a constant \hspace*{-0.1mm}power-law \hspace*{-0.1mm}index \hspace*{-0.1mm}$p\hspace*{-0.1em}\in\hspace*{-0.1em} (2,\infty)$, \hspace*{-0.1mm}the \hspace*{-0.1mm}same \hspace*{-0.1mm}has \hspace*{-0.1mm}been \hspace*{-0.1mm}report~\hspace*{-0.1mm}for~\hspace*{-0.1mm}a~\hspace*{-0.1mm}finite~\hspace*{-0.1mm}element~\hspace*{-0.1mm}approximation~\hspace*{-0.1mm}in~\hspace*{-0.1mm}\mbox{\cite{bdr-phi-stokes,JK23_inhom}}.\newpage

\begin{table}[H]
	\setlength\tabcolsep{7.0pt}
	\centering
	\begin{tabular}{c |c|c|c|c|c|c|c|c|c|c|}  \hline
		\multicolumn{1}{|c||}{\cellcolor{lightgray}$\rho_q$}	
		& \multicolumn{6}{c||}{\cellcolor{lightgray}(Case \hyperlink{case_1}{1})\vphantom{$X^{X_{X_X}}_{X^{X^X}}$}}   & \multicolumn{4}{c|}{\cellcolor{lightgray}(Case \hyperlink{case_2}{2})}\\ 
		\hline 
		
		\multicolumn{1}{|c||}{\cellcolor{lightgray}\diagbox[height=1.1\line,width=0.135\dimexpr\linewidth]{\vspace{-0.6mm}\hspace*{-1mm}$i$}{\\[-4.5mm]$p^-$\hspace*{-2mm}}}
		& \cellcolor{lightgray}1.5 & \cellcolor{lightgray}1.75  & \cellcolor{lightgray}2.0  &  \cellcolor{lightgray}2.25 & \cellcolor{lightgray}2.5  & \multicolumn{1}{c||}{\cellcolor{lightgray}2.75} & 
		\multicolumn{1}{c|}{\cellcolor{lightgray}2.0}   & \cellcolor{lightgray}2.25  & \cellcolor{lightgray}2.5  & \cellcolor{lightgray}2.75  \\ \hline\hline
		\multicolumn{1}{|c||}{\cellcolor{lightgray}$\alpha=\beta=\gamma$} &	\multicolumn{10}{c|}{\cellcolor{lightgray}$1.0$}	\\ \hline
		\multicolumn{1}{|c||}{\cellcolor{lightgray}$3$}             & 0.642 & 0.639 & 0.632 & 0.623 & 0.613 & \multicolumn{1}{c||}{0.603} & \multicolumn{1}{c|}{0.681} & 0.672 & 0.579 & 0.401 \\ \hline
		\multicolumn{1}{|c||}{\cellcolor{lightgray}$4$}             & 0.750 & 0.719 & 0.695 & 0.676 & 0.659 & \multicolumn{1}{c||}{0.645} & \multicolumn{1}{c|}{0.759} & 0.805 & 0.849 & 0.870 \\ \hline
		\multicolumn{1}{|c||}{\cellcolor{lightgray}$5$}             & 0.801 & 0.757 & 0.725 & 0.701 & 0.681 & \multicolumn{1}{c||}{0.665} & \multicolumn{1}{c|}{0.846} & 0.806 & 0.790 & 0.834 \\ \hline
		\multicolumn{1}{|c||}{\cellcolor{lightgray}$6$}             & 0.824 & 0.774 & 0.739 & 0.713 & 0.692 & \multicolumn{1}{c||}{0.674} & \multicolumn{1}{c|}{0.912} & 0.884 & 0.865 & 0.851 \\ \hline
		\multicolumn{1}{|c||}{\cellcolor{lightgray}$7$}             & 0.833 & 0.782 & 0.746 & 0.718 & 0.696 & \multicolumn{1}{c||}{0.678} & \multicolumn{1}{c|}{0.928} & 0.926 & 0.921 & 0.910 \\ \hline
		\multicolumn{1}{|c||}{\cellcolor{lightgray}$8$}             & 0.835 & 0.785 & 0.749 & 0.721 & 0.699 & \multicolumn{1}{c||}{0.680} & \multicolumn{1}{c|}{0.974} & 0.965 & 0.949 & 0.936 \\ \hline
		\multicolumn{1}{|c||}{\cellcolor{lightgray}$9$}             & 0.836 & 0.786 & 0.750 & 0.722 & 0.700 & \multicolumn{1}{c||}{0.681} & \multicolumn{1}{c|}{0.987} & 0.981 & 0.967 & 0.848 \\ \hline\hline
		\multicolumn{1}{|c||}{\cellcolor{lightgray}\small theory}   & 0.833 & 0.786 & 0.750 & 0.722 & 0.700 & \multicolumn{1}{c||}{0.682} & \multicolumn{1}{c|}{1.000} & 1.000 & 1.000 & 1.000 \\ \hline\hline
		\multicolumn{1}{|c||}{\cellcolor{lightgray}$\alpha=\beta=\gamma$} &	\multicolumn{10}{c|}{\cellcolor{lightgray}$0.5$}	\\ \hline
		\multicolumn{1}{|c||}{\cellcolor{lightgray}$3$}             & 0.660 & 0.603 & 0.509 & 0.425 & 0.367 & \multicolumn{1}{c||}{0.332} & \multicolumn{1}{c|}{0.642} & 0.596 & 0.552 & 0.516 \\ \hline
		\multicolumn{1}{|c||}{\cellcolor{lightgray}$4$}             & 0.573 & 0.512 & 0.439 & 0.381 & 0.346 & \multicolumn{1}{c||}{0.327} & \multicolumn{1}{c|}{0.568} & 0.540 & 0.515 & 0.496 \\ \hline
		\multicolumn{1}{|c||}{\cellcolor{lightgray}$5$}             & 0.530 & 0.473 & 0.413 & 0.369 & 0.345 & \multicolumn{1}{c||}{0.331} & \multicolumn{1}{c|}{0.539} & 0.522 & 0.506 & 0.495 \\ \hline
		\multicolumn{1}{|c||}{\cellcolor{lightgray}$6$}             & 0.503 & 0.451 & 0.400 & 0.366 & 0.346 & \multicolumn{1}{c||}{0.335} & \multicolumn{1}{c|}{0.522} & 0.511 & 0.502 & 0.496 \\ \hline
		\multicolumn{1}{|c||}{\cellcolor{lightgray}$7$}             & 0.486 & 0.438 & 0.393 & 0.365 & 0.348 & \multicolumn{1}{c||}{0.338} & \multicolumn{1}{c|}{0.512} & 0.506 & 0.501 & 0.497 \\ \hline
		\multicolumn{1}{|c||}{\cellcolor{lightgray}$8$}             & 0.476 & 0.430 & 0.390 & 0.365 & 0.350 & \multicolumn{1}{c||}{0.339} & \multicolumn{1}{c|}{0.506} & 0.503 & 0.500 & 0.498 \\ \hline
		\multicolumn{1}{|c||}{\cellcolor{lightgray}$9$}             & 0.470 & 0.425 & 0.388 & 0.365 & 0.350 & \multicolumn{1}{c||}{0.341} & \multicolumn{1}{c|}{0.502} & 0.501 & 0.500 & 0.499 \\ \hline\hline
		\multicolumn{1}{|c||}{\cellcolor{lightgray}\small theory}   & 0.417 & 0.393 & 0.375 & 0.361 & 0.350 & \multicolumn{1}{c||}{0.341} & \multicolumn{1}{c|}{0.500} & 0.500 & 0.500 & 0.500 \\ \hline
	\end{tabular}
	\caption{Experimental order of convergence (MINI): $\texttt{EOC}_i(e_{\bfv,i})$,~${i=3,\dots,9}$.}
	\label{tab1}
\end{table}

\begin{table}[H]
	\setlength\tabcolsep{7.0pt}
	\centering
	\begin{tabular}{c |c|c|c|c|c|c|c|c|c|c|c|c|}  \hline
		\multicolumn{1}{|c||}{\cellcolor{lightgray}$\rho_q$}	
		& \multicolumn{6}{c||}{\cellcolor{lightgray}(Case \hyperlink{case_1}{1})\vphantom{$X^{X_{X_X}}_{X^{X^X}}$}}   & \multicolumn{4}{c|}{\cellcolor{lightgray}(Case \hyperlink{case_2}{2})}\\ 
		\hline 
		
		\multicolumn{1}{|c||}{\cellcolor{lightgray}\diagbox[height=1.1\line,width=0.135\dimexpr\linewidth]{\vspace{-0.6mm}\hspace*{-1mm}$i$}{\\[-4.75mm]$p^-$\hspace*{-1.5mm}}}
		& \cellcolor{lightgray}1.5 & \cellcolor{lightgray}1.75  & \cellcolor{lightgray}2.0  &  \cellcolor{lightgray}2.25 & \cellcolor{lightgray}2.5  & \multicolumn{1}{c||}{\cellcolor{lightgray}2.75} & 
		\multicolumn{1}{c|}{\cellcolor{lightgray}2.0}   & \cellcolor{lightgray}2.25  & \cellcolor{lightgray}2.5  & \cellcolor{lightgray}2.75  \\ \hline\hline
		\multicolumn{1}{|c||}{\cellcolor{lightgray}$\alpha=\beta=\gamma$} &	\multicolumn{10}{c|}{\cellcolor{lightgray}$1.0$}	\\ \hline
		\multicolumn{1}{|c||}{\cellcolor{lightgray}$3$}             & 0.631 & 0.660 & 0.671 & 0.681 & 0.691 & \multicolumn{1}{c||}{0.709} & \multicolumn{1}{c|}{1.172} & 1.476 & 1.498 & 1.002 \\ \hline
		\multicolumn{1}{|c||}{\cellcolor{lightgray}$4$}             & 0.694 & 0.727 & 0.748 & 0.764 & 0.777 & \multicolumn{1}{c||}{0.790} & \multicolumn{1}{c|}{1.083} & 1.435 & 1.285 & 1.488 \\ \hline
		\multicolumn{1}{|c||}{\cellcolor{lightgray}$5$}             & 0.778 & 0.817 & 0.838 & 0.854 & 0.865 & \multicolumn{1}{c||}{0.875} & \multicolumn{1}{c|}{0.977} & 1.222 & 1.932 & 1.639 \\ \hline
		\multicolumn{1}{|c||}{\cellcolor{lightgray}$6$}             & 0.829 & 0.873 & 0.894 & 0.907 & 0.917 & \multicolumn{1}{c||}{0.924} & \multicolumn{1}{c|}{1.000} & 1.130 & 1.255 & 1.922 \\ \hline
		\multicolumn{1}{|c||}{\cellcolor{lightgray}$7$}             & 0.859 & 0.910 & 0.929 & 0.941 & 0.949 & \multicolumn{1}{c||}{0.954} & \multicolumn{1}{c|}{1.041} & 1.150 & 1.216 & 1.572 \\ \hline
		\multicolumn{1}{|c||}{\cellcolor{lightgray}$8$}             & 0.877 & 0.935 & 0.952 & 0.962 & 0.968 & \multicolumn{1}{c||}{0.972} & \multicolumn{1}{c|}{1.073} & 1.181 & 1.230 & 1.309 \\ \hline
		\multicolumn{1}{|c||}{\cellcolor{lightgray}$9$}             & 0.888 & 0.952 & 0.969 & 0.976 & 0.981 & \multicolumn{1}{c||}{0.984} & \multicolumn{1}{c|}{1.097} & 1.201 & 1.252 & 1.331 \\ \hline\hline
		\multicolumn{1}{|c||}{\cellcolor{lightgray}\small theory}   & 0.556 & 0.673 & 0.750 & 0.722 & 0.700 & \multicolumn{1}{c||}{0.682} & \multicolumn{1}{c|}{1.000} & 1.000 & 1.000 & 1.000 \\ \hline\hline
		\multicolumn{1}{|c||}{\cellcolor{lightgray}$\alpha=\beta=\gamma$} &	\multicolumn{10}{c|}{\cellcolor{lightgray}$0.5$}	\\ \hline
		\multicolumn{1}{|c||}{\cellcolor{lightgray}$3$}             & 0.582 & 0.673 & 0.757 & 0.775 & 0.723 & \multicolumn{1}{c||}{0.622} & \multicolumn{1}{c|}{0.768} & 0.857 & 0.920 & 0.859 \\ \hline
		\multicolumn{1}{|c||}{\cellcolor{lightgray}$4$}             & 0.725 & 0.787 & 0.791 & 0.735 & 0.627 & \multicolumn{1}{c||}{0.554} & \multicolumn{1}{c|}{0.851} & 0.874 & 0.843 & 0.785 \\ \hline
		\multicolumn{1}{|c||}{\cellcolor{lightgray}$5$}             & 0.710 & 0.760 & 0.718 & 0.650 & 0.580 & \multicolumn{1}{c||}{0.525} & \multicolumn{1}{c|}{0.814} & 0.810 & 0.789 & 0.756 \\ \hline
		\multicolumn{1}{|c||}{\cellcolor{lightgray}$6$}             & 0.680 & 0.708 & 0.678 & 0.614 & 0.562 & \multicolumn{1}{c||}{0.518} & \multicolumn{1}{c|}{0.778} & 0.782 & 0.770 & 0.759 \\ \hline
		\multicolumn{1}{|c||}{\cellcolor{lightgray}$7$}             & 0.661 & 0.677 & 0.644 & 0.596 & 0.554 & \multicolumn{1}{c||}{0.513} & \multicolumn{1}{c|}{0.758} & 0.766 & 0.762 & 0.758 \\ \hline
		\multicolumn{1}{|c||}{\cellcolor{lightgray}$8$}             & 0.643 & 0.645 & 0.612 & 0.580 & 0.545 & \multicolumn{1}{c||}{0.514} & \multicolumn{1}{c|}{0.742} & 0.754 & 0.756 & 0.756 \\ \hline
		\multicolumn{1}{|c||}{\cellcolor{lightgray}$9$}             & 0.622 & 0.621 & 0.591 & 0.562 & 0.535 & \multicolumn{1}{c||}{0.511} & \multicolumn{1}{c|}{0.727} & 0.741 & 0.749 & 0.754 \\ \hline\hline
		\multicolumn{1}{|c||}{\cellcolor{lightgray}\small theory}   & 0.278 & 0.337 & 0.375 & 0.361 & 0.350 & \multicolumn{1}{c||}{0.341} & \multicolumn{1}{c|}{0.500} & 0.500 & 0.500 & 0.500 \\ \hline
	\end{tabular}
	\caption{Experimental order of convergence (MINI): $\texttt{EOC}_i(e_{q,i})$,~${i=3,\dots,9}$.}
	\label{tab2}
\end{table}
	
	\begin{table}[H]
		\setlength\tabcolsep{7.0pt}
		\centering
		\begin{tabular}{c |c|c|c|c|c|c|c|c|c|c|}  \hline
			\multicolumn{1}{|c||}{\cellcolor{lightgray}$\rho_q$}	
			& \multicolumn{6}{c||}{\cellcolor{lightgray}(Case \hyperlink{case_1}{1})\vphantom{$X^{X_{X_X}}_{X^{X^X}}$}}   & \multicolumn{4}{c|}{\cellcolor{lightgray}(Case \hyperlink{case_2}{2})}\\ 
			\hline 
			
			\multicolumn{1}{|c||}{\cellcolor{lightgray}\diagbox[height=1.1\line,width=0.135\dimexpr\linewidth]{\vspace{-0.6mm}\hspace*{-1mm}$i$}{\\[-4.5mm]$p^-$\hspace*{-2mm}}}
			& \cellcolor{lightgray}1.5 & \cellcolor{lightgray}1.75  & \cellcolor{lightgray}2.0  &  \cellcolor{lightgray}2.25 & \cellcolor{lightgray}2.5  & \multicolumn{1}{c||}{\cellcolor{lightgray}2.75} & 
			\multicolumn{1}{c|}{\cellcolor{lightgray}2.0}   & \cellcolor{lightgray}2.25  & \cellcolor{lightgray}2.5  & \cellcolor{lightgray}2.75  \\ \hline\hline
			\multicolumn{1}{|c||}{\cellcolor{lightgray}$\alpha=\beta=\gamma$} &	\multicolumn{10}{c|}{\cellcolor{lightgray}$1.0$}	\\ \hline
			\multicolumn{1}{|c||}{\cellcolor{lightgray}$3$}             & 0.692 & 0.685 & 0.673 & 0.659 & 0.646 & \multicolumn{1}{c||}{0.634} & \multicolumn{1}{c|}{0.856} & 0.964 & 1.008 & -0.12 \\ \hline
			\multicolumn{1}{|c||}{\cellcolor{lightgray}$4$}             & 0.773 & 0.740 & 0.714 & 0.693 & 0.674 & \multicolumn{1}{c||}{0.658} & \multicolumn{1}{c|}{0.820} & 0.964 & 1.112 & 1.288 \\ \hline
			\multicolumn{1}{|c||}{\cellcolor{lightgray}$5$}             & 0.808 & 0.765 & 0.733 & 0.708 & 0.688 & \multicolumn{1}{c||}{0.670} & \multicolumn{1}{c|}{0.848} & 0.832 & 0.872 & 0.999 \\ \hline
			\multicolumn{1}{|c||}{\cellcolor{lightgray}$6$}             & 0.824 & 0.777 & 0.743 & 0.716 & 0.694 & \multicolumn{1}{c||}{0.677} & \multicolumn{1}{c|}{0.912} & 0.890 & 0.878 & 0.905 \\ \hline
			\multicolumn{1}{|c||}{\cellcolor{lightgray}$7$}             & 0.831 & 0.783 & 0.747 & 0.720 & 0.698 & \multicolumn{1}{c||}{0.679} & \multicolumn{1}{c|}{0.951} & 0.934 & 0.922 & 0.936 \\ \hline
			\multicolumn{1}{|c||}{\cellcolor{lightgray}$8$}             & 0.834 & 0.785 & 0.749 & 0.721 & 0.699 & \multicolumn{1}{c||}{0.681} & \multicolumn{1}{c|}{0.917} & 0.928 & 0.946 & 0.943 \\ \hline
			\multicolumn{1}{|c||}{\cellcolor{lightgray}$9$}             & 0.834 & 0.786 & 0.750 & 0.722 & 0.700 & \multicolumn{1}{c||}{0.681} & \multicolumn{1}{c|}{0.985} & 0.976 & 0.961 & 0.937 \\ \hline\hline
			\multicolumn{1}{|c||}{\cellcolor{lightgray}\small theory}   & 0.833 & 0.786 & 0.750 & 0.722 & 0.700 & \multicolumn{1}{c||}{0.682} & \multicolumn{1}{c|}{1.000} & 1.000 & 1.000 & 1.000 \\ \hline\hline
			\multicolumn{1}{|c||}{\cellcolor{lightgray}$\alpha=\beta=\gamma$} &	\multicolumn{10}{c|}{\cellcolor{lightgray}$0.5$}	\\ \hline
			\multicolumn{1}{|c||}{\cellcolor{lightgray}$3$}             & 0.353 & 0.308 & 0.305 & 0.309 & 0.312 & \multicolumn{1}{c||}{0.312} & \multicolumn{1}{c|}{0.432} & 0.442 & 0.454 & 0.463 \\ \hline
			\multicolumn{1}{|c||}{\cellcolor{lightgray}$4$}             & 0.371 & 0.344 & 0.337 & 0.333 & 0.330 & \multicolumn{1}{c||}{0.326} & \multicolumn{1}{c|}{0.462} & 0.468 & 0.475 & 0.479 \\ \hline
			\multicolumn{1}{|c||}{\cellcolor{lightgray}$5$}             & 0.389 & 0.365 & 0.354 & 0.346 & 0.339 & \multicolumn{1}{c||}{0.333} & \multicolumn{1}{c|}{0.479} & 0.482 & 0.486 & 0.489 \\ \hline
			\multicolumn{1}{|c||}{\cellcolor{lightgray}$6$}             & 0.401 & 0.378 & 0.364 & 0.353 & 0.345 & \multicolumn{1}{c||}{0.338} & \multicolumn{1}{c|}{0.489} & 0.491 & 0.493 & 0.494 \\ \hline
			\multicolumn{1}{|c||}{\cellcolor{lightgray}$7$}             & 0.410 & 0.385 & 0.369 & 0.358 & 0.348 & \multicolumn{1}{c||}{0.340} & \multicolumn{1}{c|}{0.494} & 0.496 & 0.496 & 0.497 \\ \hline
			\multicolumn{1}{|c||}{\cellcolor{lightgray}$8$}             & 0.415 & 0.389 & 0.373 & 0.360 & 0.350 & \multicolumn{1}{c||}{0.341} & \multicolumn{1}{c|}{0.498} & 0.498 & 0.499 & 0.499 \\ \hline
			\multicolumn{1}{|c||}{\cellcolor{lightgray}$9$}             & 0.417 & 0.392 & 0.374 & 0.361 & 0.351 & \multicolumn{1}{c||}{0.342} & \multicolumn{1}{c|}{0.500} & 0.500 & 0.500 & 0.500 \\ \hline\hline
			\multicolumn{1}{|c||}{\cellcolor{lightgray}\small theory}   & 0.417 & 0.393 & 0.375 & 0.361 & 0.350 & \multicolumn{1}{c||}{0.341} & \multicolumn{1}{c|}{0.500} & 0.500 & 0.500 & 0.500 \\ \hline
		\end{tabular}
		\caption{Experimental order of convergence (Taylor--Hood): $\texttt{EOC}_i(e_{\bfv,i})$,~${i=3,\dots,9}$.}
		\label{tab3}
	\end{table}
	
	\begin{table}[H]
		\setlength\tabcolsep{7.0pt}
		\centering
		\begin{tabular}{c |c|c|c|c|c|c|c|c|c|c|c|c|}  \hline
			\multicolumn{1}{|c||}{\cellcolor{lightgray}$\rho_q$}	
			& \multicolumn{6}{c||}{\cellcolor{lightgray}(Case \hyperlink{case_1}{1})\vphantom{$X^{X_{X_X}}_{X^{X^X}}$}}   & \multicolumn{4}{c|}{\cellcolor{lightgray}(Case \hyperlink{case_2}{2})}\\ 
			\hline 
			
			\multicolumn{1}{|c||}{\cellcolor{lightgray}\diagbox[height=1.1\line,width=0.135\dimexpr\linewidth]{\vspace{-0.6mm}\hspace*{-1mm}$i$}{\\[-4.75mm]$p^-$\hspace*{-1.5mm}}}
			& \cellcolor{lightgray}1.5 & \cellcolor{lightgray}1.75  & \cellcolor{lightgray}2.0  &  \cellcolor{lightgray}2.25 & \cellcolor{lightgray}2.5  & \multicolumn{1}{c||}{\cellcolor{lightgray}2.75} & 
			\multicolumn{1}{c|}{\cellcolor{lightgray}2.0}   & \cellcolor{lightgray}2.25  & \cellcolor{lightgray}2.5  & \cellcolor{lightgray}2.75  \\ \hline\hline
			\multicolumn{1}{|c||}{\cellcolor{lightgray}$\alpha=\beta=\gamma$} &	\multicolumn{10}{c|}{\cellcolor{lightgray}$1.0$}	\\ \hline
			\multicolumn{1}{|c||}{\cellcolor{lightgray}$3$}             & 0.484 & 0.585 & 0.634 & 0.663 & 0.699 & \multicolumn{1}{c||}{0.718} & \multicolumn{1}{c|}{1.870} & 2.184 & 2.118 & 0.724 \\ \hline
			\multicolumn{1}{|c||}{\cellcolor{lightgray}$4$}             & 0.740 & 0.789 & 0.815 & 0.833 & 0.848 & \multicolumn{1}{c||}{0.858} & \multicolumn{1}{c|}{1.292} & 2.397 & 2.882 & 2.365 \\ \hline
			\multicolumn{1}{|c||}{\cellcolor{lightgray}$5$}             & 0.811 & 0.857 & 0.879 & 0.894 & 0.906 & \multicolumn{1}{c||}{0.913} & \multicolumn{1}{c|}{1.000} & 1.338 & 1.720 & 2.311 \\ \hline
			\multicolumn{1}{|c||}{\cellcolor{lightgray}$6$}             & 0.845 & 0.897 & 0.917 & 0.930 & 0.939 & \multicolumn{1}{c||}{0.945} & \multicolumn{1}{c|}{1.037} & 1.182 & 1.744 & 1.746 \\ \hline
			\multicolumn{1}{|c||}{\cellcolor{lightgray}$7$}             & 0.867 & 0.924 & 0.943 & 0.954 & 0.962 & \multicolumn{1}{c||}{0.966} & \multicolumn{1}{c|}{1.068} & 1.181 & 1.382 & 1.795 \\ \hline
			\multicolumn{1}{|c||}{\cellcolor{lightgray}$8$}             & 0.884 & 0.941 & 0.961 & 0.970 & 0.976 & \multicolumn{1}{c||}{0.979} & \multicolumn{1}{c|}{1.091} & 1.194 & 1.335 & 1.714 \\ \hline
			\multicolumn{1}{|c||}{\cellcolor{lightgray}$9$}             & 0.895 & 0.952 & 0.970 & 0.981 & 0.985 & \multicolumn{1}{c||}{0.988} & \multicolumn{1}{c|}{1.110} & 1.207 & 1.276 & 1.669 \\ \hline\hline
			\multicolumn{1}{|c||}{\cellcolor{lightgray}\small theory}   & 0.556 & 0.673 & 0.750 & 0.722 & 0.700 & \multicolumn{1}{c||}{0.682} & \multicolumn{1}{c|}{1.000} & 1.000 & 1.000 & 1.000 \\ \hline\hline
			\multicolumn{1}{|c||}{\cellcolor{lightgray}$\alpha=\beta=\gamma$} &	\multicolumn{10}{c|}{\cellcolor{lightgray}$0.5$}	\\ \hline
			\multicolumn{1}{|c||}{\cellcolor{lightgray}$3$}             & 0.326 & 0.411 & 0.460 & 0.486 & 0.498 & \multicolumn{1}{c||}{0.488} & \multicolumn{1}{c|}{0.551} & 0.619 & 0.670 & 0.682 \\ \hline
			\multicolumn{1}{|c||}{\cellcolor{lightgray}$4$}             & 0.433 & 0.479 & 0.502 & 0.513 & 0.516 & \multicolumn{1}{c||}{0.503} & \multicolumn{1}{c|}{0.623} & 0.659 & 0.684 & 0.704 \\ \hline
			\multicolumn{1}{|c||}{\cellcolor{lightgray}$5$}             & 0.448 & 0.481 & 0.504 & 0.518 & 0.518 & \multicolumn{1}{c||}{0.508} & \multicolumn{1}{c|}{0.653} & 0.683 & 0.706 & 0.718 \\ \hline
			\multicolumn{1}{|c||}{\cellcolor{lightgray}$6$}             & 0.456 & 0.493 & 0.514 & 0.526 & 0.518 & \multicolumn{1}{c||}{0.506} & \multicolumn{1}{c|}{0.675} & 0.697 & 0.712 & 0.720 \\ \hline
			\multicolumn{1}{|c||}{\cellcolor{lightgray}$7$}             & 0.456 & 0.498 & 0.518 & 0.526 & 0.511 & \multicolumn{1}{c||}{0.501} & \multicolumn{1}{c|}{0.688} & 0.708 & 0.710 & 0.720 \\ \hline
			\multicolumn{1}{|c||}{\cellcolor{lightgray}$8$}             & 0.444 & 0.492 & 0.515 & 0.525 & 0.508 & \multicolumn{1}{c||}{0.499} & \multicolumn{1}{c|}{0.693} & 0.709 & 0.712 & 0.723 \\ \hline
			\multicolumn{1}{|c||}{\cellcolor{lightgray}$9$}             & 0.422 & 0.481 & 0.508 & 0.521 & 0.506 & \multicolumn{1}{c||}{0.498} & \multicolumn{1}{c|}{0.691} & 0.705 & 0.713 & 0.725 \\ \hline\hline
			\multicolumn{1}{|c||}{\cellcolor{lightgray}\small theory}   & 0.278 & 0.337 & 0.375 & 0.361 & 0.350 & \multicolumn{1}{c||}{0.341} & \multicolumn{1}{c|}{0.500} & 0.500 & 0.500 & 0.500 \\ \hline
		\end{tabular}
		\caption{Experimental order of convergence (Taylor--Hood): $\texttt{EOC}_i(e_{q,i})$,~${i=3,\dots,9}$.}
		\label{tab4}
	\end{table}
	
	\subsection{An example for an electro-rheological fluid flow}
	
	\hspace*{5mm}In this subsection, we examine a less academic example describing an unsteady \textit{electro-rheological fluid} flow in three dimensions.  Solutions to the steady $p(\cdot)$-Navier--Stokes equations \eqref{eq:p-navier-stokes}~also model  electro-rheological fluid flow behavior, if the right-hand side is given~via~${\mathbf{f}\coloneqq\widehat{\mathbf{f}}+\chi_E\,\textup{div}\,(\mathbf{E}\otimes\mathbf{E})\colon \Omega\to\mathbb{R}^3}$, where $\widehat{\mathbf{f}}\colon \Omega\to \mathbb{R}^3$ is a given \textit{mechanical body force},
	$\chi_E>0$ the \textit{di-eletric susceptibility}, $\mathbf{E}\colon \overline{\Omega}\to \mathbb{R}^3$ a given \textit{electric field}, solving the  \textit{quasi-static Maxwell's equations}, \textit{i.e.},
	\begin{align}\label{eq:maxwell}
		\begin{aligned}
			\textup{div}\,\mathbf{E} &= 0&&\quad\text{ in }\Omega\,,\\
			\textrm{curl}\,\mathbf{E} &= \mathbf{0}&&\quad\text{ in }\Omega\,,\\
			\mathbf{E}\cdot\mathbf{n}&=\mathbf{E}_0\cdot\mathbf{n}&&\quad\text{ on }\partial\Omega\,,
		\end{aligned}
	\end{align} 
	and the power-law index $p\colon \overline{\Omega}\to (1,+\infty)$ depends on the strength of the electric field $\vert \mathbf{E}\vert\colon \overline{\Omega}\to \mathbb{R}_{\ge 0}$, \textit{i.e.}, 
	there exists a material function $\widehat{p}\colon \mathbb{R}_{\ge 0}\to\mathbb{R}_{\ge 0} $ such that
	$p(x)\coloneqq \widehat{p}(\vert \mathbf{E}(x)\vert)$ for all $x\in \overline{\Omega}$. In the system \eqref{eq:maxwell}, by the vector field $\mathbf{n}\colon \partial \Omega\to \mathbb{S}^2$ we denote the normal vector field to $\partial\Omega$ pointing outward.
	
	In the numerical experiments, we choose as physical domain $\Omega\coloneqq (0,1)^3\setminus(B_{\frac{1}{16}}^3(\frac{1}{4}\mathbf{e}_1)\cup B_{\frac{1}{16}}^3(\frac{3}{4}\mathbf{e}_1))$\footnote{Here, $\mathbf{e}_1\coloneqq (1,0,0)^\top\in\mathbb{S}^2 $ denotes the first three-dimensional unit vector and $\smash{B_{\frac{1}{16}}^3(\frac{1}{4}\mathbf{e}_1)}$, $\smash{B_{\frac{1}{16}}^3(\frac{3}{4}\mathbf{e}_1)}$ the three-dimensional balls with radius $\frac{1}{16}$ and center $\frac{1}{4}\mathbf{e}_1$, $\frac{3}{4}\mathbf{e}_1$, respectively.},~\textit{i.e.}, the unit cube with two ball-shaped holes (\text{cf}.\ Figure \ref{fig:E}(LEFT)), as a material function $\widehat{p}\in  C^{0,1}(\mathbb{R}_{\ge 0})$, defined by $\widehat{p}(t)\coloneqq 2+\frac{2}{1+10t}$~for~all~${t\ge 0}$, and as electric field $\mathbf{E} \in C^{\infty}(\overline{\Omega};\mathbb{R}^3)$, for every $x\in \overline{\Omega}$~defined~by
	\begin{align}\label{def:E}
		\mathbf{E}(x)\coloneqq \frac{x-\frac{1}{4}\mathbf{e}_1}{\vert x-\frac{1}{4}\mathbf{e}_1\vert^3}-\frac{x-\frac{3}{4}\mathbf{e}_1}{\vert x-\frac{3}{4}\mathbf{e}_1\vert^3}\,.
	\end{align}
	in order to model \emph{shear-thickening} (note that, by definition, it holds that $p^->2$) between~two~\mbox{\emph{electrodes}},  located at the two holes of the domain $\Omega$ (\textit{cf}.\ Figure \ref{fig:E}(LEFT)).  It is readily checked that $\mathbf{E}\colon \overline{\Omega}\to \mathbb{R}^3$ indeed solves the quasi-static Maxwell's equations \eqref{eq:maxwell} if we prescribe the normal boundary~\mbox{condition}, \textit{e.g.}, \hspace*{-0.1mm}if \hspace*{-0.1mm}we \hspace*{-0.1mm}set \hspace*{-0.1mm}$\mathbf{E}_0\!\coloneqq\! \mathbf{E}$ \hspace*{-0.1mm}on \hspace*{-0.1mm}$\partial\Omega$. \hspace*{-0.1mm}For \hspace*{-0.1mm}sake \hspace*{-0.1mm}of \hspace*{-0.1mm}simplicity, \hspace*{-0.1mm}we \hspace*{-0.1mm}set \hspace*{-0.1mm}$\chi_E\!\coloneqq\! 1$. \hspace*{-0.1mm}To \hspace*{-0.1mm}\mbox{generate}~\hspace*{-0.1mm}a~\hspace*{-0.1mm}\mbox{vortex}~\hspace*{-0.1mm}flow~\hspace*{-0.1mm}around~\hspace*{-0.1mm}the~\hspace*{-0.1mm}\mbox{$x_2$-axis}, we choose as mechanical body force 
	$\widehat{\mathbf{f}}\hspace*{-0.15em}\in \hspace*{-0.15em} C^\infty(\overline{\Omega};\mathbb{R}^3)$,~\mbox{defined}~by~${ \widehat{\mathbf{f}}(x)\hspace*{-0.175em}\coloneqq \hspace*{-0.175em}(2x_2\hspace*{-0.15em}-\hspace*{-0.15em}1)\mathbf{e}_1}$~for~all~$x\hspace*{-0.175em}=\hspace*{-0.175em}(x_1,x_2,x_3)^\top\hspace*{-0.25em}\in \hspace*{-0.175em}\overline{\Omega}$, which becomes the total force $\mathbf{f}\hspace*{-0.15em}\in\hspace*{-0.15em} C^{\infty}(\overline{\Omega};\mathbb{R}^3)$ in absence of the electric field  (\textit{cf}.\ \mbox{Figure} \ref{fig:f}(LEFT)). 
	The electric field $\mathbf{E}\in C^\infty(\overline{\Omega};\mathbb{R}^3)$ and total force
	$\mathbf{f}\in C^\infty(\overline{\Omega};\mathbb{R}^3)$ are depicted in Figure~\ref{fig:E}(RIGHT) and Figure~\ref{fig:f}(RIGHT), respectively. 
	
	\begin{figure}[H]\centering
		\includegraphics[width=6.7cm]{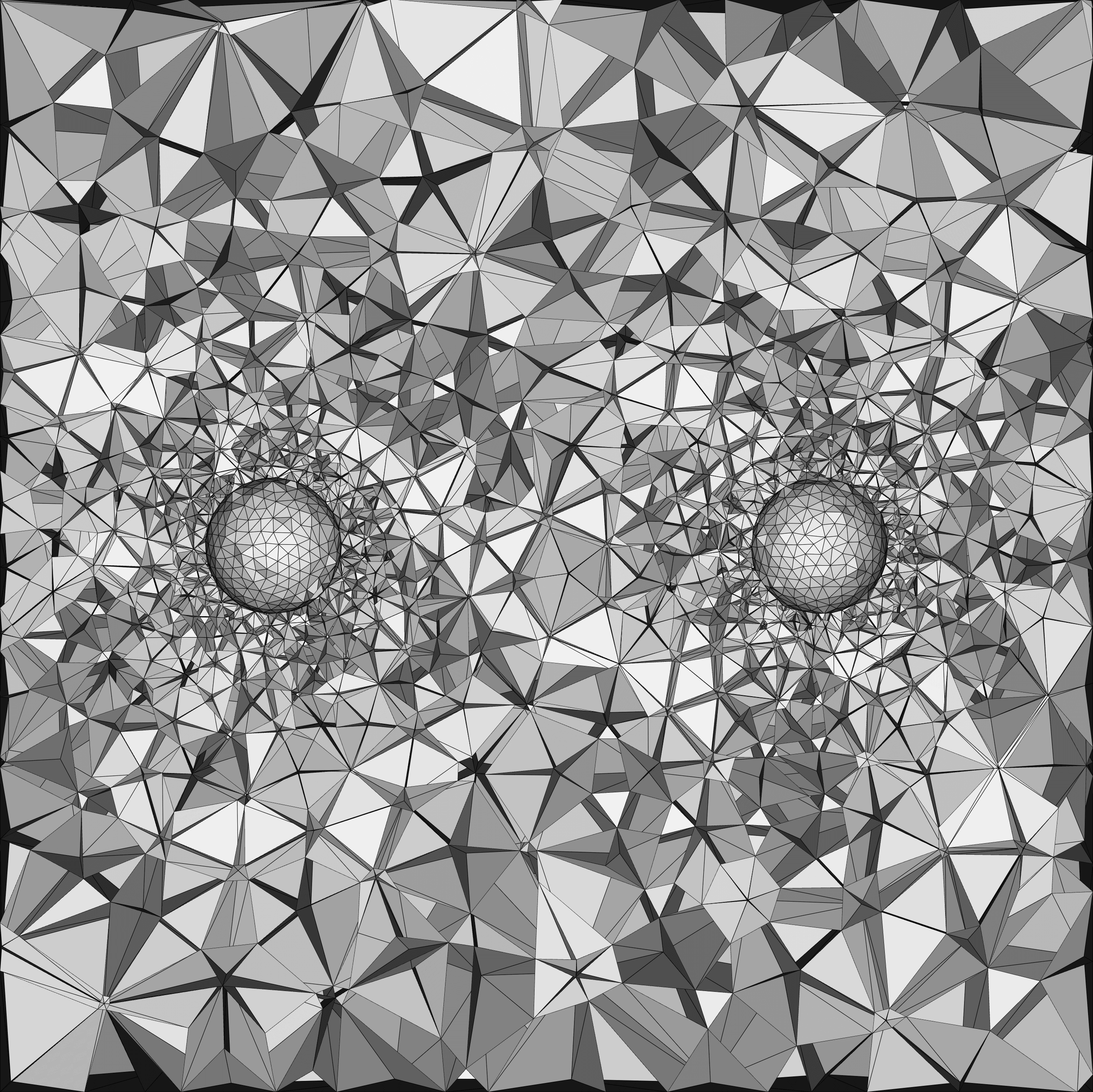} 	\includegraphics[width=8.625cm]{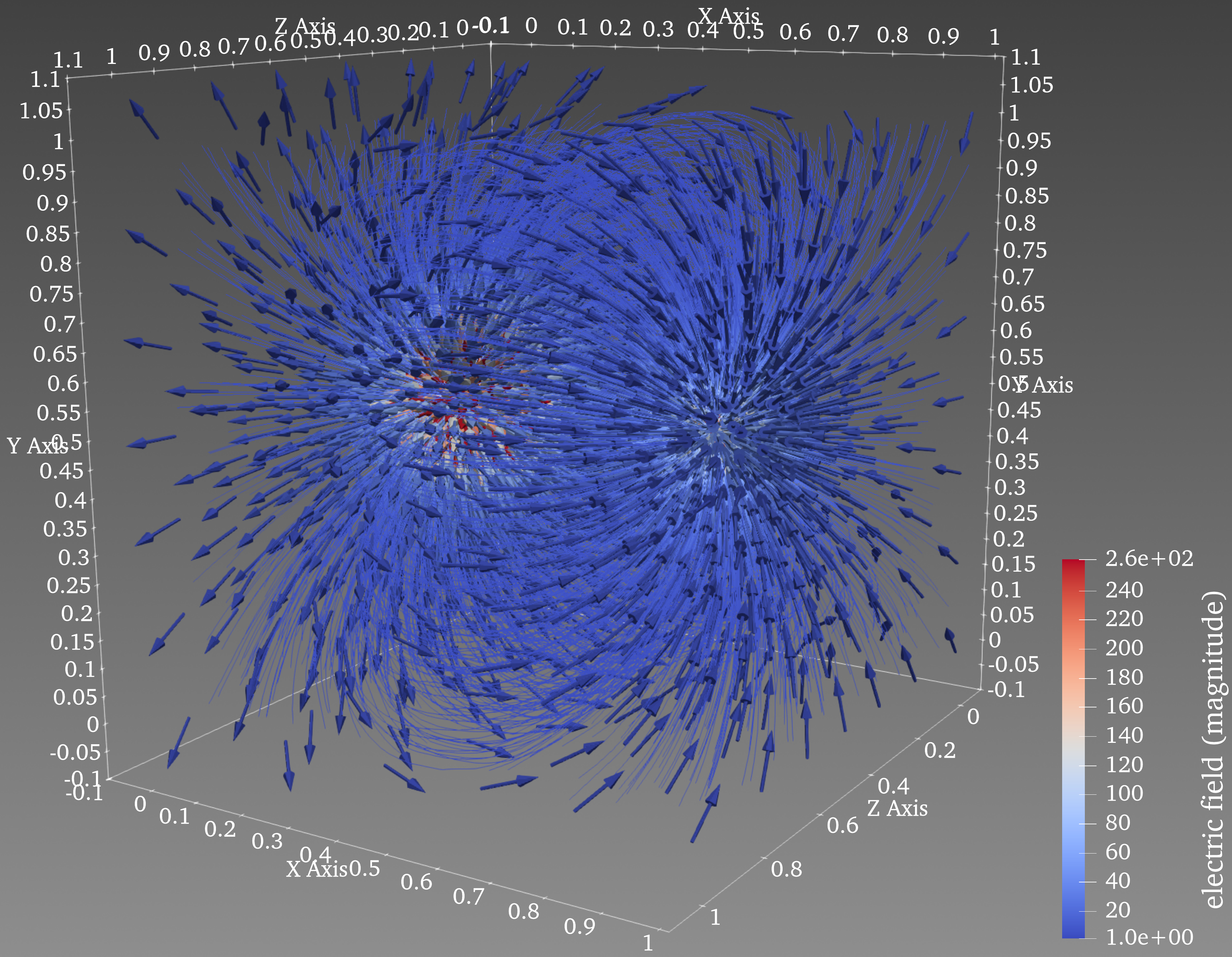} 
		\caption{LEFT: intersection of $\mathcal{T}_h$ with ($\mathbb{R}\times \{\frac{1}{2}\}\times \mathbb{R}$)-plane; RIGHT:  electric field $\mathbf{E}\in C^\infty(\overline{\Omega};\mathbb{R}^3)$.}
		\label{fig:E}
	\end{figure}\vspace*{-3mm}
	
	\begin{figure}[H]\centering
		\includegraphics[width=7.9cm]{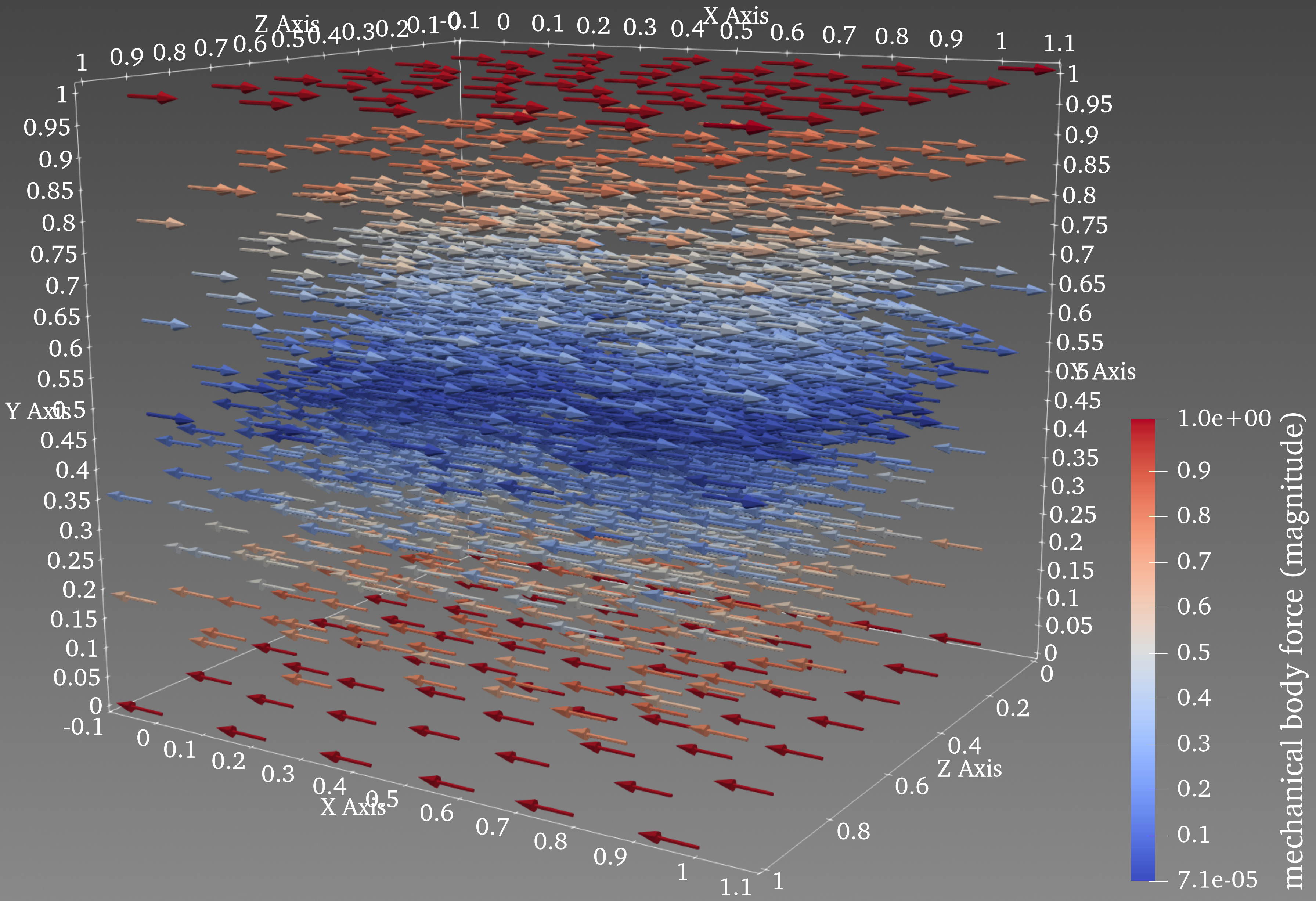}	\includegraphics[width=7.5cm]{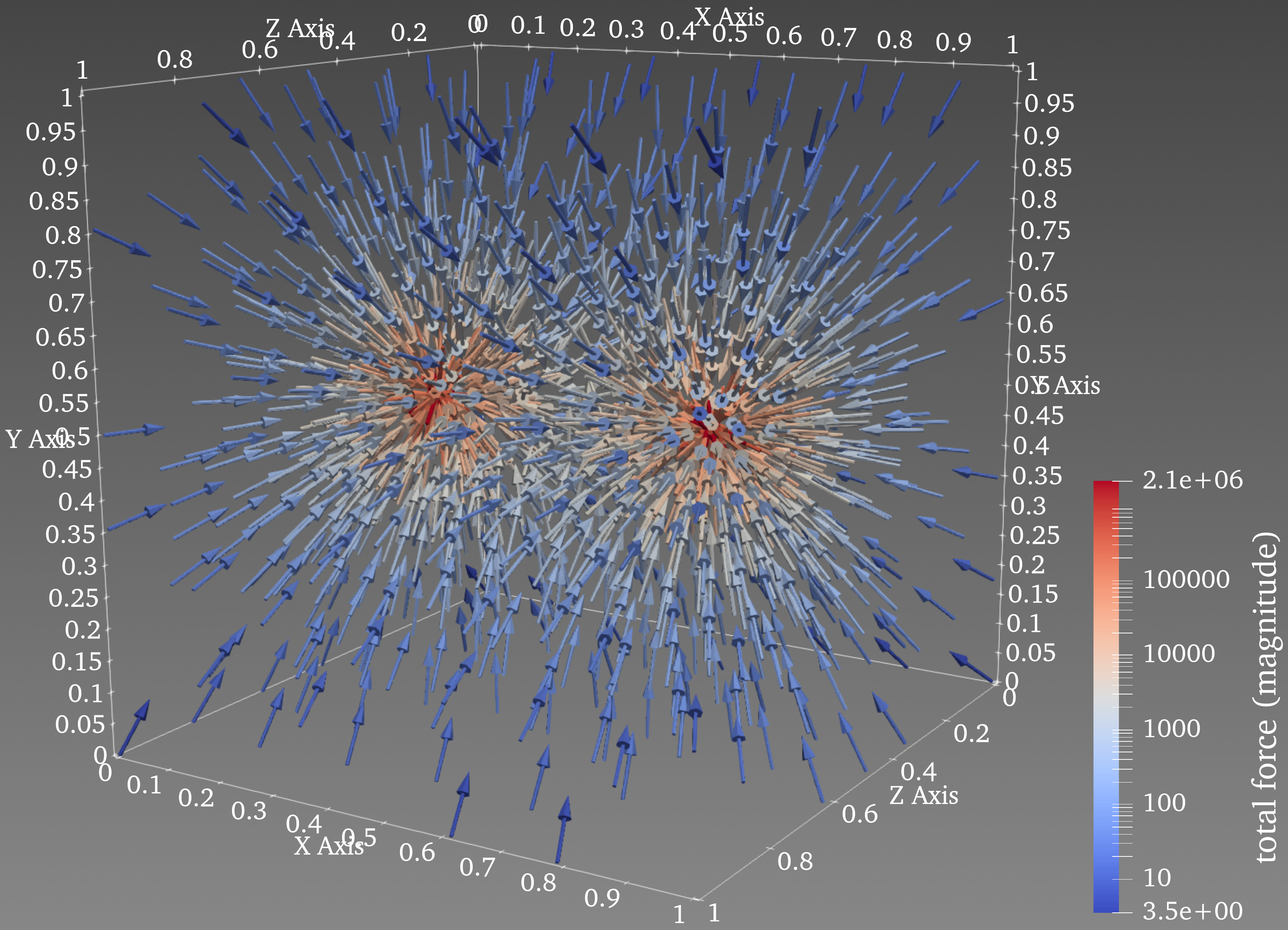} 
		\caption{LEFT: mechanical body force $\widehat{\mathbf{f}}\in C^\infty(\overline{\Omega};\mathbb{R}^3)$; RIGHT: total force $\mathbf{f}\in C^\infty(\overline{\Omega};\mathbb{R}^3)$ when the electric field is applied (in a $\log$-plot).}
		\label{fig:f}
	\end{figure}
	
	We apply Problem (\hyperlink{Qh}{Q$_h$}) (or Problem (\hyperlink{Ph}{P$_h$}), respectively)  to approximate the system~\eqref{eq:p-navier-stokes} with  $\bfS\colon\Omega\times \mathbb{R}^{3\times 3}\to\mathbb{R}^{3\times 3}_{\mathrm{sym}}$,  for every $(x,\mathbf{A})^\top\in \Omega\times\mathbb{R}^{3\times 3}$ defined by
	\begin{align*}
		\bfS(x,\bfA) \coloneqq \mu_0\,(\delta+\vert \bfA^{\textup{sym}}\vert)^{p(x)-2}\bfA^{\textup{sym}}\,,
	\end{align*}  
	where $\mu_0 = \frac{1}{2}$ and  $\delta\coloneqq 1.0\times10^{-5}$.
	
	Since  $\widehat{p}\hspace*{-0.1em}\in\hspace*{-0.1em} W^{1,\infty}(\mathbb{R}_{\ge 0})$,
	we have that $p\hspace*{-0.1em}\in\hspace*{-0.1em} W^{1,\infty}(\Omega)$. In particular, the power-law index~${p\hspace*{-0.1em}\in\hspace*{-0.1em} W^{1,\infty}(\Omega)}$ is approximated by $p_h\in \mathbb{P}^0(\mathcal{T}_h)$, $h\in (0,1]$, which is obtained by employing the same one-point quadrature rule from the previous section.
	We use the Taylor--Hood element (\textit{cf}.\ Remark \ref{FEM.V}(ii)) on a triangulation $\mathcal{T}_h$ with 4.947 vertices and 26.903 tetrahedra to compute $(\mathbf{v}_h,q_h)^\top\in\Vo_{h,0}\times \Qo_h$ solving Problem (\hyperlink{Qh}{Q$_h$}). To better compare the influence of the 
	applied electric field $\mathbf{E} \in C^{\infty}(\overline{\Omega};\mathbb{R}^3)$, we carry out two different numerical test cases:
	
	$\bullet$ \textit{Numerical test cases:}
	
	\textit{(Test Case 1)}.\ \hypertarget{TC1}{} In \hspace*{-0.1mm}this \hspace*{-0.1mm}case, \hspace*{-0.1mm}the \hspace*{-0.1mm}electric \hspace*{-0.1mm}field \hspace*{-0.1mm}is \hspace*{-0.1mm}not \hspace*{-0.1mm}applied, \hspace*{-0.1mm}\textit{i.e.}, \hspace*{-0.1mm}instead \hspace*{-0.1mm}of \hspace*{-0.1mm}\eqref{def:E},~\hspace*{-0.1mm}we~\hspace*{-0.1mm}set~\hspace*{-0.1mm}${\mathbf{E}\hspace*{-0.15em}\coloneqq\hspace*{-0.15em}\mathbf{0}}$~in~$\Omega$, so that the total force just becomes the mechanical body force, \textit{i.e.},  
	$\mathbf{f}=\widehat{\mathbf{f}}$ in $\Omega$ (\textit{cf}.\ Figure \ref{fig:f}(LEFT)). 
	Since, in this case, $p=4$ in $\Omega$, we approximate a standard shear-thickening non-Newtonian fluid.
	
	\textit{(Test Case 2)}.\ \hypertarget{TC2}{} In this case, the electric field defined by \eqref{def:E} is applied, so that the total forces is given via $\mathbf{f}=\widehat{\mathbf{f}}+\chi_E\textup{div}\,(\mathbf{E}\otimes \mathbf{E})$~in~$\Omega$~(\textit{cf}.~Figure~\ref{fig:f}(RIGHT)).
	Hence, in this case, we approximate~an electro-rheological fluid.
	
	$\bullet$ \textit{Observations:}
	
	\textit{(Test Case 1)}.\  In this case, by the construction of the mechanical body force $\mathbf{f} \in C^{\infty}(\overline{\Omega};\mathbb{R}^3)$, from the 
	 velocity vector field, we derive that a vortex flow is generated (\textit{cf}.\ Figure \ref{fig:v}(LEFT)). The magnitude of kinematic pressure is slightly larger close the the two holes but not oscillating much (\textit{cf}.\ Figure \ref{fig:p}(LEFT)).

	\textit{(Test Case 2)}.\  In this case,  the electric field  $\mathbf{E}\in C^\infty(\overline{\Omega};\mathbb{R}^3)$~is~applied (\textit{cf}.\ Figure~\ref{fig:E}(RIGHT)), so that the total force  $\mathbf{f}\in C^\infty(\overline{\Omega};\mathbb{R}^3)$ generates a high attraction close to the two holes (\textit{cf}.\ Figure~\ref{fig:f}(RIGHT)), where the electrodes are located, so
	that 
	the magnitude of both the velocity vector field and the kinematic pressure significantly increase close to these two holes. To be more precise,
	the magnitude of velocity vector field  is increased by a factor of about $43$ (\textit{cf}.\ Figure~\ref{fig:v}(RIGHT)) compared to (\hyperlink{TC2}{Test~Case~2}) (\textit{cf}.\ Figure~\ref{fig:v}(LEFT)), while the magnitude of the kinematic pressure is  increased by a factor~of~about~$1.2\times 10^5$ (\textit{cf}.\ Figure~\ref{fig:p}(RIGHT))  compared to (\hyperlink{TC2}{Test Case 2}) (\textit{cf}.\ Figure~\ref{fig:p}(LEFT)).

	\begin{figure}[H]\centering
		\includegraphics[width=7.75cm]{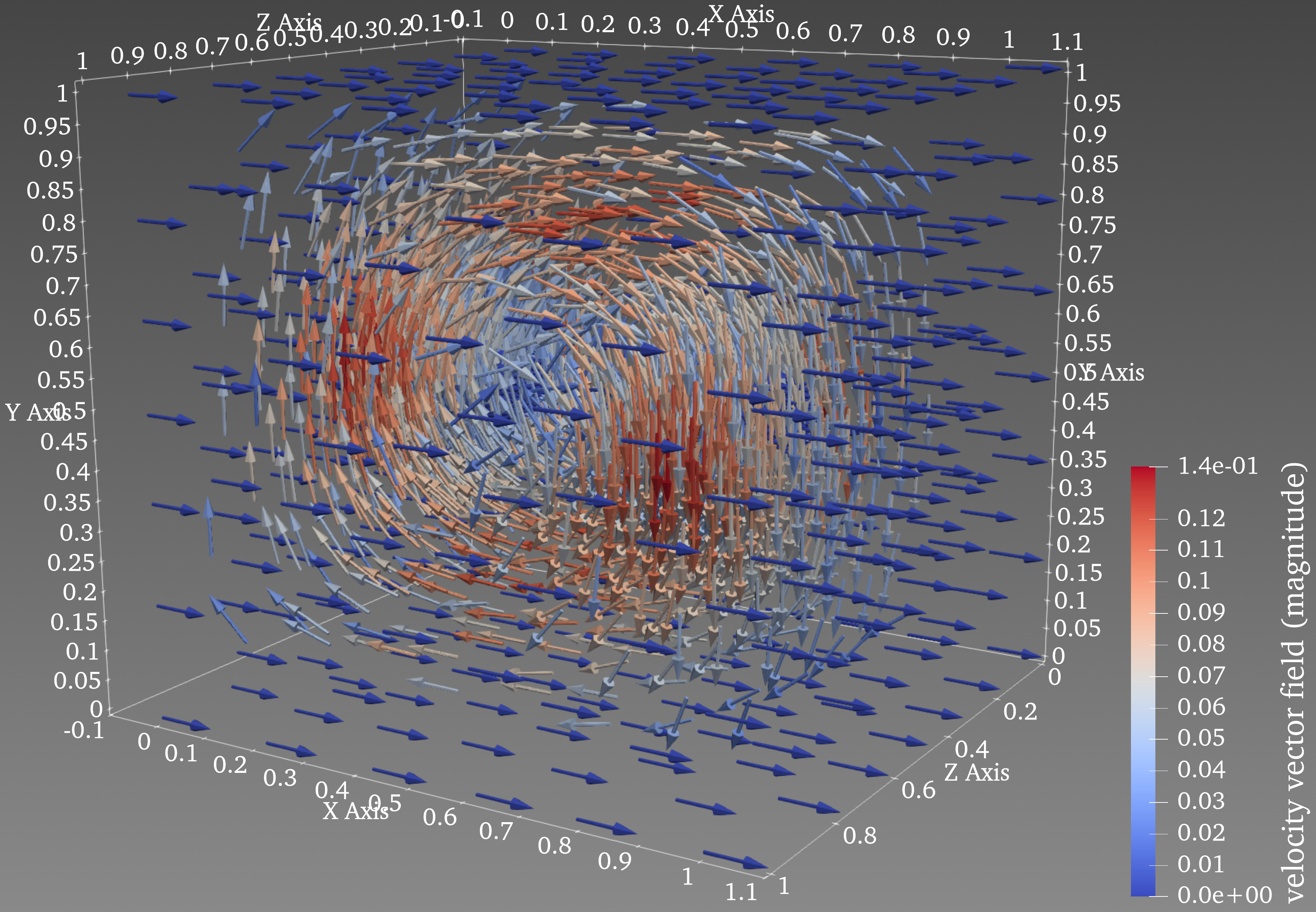}	\includegraphics[width=7.7cm]{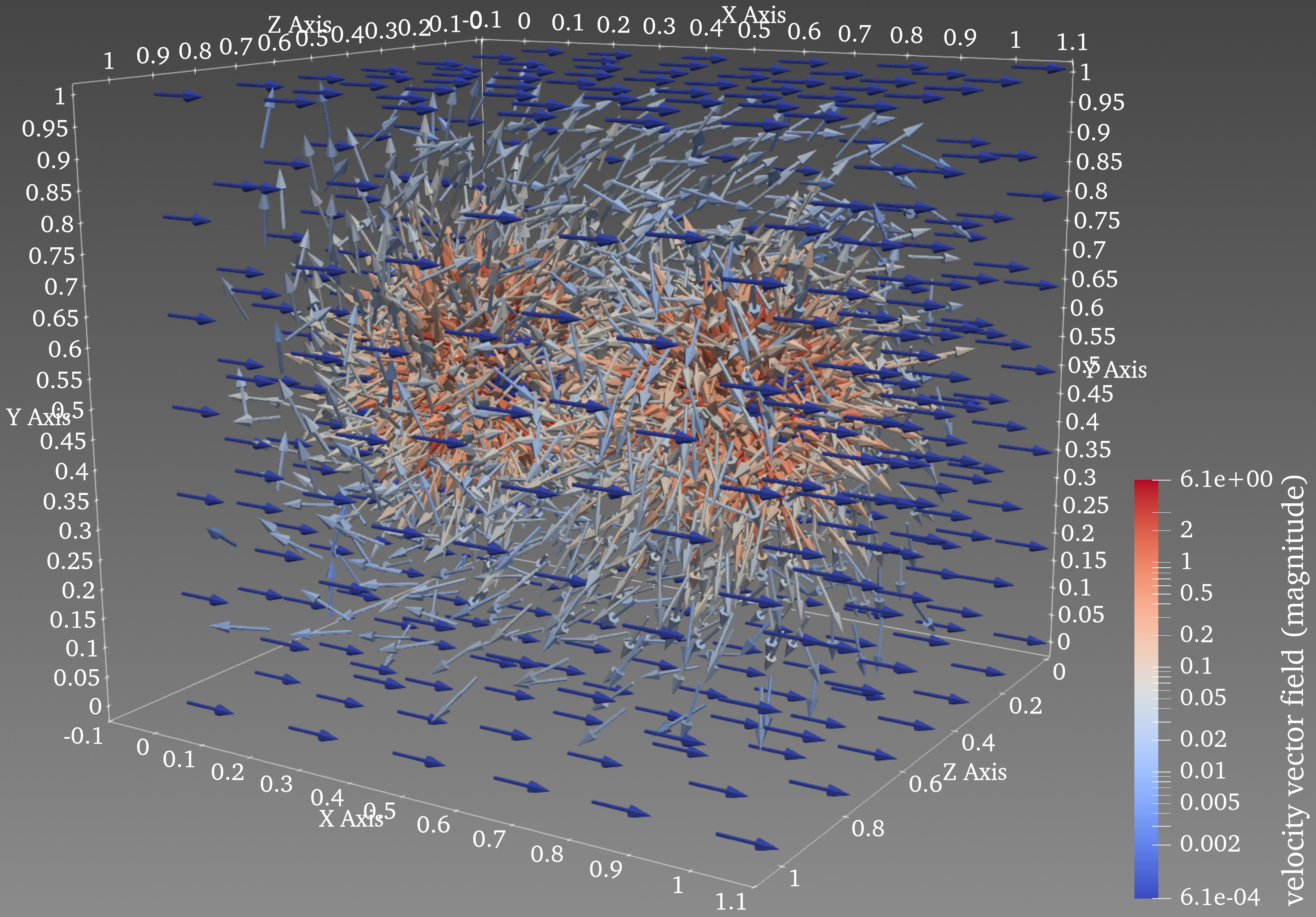}\vspace*{-1.75mm}
		\caption{LEFT: discrete velocity vector field $\mathbf{v}_h\in (\mathbb{P}^2(\mathcal{T}_h))^3$   in (\protect\hyperlink{TC1}{Test Case 1});
			RIGHT: discrete velocity vector field $\mathbf{v}_h\in (\mathbb{P}^2(\mathcal{T}_h))^3$  in (\protect\hyperlink{TC2}{Test Case 2}) (in a $\log$-plot).}
		\label{fig:v}
	\end{figure}
	
	\begin{figure}[H]\centering
		\includegraphics[width=7.7cm]{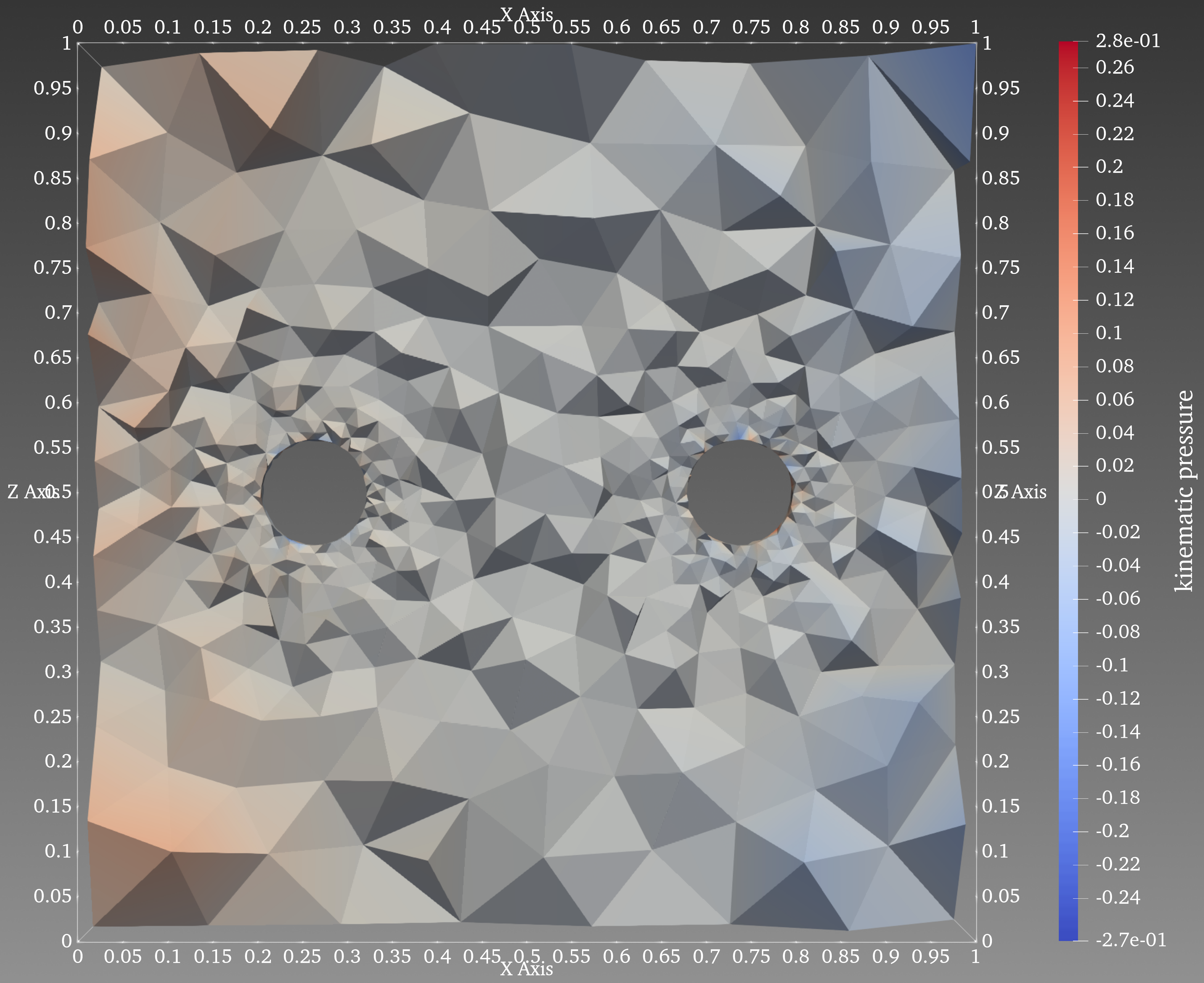}	\includegraphics[width=7.725cm]{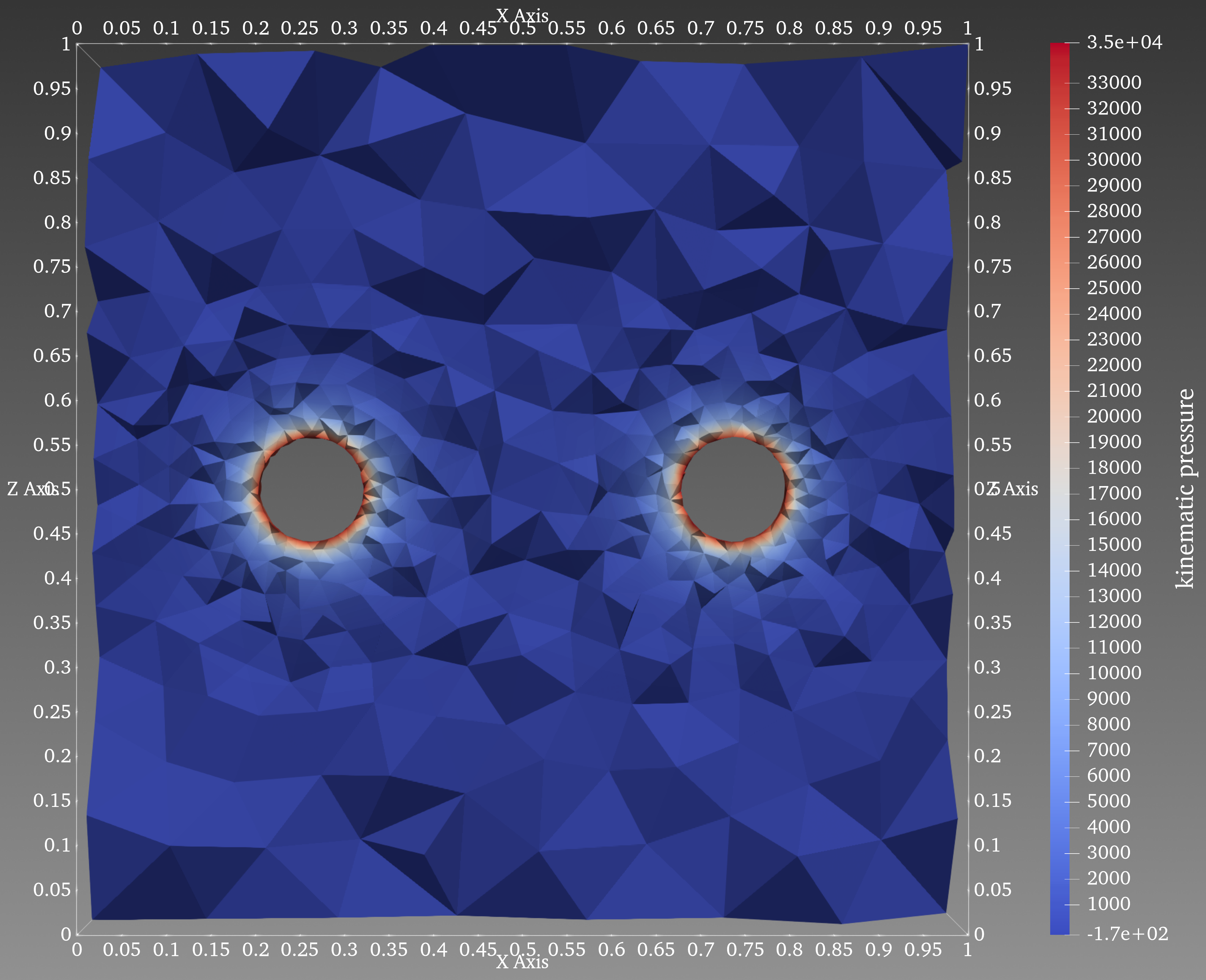}\vspace*{-1.75mm}
		\caption{LEFT: discrete pressure $q_h\in \mathbb{P}^1_c(\mathcal{T}_h)$ intersected with the ($\mathbb{R}\times \{\frac{1}{2}\}\times \mathbb{R}$)-plane in (\protect\hyperlink{TC1}{Test~Case~1});
			RIGHT: discrete pressure $q_h\in \mathbb{P}^1_c(\mathcal{T}_h)$ intersected with the ($\mathbb{R}\times \{\frac{1}{2}\}\times \mathbb{R}$)-plane in (\protect\hyperlink{TC2}{Test Case 2}).}
		\label{fig:p}
	\end{figure}

	\section{Outlook}
	 
	\textcolor{black}{\hspace*{5mm}We believe that the lower bound $p^-\ge \frac{3d}{d+2}$ can be further improved to the lower bound $p^-\ge \frac{2d}{d+1}$, when still working with the Temam modification \eqref{def:bh} (\textit{cf}.\ \cite{JK23_inhom} for the case of a constant~exponent), or even to the lower bound $p^-\ge \frac{2d}{d+2}$, when working with exactly divergence-free finite elements or with a different approximation of the convective term using a divergence correction operator. Another future research direction is the treatment of the unsteady case, \textit{i.e.}, of the so-called  unsteady $p(\cdot,\cdot)$-Navier--Stokes equations,
	 for which recently a weak convergence result for a fully-discrete approximation employing a simple backward Euler step in time and discretely divergence-free finite element in space has been established (\textit{cf}.\ \cite{BerKa23}).
	 Carrying out a thorough error analysis also for the unsteady $p(\cdot,\cdot)$-Navier--Stokes equations is the next natural step. Eventually, all results derived for the steady $p(\cdot)$-Navier--Stokes equation and the unsteady $p(\cdot,\cdot)$-Navier--Stokes equations
	 should be extended to the fully-coupled systems for smart fluids, \textit{e.g.}, electro-rheological fluids (\textit{cf}.\ \cite{RR1,rubo}), micro-polar electro-rheological~fluids (\textit{cf}.\ \cite{win-r,eringen-book}), magneto-rheological fluids (\textit{cf}.\ \cite{magneto}),  chemically reacting fluids (\textit{cf}.\ \cite{LKM78,HMPR10}),~and~\mbox{thermo-rheological}~fluids~(\textit{cf}.~\cite{Z97,AR06}).}
	
	\newpage
	\appendix
	
	\section{Modular-to-norm estimates}
	
	\hspace*{5mm}In this section, we give a proof for two essential modular-to-norm estimates, which can be used to transfer convergence rates measured in modulars into convergence rates measured in Luxembourg~norms. Note that the inverse is always easier, since, appealing to \cite[Lem.\  3.2.5]{dhhr}, for a function $f\in L^{p(\cdot)}(M)$, from $\|f\|_{p(\cdot),M}\leq\gamma\leq 1$, it follows that $\rho_{p(\cdot),M}(f)\leq \|f\|_{p(\cdot),M}\leq \gamma$.\enlargethispage{5mm}
	
	\begin{lemma}\label{lem:just_to_mighty}
		Let $M\subseteq \mathbb{R}^d$, $d\in \mathbb{N}$, be a (Lebesgue) measurable set, $p\in \mathcal{P}^\infty(M)$ with  $p^-> 1$,  $r\coloneqq \min\{2,p\}\in \mathcal{P}^\infty(M)$,  $\delta\ge 0$, and $a\in L^{p(\cdot)}(M;\mathbb{R}_{\ge 0})$. Then, 
		the~\mbox{following}~\mbox{statements}~apply:
		\begin{itemize}[noitemsep,topsep=2pt,labelwidth=\widthof{(ii)},leftmargin=!,font=\normalfont\itshape]
			\item[(i)] For every $c_0\ge 1$, $\gamma\leq 1$, and $f\in L^{p(\cdot)}(M)$ from $\rho_{\varphi_a,M}(f)\leq c_0\,\gamma$, 
			it follows~that 
			\begin{align*}
				\|f\|_{\varphi_a,M}\lesssim  c_0^{\frac{1}{r^-}}\, \gamma^{\frac{1}{r^+}}\,,
			\end{align*}
			where $\lesssim$ depends only on $p^+$.
			
			\item[(ii)] For every $c_0\ge 1$, $\gamma\leq 1$, and $g\in L^{p'(\cdot)}(M)$ from $\rho_{(\varphi_a)^*,M}(g)\leq c_0\, \gamma$, 
			it follows~that
			\begin{align*}
				\|g\|_{(\varphi_a)^*,M}\lesssim c_0^{\frac{1}{(r^+)'}} \,\gamma^{\frac{1}{(r^-)'}}\,,
			\end{align*}
			where $\lesssim$ depends only on $p^-$ and $p^+$.
		\end{itemize}
	\end{lemma}
	
	\begin{proof}
		\textit{ad (i)}. For a.e.\ $x\in M$, we distinguish the cases $p(x)\ge 2$ and $p(x)<2$.
		
		\textit{Case $p(x)\ge 2$.} Observing that, owing to \cite[Lem. 5.1, (5.11)]{dr-nafsa}, \eqref{eq:phi_shifted}, \eqref{eq:def_phi}, the $\Delta_2$-condition~of $\varphi_{a(x)}(x,\cdot)\colon\mathbb{R}^{\ge 0}\to \mathbb{R}^{\ge 0}$, and \cite[Lem.~5.3]{dr-nafsa}, for every $\lambda \leq 1$, $c\ge 1$,  and $t\ge 0$, it holds that
		\begin{align*}
			\varphi_{a(x)}\Big(x,\frac{t}{c\lambda}\Big)
			&\leq
			(\varphi_{a(x)})'\Big(x,\frac{t}{c\lambda}\Big)\frac{t}{c\lambda}
			%        \\
			%        &
			%        =
			%        \Big(\delta+a(x)+\frac{t}{c\lambda}\Big)^{p(x)-2}\Big(\frac{t}{c\lambda}\Big)^2
			\\& \leq
			\Big(\frac{\delta}{c\lambda}+\frac{a(x)}{c\lambda}+\frac{t}{c\lambda}\Big)^{p(x)-2}\Big(\frac{t}{c\lambda}\Big)^2
			\\&\leq 
			\frac{1}{\lambda^{p(x)} c^2}(\delta+a(x)+t)^{p(x)-2}t^2
			\\
			&
			=
			\frac{1}{\lambda^{p^+}c^2}(\varphi_{a(x)})'(x,t)t
			\\& \leq 
			\frac{1}{\lambda^{p^+} c^2}\varphi_{a(x)}(x,2t)
			%         \\&\leq 
			%        \frac{\Delta_2(\varphi_{a(x)}(x,\cdot))}{\lambda^{p^+} c^2}\varphi_{a(x)}(x,t)
			\leq 
			\frac{2^{p^+}}{\lambda^{p^+} c^2}\varphi_{a(x)}(x,t)\,.
		\end{align*}
		
		\textit{Case $p(x)< 2$.}  Observing that, owing to \cite[Lem. 5.1, (5.11)]{dr-nafsa}, \eqref{eq:phi_shifted}, \eqref{eq:def_phi}, the $\Delta_2$-condition~of $\varphi_{a(x)}(x,\cdot)\colon\mathbb{R}^{\ge 0}\to \mathbb{R}^{\ge 0}$ and \cite[Lem.~5.3]{dr-nafsa}, for every $\lambda \leq 1$, $c\ge 1$,  and $t\ge 0$, it holds that
		\begin{align*}
			\varphi_{a(x)}\Big(x,\frac{t}{c\lambda}\Big)
			&\leq
			\Big(\delta+a(x)+\frac{t}{c\lambda}\Big)^{p(x)-2}\Big(\frac{t}{c\lambda}\Big)^2
			\\
			&\leq 
			\frac{1}{\lambda^2}
			\Big(\frac{\delta}{c}+\frac{a(x)}{c}+\frac{t}{c}\Big)^{p(x)-2}\Big(\frac{t}{c}\Big)^2
			\\
			&\leq 
			\frac{1}{\lambda^2c^{p(x)}}
			(\delta+a(x)+t)^{p(x)-2}t^2
			%\\
			%&
			%	\le   \frac{1}{\lambda^2c^{(p^+)'}}(\varphi_{a(x)})'(x,t)t
			%\\
			%&
			\le   \frac{2^{p^+}}{\lambda^2c^{p^-}}\varphi_{a(x)}(x,t)\,.
		\end{align*}
		Therefore, choosing $\lambda = \gamma^{\smash{\frac{1}{r^+}}}\leq 1$ and $c=\smash{c_0^{\frac{1}{r^-}}}2^{p^+}\ge 1$, we find that
		\begin{align*}
			\rho_{\varphi_a,M}\bigg(
			\frac{f}{(c_0^{1/r^-}2^{p^+})^{1/(r^+)'} \gamma^{1/\max\{2,p^+\}}}
			\bigg)\leq \frac{1}{c_0 \gamma}\rho_{\varphi_a,M}(f)\leq 1\,,
		\end{align*}
		so that, from the definition of the Luxembourg norm, we conclude the assertion.
		
		\textit{ad (ii)}. Observing that, for every $\lambda \leq 1$, $c\ge 1$,  $t\ge 0$,  and a.e.\ $x\in M$,~it~holds that
		\begin{align*}
			(\varphi_{a(x)})^*\Big(x,\frac{t}{c\lambda}\Big)
			&\sim 
			\Big(\delta^{p(x)-1}+a(x)^{p(x)-1}+\frac{t}{c\lambda}\Big)^{p'(x)-2}\Big(\frac{t}{c\lambda}\Big)^2\,,
		\end{align*}
		where $\sim$ depends only on $p^-$ and $p^+$, the claimed inequality follows analogously to (i).
	\end{proof}

	\section{Sobolev-to-natural-distance estimates}
	
	\hspace*{5mm}In this section, we derive estimates that bound Sobolev norm distances by the natural distance.
	Before stating a general estimate, we derive two estimates distinguishing the cases $p^+\leq 2$ and $p^-\ge 2$. 
	
	\begin{lemma}\label{lem:sobolev2F_first}
		Let $M\subseteq \mathbb{R}^d$, $d\in \mathbb{N}$, be a (Lebesgue) measurable set and  $p\in \mathcal{P}^{\infty}(M)$ with $p^+\le 2$. Then,
		for every $\bfA,\bfB\in L^{p(\cdot)}(M;\mathbb{R}^{d\times d})$, it holds that
		\begin{align*}
			\|\bfA-\bfB\|_{p(\cdot),M}^2&\lesssim \|\bfF(\cdot,\bfA)-\bfF(\cdot,\bfB)\|_{2,M}\|(\vert \bfA\vert+\vert \bfB\vert)^{\smash{\frac{2-p(\cdot)}{2}}}\|_{\frac{2p(\cdot)}{p(\cdot)-2},M}^2\\&\lesssim 
			\|\bfF(\cdot,\bfA)-\bfF(\cdot,\bfB)\|_{2,M}^2\,\big(1+\rho_{p(\cdot),M}(\vert \bfA\vert+\vert \bfB\vert)\big)^{\smash{\frac{2}{p^-}}}\,,
		\end{align*}
		where $\lesssim$ depends only on $p^+$ and $p^-$.
	\end{lemma}
	
	\begin{proof}
		Using the generalized Hölder inequality \eqref{eq:gen_hoelder} and \eqref{eq:hammera}, we find that
		\begin{align}\label{eq:sobolev2F_first.1}
			\begin{aligned}
				\|\bfA-\bfB\|_{p(\cdot),M}&=\|\vert \bfA-\bfB\vert(\vert \bfA\vert+\vert \bfB\vert)^{\smash{\frac{p(\cdot)-2}{2}}}(\vert \bfA\vert+\vert \bfB\vert)^{\smash{\frac{2-p(\cdot)}{2}}}\|_{p(\cdot),M}\\&\leq
				2\,\|\vert \bfA-\bfB\vert(\vert \bfA\vert+\vert \bfB\vert)^{\smash{\frac{p(\cdot)-2}{2}}}\|_{2,M}\\&\quad\times\|(\vert \bfA\vert+\vert \bfB\vert)^{\smash{\frac{2-p(\cdot)}{2}}}\|_{\smash{\frac{2p(\cdot)}{p(\cdot)-2}},M}
				\\&\leq 
				c\,\|\bfF(\cdot,\bfA)-\bfF(\cdot,\bfB)\|_{2,M}\|(\vert \bfA\vert+\vert \bfB\vert)^{\smash{\frac{2-p(\cdot)}{2}}}\|_{\smash{\frac{2p(\cdot)}{p(\cdot)-2}},M}\,.
			\end{aligned}
		\end{align}
		Moreover,  appealing to \cite[Lem.\  3.2.5]{dhhr}, we have that
		\begin{align}\label{eq:sobolev2F_first.2}
			\begin{aligned}
				\|(\vert \bfA\vert+\vert \bfB\vert)^{\smash{\frac{2-p(\cdot)}{2}}}\|_{\frac{2p(\cdot)}{p(\cdot)-2},M}^{p^-}-1&\leq \rho_{\smash{\frac{2p(\cdot)}{2-p(\cdot)}},M}((\vert \bfA\vert+\vert \bfB\vert)^{\smash{\frac{2-p(\cdot)}{2}}})
				\\&= \rho_{p(\cdot),M}(\vert \bfA\vert+\vert \bfB\vert)\,.
			\end{aligned}
		\end{align}
		Using \eqref{eq:sobolev2F_first.2} in \eqref{eq:sobolev2F_first.1}, we conclude the claimed estimate.
	\end{proof}
	
	\begin{lemma}\label{lem:sobolev2F_second}
	Let $M\subseteq \mathbb{R}^d$, $d\in \mathbb{N}$, be a (Lebesgue) measurable set, $\delta>0$, and  $p\in \mathcal{P}^{\infty}(M)$~with~$p^-\ge  2$. Then,
		for every $\bfA,\bfB\in L^{p(\cdot)}(M;\mathbb{R}^{d\times d})$, it holds that
		\begin{align*}
			\|\bfA-\bfB\|_{2,M}^2\lesssim (\min\{1,\delta\})^{2-p^+}\,\|\bfF(\cdot,\bfA)-\bfF(\cdot,\bfB)\|_{2,M}^2\,,
		\end{align*}
		where $\lesssim$ depends only on $p^+$ and $p^-$.
	\end{lemma}
	
	\begin{proof}
		Due to $\varphi_{\vert \bfD\bfv(x)\vert }(x,t)\geq c\, \delta^{p(x)-2}t^2\ge  c\, (\min\{1,\delta\})^{p(x)-2}t^2\ge  c\, (\min\{1,\delta\})^{p^+-2}t^2$ for a.e. $x\in M$ and $t\ge 0$ (\textit{cf}.\ \eqref{rem:phi_a.1}), where $c>0$  depends only on $p^+,p^->1$, and \eqref{eq:hammera}, we have that 
		\begin{align*}
			\|\bfA-\bfB\|_{2,M}^2&\leq c\, (\min\{1,\delta\})^{2-p^+}\rho_{\varphi_{\vert \bfA\vert },M}(\bfA-\bfB)
			\\&\leq c\, (\min\{1,\delta\})^{2-p^+}\,\|\bfF(\cdot,\bfA)-\bfF(\cdot,\bfB)\|_{2,M}^2\,,
		\end{align*}
		which is the claimed estimate.
	\end{proof}
	
	Combining Lemma \ref{lem:sobolev2F_first} and Lemma \ref{lem:sobolev2F_second}, we arrive at the following general estimate relating  Sobolev norm distances to the natural distance.\enlargethispage{2mm}
	
	\begin{lemma}\label{lem:sobolev2F}
		Let $M\subseteq \mathbb{R}^d$, $d\in \mathbb{N}$, be (Lebesgue) measurable set, $\delta>0$, and  $p\in \mathcal{P}^{\infty}(G)$ with $p^-> 1$.  Moreover, set $r\coloneqq \min\{2,p\}\in \mathcal{P}^{\infty}(M)$. Then,  for every $\bfA,\bfB\in L^{p(\cdot)}(M;\mathbb{R}^{d\times d})$, it holds that
		\begin{align*}
			\|\bfA-\bfB\|_{r(\cdot),M}^2&\lesssim
			\|\bfF(\cdot,\bfA)-\bfF(\cdot,\bfB)\|_{2,M}^2\\[-1mm]&\quad\times\big((1+\rho_{p(\cdot),M}(\vert \bfA\vert+\vert \bfB\vert))^{\smash{\frac{2}{p^-}}}+(\min\{1,\delta\})^{2-p^+}\big)\,,
		\end{align*}
		where $\lesssim$ depends only on $p^+$ and $p^-$.
	\end{lemma}
	
	\begin{proof}
		Using the decomposition
		\begin{align*}
			\|\bfA-\bfB\|_{r(\cdot),M}
			&\leq  \|\bfA-\bfB\|_{r(\cdot),\{p\leq 2\}}+\|\bfA-\bfB\|_{r(\cdot),\{p> 2\}}
			\\&= \|\bfA-\bfB\|_{p(\cdot),\{p\leq 2\}}+\|\bfA-\bfB\|_{2,\{p> 2\}}\,,
		\end{align*}
		in conjunction with Lemma \ref{lem:sobolev2F_first} and Lemma \ref{lem:sobolev2F_second}, we conclude the claimed estimate.
	\end{proof}

	\subsection{Discrete-to-continuous-and-vice-versa inequalities}
	\hspace*{5mm}The following results bound the error resulting from switching from $\bfS_h\colon\Omega\times\mathbb{R}^{d\times d}\to \mathbb{R}^{d\times d}_{\textup{sym}}$, $h\in (0,1]$, to $\bfS\colon\Omega\times\mathbb{R}^{d\times d}\to \mathbb{R}^{d\times d}_{\textup{sym}}$ or 
	from switching from
	$(\varphi_h)^*\colon\Omega\times\mathbb{R}_{\ge 0}\to \mathbb{R}_{\ge 0}$, $h\in (0,1]$, to $\varphi^*\colon\Omega\times\mathbb{R}_{\ge 0}\to \mathbb{R}_{\ge 0}$ and vice versa, respectively.\enlargethispage{3mm}
	%from~switching~from
%	$\bfF_h\colon\Omega\times\mathbb{R}^{d\times d}\to \mathbb{R}^{d\times d}_{\textup{sym}}$, $h>0$, to $\bfF\colon\Omega\times\mathbb{R}^{d\times d}\to \mathbb{R}^{d\times d}_{\textup{sym}}$ and vice versa, respectively.
	
	\begin{proposition}\label{lem:A-Ah}
		Let $p\in C^{0,\alpha}(\overline{\Omega})$ with $p^->1$ and $\alpha\in (0,1]$, and let $\delta\ge 0$.   Then, there exists a constant ${s>1}$, which can chosen to be close to $0$ if $h_T>0$ is close to $0$, such that
		for every $T\in \mathcal{T}_h$,  $g\in L^{p'(\cdot)s}(T)$, $\bfA\in L^{p(\cdot)s}(T;\mathbb{R}^{d\times d})$, and $\lambda\in [0,1]$, it holds that
		\begin{align}
		\| \bfF_h(\cdot,\bfA)-\bfF(\cdot,\bfA)\|_{2,T}^2&\lesssim  h_T^{2\alpha}\,\|1+\vert \bfA\vert^{p(\cdot)s}\|_{1,T}
			\,,\label{eq:Fh-F}\\
		%	\| \bfF_h^*(\cdot,\bfB)-\bfF^*(\cdot,\bfB)\|_{2,T}^2&\lesssim  h_T^{2\alpha}\,\|1+\vert \bfB\vert^{p'(\cdot)s}\|_{1,T}
		%	\,,\label{eq:F*h-F*}\\
			\| \bfF_h^*(\cdot,\bfS_h(\cdot,\bfA))-\bfF^*_h(\cdot,\bfS(\cdot,\bfA))\|_{2,T}^2&\lesssim h_T^{2\alpha}\,\|1+\vert\bfA\vert^{p(\cdot)s}\|_{1,T}
			\,,\label{eq:Ah-A}\\
		%	\rho_{(\varphi_h)_{\vert \bfA\vert},T}(\lambda\,f) &\lesssim \rho_{\varphi_{\vert \bfA \vert},T}(\lambda\,f)\label{eq:phih-phi.0}\\&\quad +\smash{\lambda^{\smash{{2\wedge p^-}}}}h_T^{\alpha}\,\|1+\vert \bfA\vert^{p(\cdot)s}+\vert f\vert^{p(\cdot)s}\|_{1,T}\,,\notag\\
			\rho_{((\varphi_h)_{\vert \bfA\vert})^*,T}(\lambda\,g) &\lesssim \rho_{(\varphi_{\vert \bfA \vert})^*,T}(\lambda\,g)\label{eq:phih-phi}\\&\quad+\smash{\lambda^{\smash{\min\{2, (p^+)'\}}}}h_T^{\alpha}\,\|1+\vert \bfA\vert^{p(\cdot)s}+\vert g\vert^{p'(\cdot)s}\|_{1,T}\,,\notag
		\end{align} 
		where the hidden constants also depend on $s>1$ and the chunkiness $\omega_0>0$.
	\end{proposition}
	
	\begin{proof}
		See \cite[Prop.\  2.10]{BK23_pxDirichlet}.
	\end{proof}
	
	\begin{lemma}\label{lem:poincare_F}
		Let $p\in C^{0,\alpha}(\overline{\Omega})$ with $p^->1$ and $\alpha\in (0,1]$, and let $\bfA\in L^{p(\cdot)}(\Omega;\mathbb{R}^{d\times d})$ be such that
		$\bfF(\cdot,\bfA)\in N^{\beta,2}(\Omega;\mathbb{R}^{d\times d})$ with $\beta\in (0,1]$. Then,  there exists a constant $s>1$ which can chosen to be close to $1$ if $h>0$ is close to $0$,  such that for every $T\in \mathcal{T}_h$, it holds that
		\begin{align}
			\|\bfF(\cdot,\bfA)-\bfF(\cdot,\langle\bfA\rangle_{T})\|_{2,T}^2&\lesssim 
			h_T^{2\beta}\,[ \bfF(\cdot,\bfA)]_{N^{\beta,2}(\textcolor{black}{T})}^2
			+ h_T^{2\alpha}\,\|1+\vert \bfA\vert^{p(\cdot)s}\|_{1,T}\,,\label{lem:poincare_F.1}\\
			\|\bfF(\cdot,\bfA)-\bfF(\cdot,\langle\bfA\rangle_{\omega_T})\|_{2,\omega_T}^2&\lesssim
			h_T^{2\beta}\,[ \bfF(\cdot,\bfA)]_{N^{\beta,2}(\textcolor{black}{\omega_T})}^2
			+ h_T^{2\alpha}\,\|1+\vert \bfA\vert^{p(\cdot)s}\|_{1,\omega_T}\,,\label{lem:poincare_F.2}
		\end{align}
		where %$\omega_T^{2\times}\coloneqq \bigcup_{T'\in \omega_T}{\omega_{T'}}$ and 
		$\lesssim$ depends on $p^-$, $p^+$, $[p]_{\alpha,\Omega}$, $s$, and $\omega_0$.
		In particular, it holds that
		\begin{align}
			\|\bfF(\cdot,\bfA)-\bfF(\cdot,\Pi_h^0\bfA)\|_{2,\Omega}^2&\lesssim 
			h^{2\beta}\,[ \bfF(\cdot,\bfA)]_{N^{\beta,2}(\Omega)}^2
			+ h^{2\alpha}\,\|1+\vert \bfA\vert^{p(\cdot)s}\|_{1,\Omega}\,,\label{lem:poincare_F.3}\\
			\sum_{T\in \mathcal{T}_h}{	\|\bfF(\cdot,\bfA)-\bfF(\cdot,\langle\bfA\rangle_{\omega_T})\|_{2,\omega_T}^2}&\lesssim 
			h^{2\beta}\,[ \bfF(\cdot,\bfA)]_{N^{\beta,2}(\Omega)}^2
			+ h^{2\alpha}\,\|1+\vert \bfA\vert^{p(\cdot)s}\|_{1,\Omega}\,.\label{lem:poincare_F.4}
		\end{align} 
	\end{lemma}
	
	\begin{proof} 
		
		\textit{ad \eqref{lem:poincare_F.1}.}
		Using \eqref{eq:hammera}, $\vert \bfA(x)-\langle\bfA\rangle_T\vert =\vert \langle\bfA(x)-\bfA\rangle_T\vert\leq \langle\vert  \bfA(x)-\bfA\vert\rangle_T$ for a.e.\ $x\in T$, Jensen's inequality for a.e.\ fixed $x\in T$, \eqref{eq:nikolski_semi-norm}, that there exists a constant $s>1$ that can chosen to be close to $1$ if $h_T>0$ is close to $0$ such that for every $x,y\in T$, it holds that
		\begin{align}
			\vert \bfF(x,\bfA(y))-\bfF(y,\bfA(y))\vert^2\lesssim \vert p(x)-p(y)\vert^2\,(1 +\vert \bfA\vert^{p(y)s})\,,\label{eq:Fxh-Fx}
		\end{align}
		where the hidden constants also depend on $s>1$ and the chunkiness $\omega_0>0$,   
		we find that
		\begin{align}\label{lem:poincare_F.5}
			\begin{aligned}
				\|\bfF(\cdot,\bfA)-\bfF(\cdot,\langle\bfA\rangle_T)\|_{2,T}^2&\lesssim c\int_T{\varphi_{\vert \bfA(x)\vert }(x,\langle\vert  \bfA(x)-\bfA\vert\rangle_T)\,\mathrm{d}x}
				\\&\lesssim \int_T{\fint_T{\varphi_{\vert \bfA(x)\vert }(x,\vert  \bfA(x)-\bfA(y)\vert)\,\mathrm{d}y}\,\mathrm{d}x}
				%	\\&\lesssim \int_T{\fint_T{\vert  \bfF(x,\bfA(x))-\bfF(x,\bfA(y))\vert^2\,\mathrm{d}y}\,\mathrm{d}x}
				\\&\lesssim \int_T{\fint_T{\vert  \bfF(x,\bfA(x))-\bfF(y,\bfA(y))\vert^2\,\mathrm{d}y}\,\mathrm{d}x}\\&\quad+\int_T{\fint_T{\vert  \bfF(y,\bfA(y))-\bfF(x,\bfA(y))\vert^2\,\mathrm{d}y}\,\mathrm{d}x}
				\\&\lesssim \frac{1}{\vert T\vert }\int_{\vert\cdot \vert \leq h_T}{\int_{T\cap (T-y)}{\vert  \bfF(x+y,\bfA(x+y)-\bfF(x,\bfA(x))\vert^2\,\mathrm{d}x}\,\mathrm{d}y}
				\\&\quad + h_T^{2\alpha}\,\|1+\vert \bfA\vert^{p(\cdot)s}\|_{1,T}
				\\&\lesssim h_T^{2\beta}\, [\bfF(\cdot,\bfA)]_{N^{\beta,2}(\textcolor{black}{T})}^2
				+ h_T^{2\alpha}\,\|1+\vert \bfA\vert^{p(\cdot)s}\|_{1,T}\,,
			\end{aligned}
		\end{align}
		which is the claimed estimate \eqref{lem:poincare_F.1}.
		
		\textit{ad \eqref{lem:poincare_F.2}.} Using \eqref{eq:hammera}, $\vert \bfA(x)-\langle\bfA\rangle_{\omega_T}\vert =\vert \langle\bfA(x)-\bfA\rangle_{\omega_T}\vert\leq \langle\vert  \bfA(x)-\bfA\vert\rangle_{\omega_T}$ for a.e.\ $x\in \omega_T$, Jensen's inequality for a.e.\ fixed $x\in T$,  \eqref{eq:Fxh-Fx},  \eqref{eq:nikolski_semi-norm}, and $\textup{diam}(\omega_T)\sim h_T$, with a constant depending only~on~$\omega_0$,
		we find that
		\begin{align*}
			\|\bfF(\cdot,\bfA)-\bfF(\cdot,\langle\bfA\rangle_{\omega_T})\|_{2,\omega_T}^2&
			= \int_{\omega_T}{\fint_{\omega_T}{\vert  \bfF(x,\bfA(x))-\bfF(y,\bfA(y))\vert^2\,\mathrm{d}y}\,\mathrm{d}x}\\&\quad+\int_{\omega_T}{\fint_{\omega_T}{\vert  \bfF(y,\bfA(y))-\bfF(x,\bfA(y))\vert^2\,\mathrm{d}y}\,\mathrm{d}x}
			\\&\lesssim \frac{1}{\vert \omega_T\vert }\int_{\vert\cdot \vert \leq \textup{diam}(\omega_T)}{\int_{\omega_T\cap (\omega_T-y)}{\vert  \bfF(x+y,\bfA(x+y)-\bfF(x,\bfA(x))\vert^2\,\mathrm{d}x}\,\mathrm{d}y}
			\\&\quad + h_T^{2\alpha}\,\|1+\vert \bfA\vert^{p(\cdot)s}\|_{1,\omega_T}
			\\&\lesssim \textup{diam}(\omega_T)^{2\beta} \,[\bfF(\cdot,\bfA)]_{N^{\beta,2}(\textcolor{black}{\omega_T})}^2
			+ h_T^{2\alpha}\,\|1+\vert \bfA\vert^{p(\cdot)s}\|_{1,\omega_T}
			\\&\lesssim h_T^{2\beta} \,[\bfF(\cdot,\bfA)]_{N^{\beta,2}(\textcolor{black}{\omega_T})}^2
			+ h_T^{2\alpha}\,\|1+\vert \bfA\vert^{p(\cdot)s}\|_{1,\omega_T}\,,
		\end{align*}
		which is the claimed estimate \eqref{lem:poincare_F.2}.\enlargethispage{2mm}
		
		\textit{ad \eqref{lem:poincare_F.3}.} Using \eqref{lem:poincare_F.5}, via summation with respect to $T\in \mathcal{T}_h$, and using that $h_T\sim h$~for~all~$T\in \mathcal{T}_h$, with a constant depending only on $\omega_0$,
		we find that
		\begin{align*}
			\|\bfF(\cdot,\bfA)-\bfF(\cdot,\Pi_h^0\bfA)\|_{2,\Omega}^2&\leq \sum_{T\in \mathcal{T}_h}{\|\bfF(\cdot,\bfA)-\bfF(\cdot,\langle\bfA\rangle_T)\|_{2,T}^2}
			\\&\lesssim \sum_{T\in \mathcal{T}_h}{\frac{1}{\vert T\vert }\int_{\vert\cdot \vert \leq h_T}{\int_{T\cap (T-y)}{\vert  \bfF(x+y,\bfA(x+y)-\bfF(x,\bfA(x))\vert^2\,\mathrm{d}x}\,\mathrm{d}y}}
			\\&\quad + \sum_{T\in \mathcal{T}_h}{h_T^{2\alpha}\,\|1+\vert \bfA\vert^{p(\cdot)s}\|_{1,T}}
			\\&\lesssim \frac{1}{h^d }\int_{\vert\cdot \vert \leq h}{\int_{\Omega\cap (\Omega-y)}{\vert  \bfF(x+y,\bfA(x+y)-\bfF(x,\bfA(x))\vert^2\,\mathrm{d}x}v\,\mathrm{d}y}
			\\&\quad + h^{2\alpha}\,\|1+\vert \bfA\vert^{p(\cdot)s}\|_{1,\Omega}
			\\&\lesssim h^{2\beta}\, [\bfF(\cdot,\bfA)]_{N^{\beta,2}(\Omega)}^2
			+  h^{2\alpha}\,\|1+\vert \bfA\vert^{p(\cdot)s}\|_{1,\Omega}\,,
		\end{align*}
		which is the claimed estimate \eqref{lem:poincare_F.3}.
		
		\textit{ad \eqref{lem:poincare_F.4}.} We proceed similar to the proof of \eqref{lem:poincare_F.3} but now using \eqref{lem:poincare_F.2}.
	\end{proof}	
	
	\section*{Acknowledgments}
	
	\hspace*{5mm}Alex Kaltenbach acknowledges  the hospitality of the Department of Applied Mathematics of the University of Pisa.
	
	{\setlength{\bibsep}{0pt plus 0.0ex}\small
		
%		\bibliographystyle{aomplain}
%		\bibliography{fluid,rose}

\def\cprime{$'$} \def\cprime{$'$} \def\cprime{$'$}
\providecommand{\bysame}{\leavevmode\hbox to3em{\hrulefill}\thinspace}
\providecommand{\noopsort}[1]{}
\providecommand{\mr}[1]{\href{http://www.ams.org/mathscinet-getitem?mr=#1}{MR~#1}}
\providecommand{\zbl}[1]{\href{http://www.zentralblatt-math.org/zmath/en/search/?q=an:#1}{Zbl~#1}}
\providecommand{\jfm}[1]{\href{http://www.emis.de/cgi-bin/JFM-item?#1}{JFM~#1}}
\providecommand{\arxiv}[1]{\href{http://www.arxiv.org/abs/#1}{arXiv~#1}}
\providecommand{\doi}[1]{\url{https://doi.org/#1}}
\providecommand{\MR}{\relax\ifhmode\unskip\space\fi MR }
\providecommand{\MRhref}[2]{%
	\href{http://www.ams.org/mathscinet-getitem?mr=#1}{#2}
}
\providecommand{\href}[2]{#2}

	}
	
\end{document}